\documentclass[lettersize,journal]{IEEEtran}
\usepackage{amsmath,amsfonts,amsthm,mathtools}
\usepackage{algorithmic}
\usepackage{algorithm}
\usepackage{array}
\usepackage[caption=false,font=normalsize,labelfont=sf,textfont=sf]{subfig}
\usepackage{textcomp}
\usepackage{stfloats}
\usepackage{url}
\usepackage[table]{xcolor}
\usepackage{verbatim}
\usepackage{graphicx}
\usepackage{cite}
\usepackage{hyperref}
\hypersetup{hidelinks=true}
\def\BibTeX{{\rm B\kern-.05em{\sc i\kern-.025em b}\kern-.08em
    T\kern-.1667em\lower.7ex\hbox{E}\kern-.125emX}}
\usepackage{balance}

\newif\ifcameraready
\camerareadytrue   
\newif\iffullversion
\newif\ifshowchanges
\newif\ifsupplement

\fullversiontrue    \showchangesfalse   \supplementfalse    

\ifshowchanges
  \usepackage[switch,pagewise]{lineno}
\fi

\ifshowchanges
  \newcommand{\rev}[1]{\textcolor{blue}{#1}}
\else
  \newcommand{\rev}[1]{#1}
\fi

\ifsupplement
\fi

\usepackage{etoolbox}
\makeatletter
\patchcmd{\@makecaption}
  {\scshape}
  {}
  {}
  {}
\makeatother

\graphicspath{{./figures/}}
\DeclareGraphicsExtensions{.pdf,.eps}

\newtheorem{theorem}{Theorem}[section]
\newtheorem{lemma}[theorem]{Lemma}
\newtheorem{corollary}[theorem]{Corollary}
\newtheorem{proposition}[theorem]{Proposition}
\newtheorem{assumption}[theorem]{Assumption}
\newtheorem{definition}[theorem]{Definition}
\newtheorem{remark}[theorem]{Remark}
\newtheorem*{remark*}{Remark}
\newtheorem{exmp}[theorem]{Example}

\begin{document}

\ifsupplement
\else
\renewcommand{\include}{\input}
\fi

\ifshowchanges\linenumbers\fi

\ifsupplement
    \title{Supplementary Material for\\
           ``Optimization via a Control-Centric Framework''}
    \author{}
    \maketitle

    \noindent
    This document contains the full proofs, supporting material, and the additional numerical results deferred from the main paper.
    Theorem, equation, and section numbers of the form ``Theorem~IV.6,''
    ``(44),'' etc.\ refer to the main paper
    ``Optimization via a Control-Centric Framework.''
    \vspace{0.5em}
\fi

\title{Optimization via a Control-Centric Framework}
\ifcameraready\author{Liraz Mudrik,~\IEEEmembership{Member,~IEEE}, Isaac Kaminer, Sean Kragelund, and Abram H.~Clark%
    \thanks{This work was supported in part by the Office of Naval Research Science of Autonomy Program under Grant No.\ N0001425GI01545 and Consortium for Robotics Unmanned Systems Education and Research at the Naval Postgraduate School.}%
    \thanks{L. Mudrik, I. Kaminer, and S. Kragelund are with the Department of Mechanical and Aerospace Engineering, Naval Postgraduate School, Monterey, CA 93943 USA.}%
    \thanks{A. H. Clark is with the Department of Physics, Naval Postgraduate School, Monterey, CA 93943 USA.}}%
\else\author{}\fi

    \ifcameraready\else
    \markboth{Journal of \LaTeX\ Class Files,~Vol.~XX, No.~X, XXXX~202X}%
    {Optimization via a Control-Centric Framework}
    \fi

    \ifcameraready\else\IEEEpubid{0000--0000/00\$00.00~\copyright~202X IEEE}\fi
	
	\maketitle

    \begin{abstract}
    Optimization plays a central role in intelligent systems and cyber-physical technologies, where speed and reliability of convergence directly impact performance. In control theory, optimization is the foundation of widely-used design methodologies such as linear quadratic regulation, $H_\infty$ control, and model predictive control. In contrast, this paper develops a control-centric framework for optimization itself, where algorithms are constructed directly from Lyapunov stability principles rather than being proposed first and analyzed afterward. 
    A key element is the stationarity \rev{mapping}, which encodes first-order optimality conditions and enables Lyapunov-based convergence analysis. 
    By pairing a Lyapunov function with a selectable decay law, we obtain continuous-time dynamics with guaranteed exponential, finite-time, fixed-time, or prescribed-time convergence. 
    \rev{The convergence regime thus becomes a design input: choosing among these regimes changes only the decay law, not the algorithm or its analysis.}
    Within this framework, we introduce three feedback realizations of increasing restrictiveness: the Hessian-gradient, Newton, and gradient dynamics.
    These constructions unify unconstrained optimization, extend to constrained problems via Lyapunov-consistent primal-dual dynamics, and broaden results for minimax and generalized Nash equilibrium seeking problems beyond exponential stability\rev{, the last two through centralized solvers}.
    In total, the framework spans six problem classes and four convergence regimes, yielding a unified design recipe across twenty-four combinations, \rev{some} of which, to the best of our knowledge, have no direct continuous-time counterpart in the prior literature.
    \end{abstract}

    \begin{IEEEkeywords}
    Lyapunov methods; control-centric optimization methods; gradient methods; Newton method; constrained optimization; finite-time convergence; feedback.
    \end{IEEEkeywords}

    \section{Introduction}\label{sec:intro}
    \IEEEPARstart{O}{ptimization} is a cornerstone of intelligent systems and cyber-physical technologies, where it underlies decision-making, resource allocation, and learning-based control. 
    In modern control theory, optimization-based approaches are likewise central. 
    Prominent examples include the linear quadratic regulator (LQR)~\cite{kalman_contributions_1960}, $H_2$ and $H_{\infty}$ control~\cite{doyle_state-space_1989}, and model predictive control (MPC)~\cite{rawlings_model_2024}, where the control design is obtained by formulating and solving an optimization problem. 
    These methods may be viewed as optimization-centric frameworks for control, since the optimization problem itself constitutes the foundation upon which control laws are built.    
    
    In contrast to optimization-centric approaches to control, this paper focuses on a control-centric framework for optimization itself, where continuous-time dynamics are designed as optimizers by directly invoking Lyapunov stability theory and its modern extensions. 
    This perspective parallels the classical theory of control Lyapunov functions (CLFs), in which stabilizing feedback laws are constructed directly from the Lyapunov function. 
    Foundational results include Artstein's equivalence between CLF existence and feedback stabilizability~\cite{artstein_stabilization_1983}, Sontag's universal construction of continuous stabilizing feedbacks~\cite{sontag_lyapunov-like_1983,sontag_universal_1989}, and Brockett's necessary condition for continuous static stabilization~\cite{brockett_asymptotic_1983}. 
    Building on this foundation, we introduce the notion of optimization Lyapunov functions (OLFs), which transfer these control-theoretic principles to the design of optimization algorithms. 
    In particular, tools such as finite-time (FT) stability~\cite{bhat_finite-time_2000}, fixed-time (FxT) stability~\cite{polyakov_nonlinear_2012}, and prescribed-time (PT) stability~\cite{song_time-varying_2017} provide systematic ways to guarantee convergence with user-specified temporal properties. 
    This control-centric viewpoint enables the construction of optimizer dynamics that meet explicit convergence requirements, rather than analyzing convergence only after an algorithm has been specified.
    
    A substantial body of prior work has applied Lyapunov-based methods to continuous-time optimization, primarily as analysis rather than design. The study of continuous-time primal-dual dynamics for constrained saddle-point problems originates with the seminal work of Arrow, Hurwicz, and Uzawa~\cite{arrow_studies_1958}, which laid the foundation for subsequent developments. 
    \rev{More recent works reach FT, FxT, predefined-time, or PT convergence for convex programs with nonlinear inequality constraints~\cite{wang_gradient-type_2022,wang_varying-parameter_2023}, in distributed optimization~\cite{zhang_balance_2025,zuo_achieving_2026,xia_predefined-time_2025}, in Nash equilibrium (NE) and generalized NE (GNE) seeking~\cite{shi_finite-time_2026,chen_prescribed_2025,chen_predefined_2025,guo_distributed_2025}, and in minimax optimization~\cite{song_predefined_2026}, each with a design and an analysis tailored to one problem class; others extend NE and GNE seeking to event-triggered communication and switching networks~\cite{zhang_distributed_2025,he_distributed_2025}.}
    Table~\ref{tab:landscape} locates our contributions in the landscape of continuous-time optimization across the different problem classes and convergence regimes.
    Complementary perspectives connect continuous-time optimizer design to optimal-control principles~\cite{ross_optimal_2019,ross_derivation_2023} and to the discrete-time synthesis of accelerated first-order algorithms via integral quadratic constraints~\cite{scherer_convex_2021,scherer_tutorial_2025}.

    \begin{table*}[!t]
    \centering
    \caption{Landscape of continuous-time optimization\rev{; bold entries name the OLF realizations new to each cell: Hessian-gradient dynamics (HGD) and Newton dynamics (ND); superscripts mark methods with a priori guarantees only for equality (E) or linear (L) constraints}.}
    \label{tab:landscape}
    \footnotesize
    \setlength{\tabcolsep}{3pt}
    \renewcommand{\arraystretch}{1.2}
    \begin{tabular}{l|c|c|c|c}
    \hline
    \textbf{Problem} & \textbf{Asymp./Exp.} & \textbf{FT} & \textbf{FxT} & \textbf{PT} \\
    \hline
    \hline
    Unconstrained optimization
    & \cite{brown_effective_1989}, {\color{green!45!black}\textbf{\rev{HGD}}}
    & \cite{cortes_finite-time_2006,romero_finite-time_2020}
    & \cite{garg_fixed-time_2021,budhraja_breaking_2022}, {\color{green!45!black}\textbf{\rev{HGD}}}
    & \cite{aal_controlled_2025}\rev{, \cite{zuo_achieving_2026}}, {\color{green!45!black}\textbf{\rev{HGD}}} \\
    Unconstrained minimax
    & \cite{cherukuri_saddle-point_2017,cherukuri_role_2018}, {\color{green!45!black}\textbf{\rev{HGD}}}
    & \cite{garg_fixed-time_2021}, {\color{green!45!black}\textbf{\rev{HGD}}}
    & \cite{garg_fixed-time_2021,garg_fixed-time_2022}, {\color{green!45!black}\textbf{\rev{HGD}}}
    & \rev{\cite{song_predefined_2026}}, {\color{green!45!black}\textbf{\rev{HGD}}} \\
    Nash equilibrium (NE)
    & \cite{gharesifard_distributed_2013}\rev{, \cite{zhang_distributed_2025}}, {\color{green!45!black}\textbf{\rev{HGD}}}
    & \rev{\cite{shi_finite-time_2026}}, {\color{green!45!black}\textbf{\rev{HGD}}}
    & \cite{poveda_fixed-time_2023}, {\color{green!45!black}\textbf{\rev{HGD}}}
    & \cite{xue_nash_2026}\rev{, \cite{chen_prescribed_2025}}, {\color{green!45!black}\textbf{\rev{HGD}}} \\
    \hline
    Constrained optimization
    & \cite{cherukuri_saddle-point_2017}\rev{\textsuperscript{E}}, \cite{feijer_stability_2010,cherukuri_asymptotic_2016,cherukuri_role_2018,allibhoy_control-barrier-function-based_2024}, {\color{green!45!black}\textbf{\rev{HGD}}}
    & \cite{ozaslan_exponential_2024}\rev{\textsuperscript{L}}, \cite{chen_convex_2018}\rev{\textsuperscript{E}, \cite{shi_finite-time_2023}\textsuperscript{L}, \cite{wang_gradient-type_2022}}, {\color{green!45!black}\textbf{\rev{HGD}}}
    & \rev{\cite{garg_fixed-time_2021}\textsuperscript{E},} \cite{ozaslan_exponential_2024}\rev{\textsuperscript{L}, \cite{wang_varying-parameter_2023}}, {\color{green!45!black}\textbf{\rev{HGD}}}
    & \cite{chen_achieving_2026}\rev{\textsuperscript{E}}, {\color{green!45!black}\textbf{\rev{HGD}}} \\
    Constrained minimax
    & {\color{green!45!black}\textbf{\rev{HGD, ND}}}
    & {\color{green!45!black}\textbf{\rev{HGD, ND}}}
    & {\color{green!45!black}\textbf{\rev{HGD, ND}}}
    & {\color{green!45!black}\textbf{\rev{HGD, ND}}} \\
    Generalized NE (GNE)
    & \cite{sun_continuous-time_2021}\rev{, \cite{martinez-piazuelo_payoff_2022}}, {\color{green!45!black}\textbf{\rev{HGD}}}
    & \rev{\cite{shi_finite-time_2026}\textsuperscript{L}}, {\color{green!45!black}\textbf{\rev{HGD, ND}}}
    & \rev{\cite{hua_optimizing_2025}\textsuperscript{L}}, {\color{green!45!black}\textbf{\rev{HGD, ND}}}
    & {\color{green!45!black}\textbf{\rev{HGD, ND}}} \\
    \hline
    \end{tabular}
    \end{table*}

    \IEEEpubidadjcol
    
    Despite this progress, most existing methods follow the pattern of proposing an optimization algorithm first and only then analyzing its convergence using control-theoretic tools.
    In contrast, our contribution is a framework for designing optimizers directly according to desired convergence properties.
    \rev{Designing first has two advantages over this pattern. First, the convergence regime becomes a design input: under two of our realizations, the OLF follows the chosen decay law exactly, so the time at which it reaches any tolerance is known before the run, and choosing among different convergence regimes changes only a scalar law, not the structure of the algorithm or its analysis. Second, the Lyapunov function that generates the dynamics also certifies their convergence, so extending the framework to a new problem class amounts to writing its stationarity mapping and verifying the regularity conditions of the chosen realization.}
    This perspective not only recovers known results as special cases but also closes several gaps in the literature.
    Our framework further enables the design of second-order continuous-time optimizers that exploit curvature information without requiring explicit Hessian inversion.

    \begin{figure}[t]
    \centering
    \includegraphics[width=0.9\linewidth]{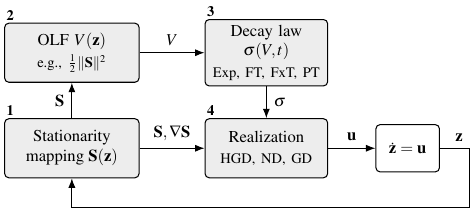}
    \caption{\rev{The optimizer as a feedback loop around the integrator $\dot{\mathbf{z}}=\mathbf{u}$. The shaded blocks are the four design choices: (1) the stationarity mapping $\mathbf{S}$, which encodes the optimality conditions; (2) the OLF $V$, which vanishes exactly at the zeros of $\mathbf{S}$; (3) the decay law $\sigma$, which sets the convergence regime; (4) the realization, which chooses $\mathbf{u}$ so that $V$ decays as $\sigma$ prescribes.}}
    \label{fig:concept}
    \end{figure}

    The main contribution of this paper is the introduction of the optimization Lyapunov function (OLF) framework, which adapts the classical notion of CLFs to continuous-time optimization by decoupling the design of the stabilizing feedback from the specification of the decay law. \rev{Fig.~\ref{fig:concept} summarizes its four design choices.} The familiar exponential, FT, FxT, and PT behaviors are presented as representative examples of this general design capability. \rev{The PT gain of the analysis is singular at the prescribed time; the simulations use a bounded version that reaches a preset tolerance at that time.} Within this formulation, we develop three continuous-time realizations of increasing restrictiveness: the Hessian-gradient dynamics (HGD), Newton dynamics (ND), and gradient dynamics (GD), which share the same Lyapunov-based synthesis principle while differing in structure and computational complexity. \rev{The result is a unified design recipe in which a single realization, HGD, applies in all twenty-four problem-rate combinations of Table~\ref{tab:landscape} and enforces the chosen decay law exactly without inverting any matrix (for unconstrained optimization under the FT law, it coincides with the flow of~\cite{romero_finite-time_2020}). To the best of our knowledge, it also gives the first projection-free FxT and PT methods for constrained minimax and GNE problems with coupled inequality constraints.} HGD is the universal realization, used in every theorem; ND requires an invertible Jacobian of the stationarity \rev{mapping} and applies in the strongly convex and strongly monotone settings; GD applies in the unconstrained classes. The framework extends to constrained programs via Lyapunov-consistent Karush--Kuhn--Tucker (KKT) dynamics, and to minimax and GNE seeking problems\rev{, with the same four decay laws}.
    
    The remainder of this paper is organized as follows. 
    Section~\ref{sec:back} reviews the required background and introduces a family of \rev{decay} laws covering exponential, FT, FxT, and PT regimes. 
    Section~\ref{sec:control_centric_olf} formalizes the notion of OLFs and presents the control-centric methodology for constructing optimizer dynamics with selectable convergence rates. 
    Sections~\ref{sec:constrained}--\ref{sec:gne} apply this methodology to different problem classes, beginning with constrained optimization and continuing with minimax formulations and GNE seeking problems. 
    Finally, Section~\ref{sec:conc} concludes the paper.

    \section{Mathematical Background}
    \label{sec:back}
    In this section, we introduce the general mathematical framework and the assumptions that will be used throughout the paper. These assumptions are kept minimal and apply across the different problem classes treated in later sections. Additional assumptions required for specific cases (e.g., convexity, constraint qualifications) will be stated within their respective sections.

    We begin by presenting the notation used in this paper. 
    Vectors \rev{and matrices} are denoted in bold, e.g., $\mathbf{x}\in\mathbb{R}^n$\rev{, $\mathbf{1}$ is the vector of ones}, and $\|\cdot\|$ denotes the Euclidean norm. 
    We write $J\in\mathcal{C}^k$ to denote that $J$ is $k$-times continuously differentiable with respect to its arguments.
    For $J\in\mathcal{C}^2$, its gradient and Hessian are denoted by $\nabla J(\mathbf{x})$ and $\nabla^2 J(\mathbf{x})$, respectively. 
    Inner products are written as $\langle \mathbf{u},\mathbf{v}\rangle = \mathbf{u}^\top \mathbf{v}$. 
    For vectors (or blocks) $\mathbf{v}_1,\dots,\mathbf{v}_k$, we use the stacking operator
    \begin{equation}
    \operatorname{col}(\mathbf{v}_1,\dots,\mathbf{v}_k) \coloneqq 
    \begin{bmatrix} \mathbf{v}_1^\top & \cdots & \mathbf{v}_k^\top \end{bmatrix}^\top .
    \end{equation}

    We also make the following underlying assumptions, which will be used throughout this work, even if not explicitly stated.
    \begin{assumption}
    \label{ass:smooth}
    The objective function $J(\mathbf{x})$ is \rev{twice continuously differentiable, $J\in\mathcal{C}^2$}.
    \end{assumption}
    \begin{assumption}
    \label{ass:finite}
    The optimization problem admits a finite optimal value, i.e., $J^\star = \min_{\mathbf{x}} J(\mathbf{x})$.
    \end{assumption}

    We emphasize that, under these standing assumptions, the convergence results established in this paper are generally local, unless stronger properties such as strong convexity or \rev{strong} convexity-concavity are imposed. \rev{The decay law does not enter these hypotheses, so every result holds for all four laws; \iffullversion Appendix~\ref{app:assumptions}\fi{} summarizes the hypotheses and reach of each result.}
    \subsection{Stationarity \rev{Mapping}}
    The central object in our framework is the stationarity \rev{mapping} $\mathbf{S}\colon\mathbb{R}^N\to\mathbb{R}^N$, which encodes the stationarity conditions of the optimization problem at hand. The argument $\mathbf{z}\in\mathbb{R}^N$ stacks all relevant variables, primal and dual, and the dimension $N$ varies depending on the problem class (unconstrained, constrained, minimax, or GNE), as specialized in Secs.~\ref{sec:constrained}--\ref{sec:gne}. 
    The unifying principle is that \rev{$\mathbf{S}$ vanishes at the solutions sought or, with inequality constraints, at smoothed approximations of them}.

    \begin{exmp}[Stationarity in Unconstrained Optimization]
        For the unconstrained minimization problem $\min_{\mathbf{x}} J(\mathbf{x})$, the stationarity \rev{mapping} reduces to the gradient, $\mathbf{S}(\mathbf{z}) = \nabla J(\mathbf{x})$, where $\mathbf{z} = \mathbf{x}$. 
        Hence, we consider the quadratic Lyapunov candidate
        \begin{equation}
            V(\mathbf{x}) = \tfrac{1}{2}\|\mathbf{S}(\mathbf{z})\|^2
            = \tfrac{1}{2}\|\nabla J(\mathbf{x}) \|^2,
            \label{eq:lyap}
        \end{equation}
        which vanishes precisely at stationary points.
    \end{exmp} 
    Throughout the paper, feedback dynamics will be designed so that $\dot V(\mathbf{z})$ satisfies a \rev{chosen decay} law, such as exponential, FT, FxT, or PT decay, as detailed below.

    \subsection{Desired Convergence Properties}
    \label{sec:conv_prop}
    We introduce a family of decay conditions for Lyapunov functions that generate different convergence behaviors. 
    Rather than analyzing specific algorithms, we adopt a template-based viewpoint: given a candidate Lyapunov function $V$, the choice of a decay inequality $\dot V \leq -\sigma(V,t)$ determines the temporal properties of convergence. 
    By varying the form of $\sigma$, one recovers exponential, FT, FxT, and PT regimes. 
    The following lemmas summarize well-established results from control and optimization theory that will serve as fundamental ingredients throughout the remainder of the paper.
    
    \begin{lemma}[Exponential Stability~\cite{khalil_nonlinear_2002}]
    \label{lem:exp}
    If $\dot V(\mathbf{z})\le -\rev{c} V(\mathbf{z})$ for some $\rev{c}>0$, then $V(\mathbf{z}(t))\le V(\mathbf{z}(0))e^{-\rev{c} t}$.
    \end{lemma}
    
    \begin{lemma}[FT Stability~\cite{bhat_finite-time_2000}]
    \label{lem:finite}
    If $\dot V(\mathbf{z})\le -\rev{k} V(\mathbf{z})^{\rev{\gamma}}$ for some $\rev{k}>0$ and $\rev{\gamma}\in(0,1)$, then $V(\mathbf{z}(t))$ reaches zero in finite time $\rev{T_{\mathrm{s}}}(\mathbf{z}(0))\le V(\mathbf{z}(0))^{1-\rev{\gamma}}/[\rev{k}(1-\rev{\gamma})]$.
    \end{lemma}
    
    \begin{lemma}[FxT Stability~\cite{polyakov_nonlinear_2012}]
    \label{lem:fixed}
    If $\dot V(\mathbf{z})\le -\rev{a} V(\mathbf{z})^{\rev{\gamma}} - \rev{b} V(\mathbf{z})^{\rev{\delta}}$ for some $\rev{a,b}>0$, $\rev{\gamma}\in(0,1)$, and $\rev{\delta}>1$, then $V(\mathbf{z}(t))$ converges to zero in fixed time $\rev{T_{\mathrm{s}}}\le 1/[\rev{a}(1-\rev{\gamma})]+1/[\rev{b}(\rev{\delta}-1)]$. This bound is independent of $\mathbf{z}(0)$, the defining feature distinguishing FxT from FT stability.
    \end{lemma}
    
    \begin{lemma}[PT Stability~\cite{song_time-varying_2017}]
    \label{lem:prescribed}
    If $\dot V(\mathbf{z})\le -\dfrac{\mu}{T-t}\,V(\mathbf{z})$ for some $\mu>0$ and prescribed $T>0$, $t\in[0,T)$, then \rev{$V(\mathbf{z}(t))\le V(\mathbf{z}(0))\,(1-t/T)^{\mu}$, so $V$ reaches zero no later than $t=T$}, independent of the initial condition\rev{, and exactly at $t=T$ when the inequality holds with equality}.
    \end{lemma}
    This work focuses on these four convergence regimes, but we note that other options exist for which the generalization of our framework is straightforward.
    For other convergence regimes, the reader is referred to~\cite{song_prescribed-time_2023}.

    The \rev{decay} laws introduced here are independent of any specific problem class. 
    They can be combined with various Lyapunov constructions in the subsequent sections. 
    In the following developments, these laws are employed to design optimizer dynamics for unconstrained and constrained optimization, minimax problems, and GNE seeking, where the desired convergence behavior is achieved by selecting the appropriate decay condition.

    \section{Control-Centric Framework}
    \label{sec:control_centric_olf}
    This section develops a control-centric recipe for constructing continuous-time dynamics that solve optimization and equilibrium problems via Lyapunov methods. 
    The key idea is to specialize a Lyapunov-driven design to a “plant” with ideal dynamics
    \begin{equation}
        \dot{\mathbf{z}} \;=\; \mathbf{u},
        \label{eq:plant}
    \end{equation}
    and to choose the Lyapunov function so that its set of minima coincides with the stationarity set of the optimization problem. 
    
    This integrator structure is, in fact, necessary for the convergence of first-order algorithms~\cite{scherer_convex_2021}.
    For any candidate $V:\mathbb{R}^N\to\mathbb{R}_{\ge 0}$, the derivative along~\eqref{eq:plant} is
    \begin{equation}
        \dot V(\mathbf{z},t) \;=\; \nabla V(\mathbf{z})^\top \mathbf{u}(\mathbf{z},t).
        \label{eq:Vdot_general}
    \end{equation}
    Our aim is to realize $\mathbf{u}(\mathbf{z},t)$ so that the decay inequality
    \begin{equation}
        \dot V(\mathbf{z},t) \;\le\; -\,\sigma\!\big(V(\mathbf{z}),t\big)
        \label{eq:V_decay}
    \end{equation}
    holds, where the \rev{decay} law $\sigma(\cdot,\cdot)$ encodes the desired convergence rate, e.g., exponential, FT, FxT, or PT, and will be instantiated via the rate lemmas collected in Sec.~\ref{sec:conv_prop}.

    \subsection{Optimization Lyapunov Function (OLF)}
    A central ingredient of our framework is the OLF which, 
    analogously to Lyapunov functions for stability, certifies convergence of a dynamical optimizer to the stationarity set.
    This notion allows us to treat algorithm design as a control problem: the OLF encodes the desired optimality conditions, and the choice of decay law shapes the rate of convergence.
    
    \begin{definition}[Optimization Lyapunov Function]
    \label{def:olf}
    Let $\mathcal{S} \!=\! \{\mathbf{z}\in \mathbb{R}^N : \mathbf{S}(\mathbf{z})=\mathbf{0}\}$ denote the stationarity set. An optimization Lyapunov function is any continuously differentiable map $V:\mathbb{R}^N\!\to\!\mathbb{R}_{\ge 0}$ such that $V(\mathbf{z})=0$ if and only if $\mathbf{z}\in\mathcal{S}$.
    \end{definition}
    
    
    
    \subsection{OLF Regularity Condition}
    Because $\mathbf{u}$ enters~\eqref{eq:Vdot_general} through $\nabla V$, one cannot force $\dot V<0$ at states where $\nabla V=\mathbf{0}$ while $V>0$.
    
    \begin{assumption}[OLF Regularity]
    \label{ass:OLF_reg}
    Let $V(\mathbf{z})$ be an OLF as in Definition~\ref{def:olf}. 
    There exists an open set $\mathcal{U}\subseteq\mathbb{R}^N$ with $\mathcal{S}\subseteq\mathcal{U}$ such that
    \begin{equation}
        \nabla V(\mathbf{z}) \;\neq\; \mathbf{0}
        \label{eq:nonsing}
    \end{equation}
    for all $\mathbf{z}\in\mathcal{U}$ with $V(\mathbf{z})\neq 0$.
    \end{assumption}
    If $\nabla V(\mathbf{z})=\mathbf{0}$ while $V(\mathbf{z})>0$, then~\eqref{eq:Vdot_general} gives $\dot V(\mathbf{z},t)=0$ for all admissible $\mathbf{u}(\mathbf{z},t)$. Hence, no feedback based on $V$ can enforce $\dot V<0$ at such points.
    This assumption holds automatically in certain structured problems, such as a strongly convex cost function with convex inequality and affine equality constraints. We establish this formally in the ensuing results.
    
    When the quadratic OLF, $V(\mathbf{z})=\tfrac12\|\mathbf{S}(\mathbf{z})\|^2$, is used, the OLF regularity assumption becomes
    \begin{equation}
        \nabla V(\mathbf{z}) \;=\; \nabla \mathbf{S}(\mathbf{z})^\top \mathbf{S}(\mathbf{z})
        \;\neq\; \mathbf{0}
        \label{eq:nonortho}
    \end{equation}
    whenever $\mathbf{S}(\mathbf{z})\neq \mathbf{0}$. Equivalently,
    \begin{equation}
    \mathbf{S}(\mathbf{z}) \;\notin\; \ker\big(\nabla \mathbf{S}(\mathbf{z})^\top\big),
    \label{eq:SnotKer}
    \end{equation}
    where $\ker(\cdot)$ denotes the null space. We emphasize that OLF regularity is a condition on $\nabla V$, not on $\nabla \mathbf{S}$: it does not require $\nabla\mathbf{S}(\mathbf{z})$ to be invertible, only that~\eqref{eq:SnotKer} holds.
    
    
    \subsection{Feedback Designs Enforcing the \rev{Decay} Law}
    We present three constructive designs that enforce the decay law~\eqref{eq:V_decay}; they are ordered by restrictiveness on $\nabla V(\mathbf{z})$. Each can be paired with any admissible \rev{decay} law $\sigma(\cdot,\cdot)$.
    
    \begin{lemma}[Hessian-Gradient Dynamics (HGD)]
    \label{lem:w_pseudo}
    Let $V(\mathbf{z})$ be an OLF.
    If Assumption~\ref{ass:OLF_reg} holds on $\mathcal{U}$, then for every $\mathbf{z}\in\mathcal{U}\rev{\setminus\mathcal{S}}$ the feedback
    \begin{equation}
        \mathbf{u}(\mathbf{z},t) \;=\; -\,\frac{\sigma\!\big(V(\mathbf{z}),t\big)}{\|\nabla V(\mathbf{z})\|^2}\,\nabla V(\mathbf{z})
        \label{eq:u_w}
    \end{equation}
    enforces~\eqref{eq:V_decay} with equality $\dot V=-\sigma(V,t)$,
    where the \rev{decay} law $\sigma$ takes one of the following forms:
    \begin{equation}
    \sigma(V,t) \;=\;
    \begin{cases}
    c\,V, & \text{(Exp),} \\[4pt]
    k\,V^{\gamma}, & \text{(FT), } \; 0<\gamma<1, \\[4pt]
    a\,V^{\gamma} + b\,V^{\delta}, & \text{(FxT), } \; 0<\gamma<1<\delta, \\[4pt]
    \dfrac{\mu}{T - t}\,V, & \text{(PT), } \; t\in[0,T),
    \end{cases}
    \label{eq:sigma_def}
    \end{equation}
    with positive constants $a,b,c,k,\mu$ and $T$.
    \end{lemma}
    The term Hessian-gradient dynamics (HGD) is motivated by the unconstrained case: when $\mathbf{S}(\mathbf{z}) = \nabla J(\mathbf{x})$ with a quadratic Lyapunov function~\eqref{eq:lyap}, the feedback involves the Hessian-gradient product $\nabla^2 J(\mathbf{x})\,\nabla J(\mathbf{x})$. This yields a normalized second-order descent distinct from both gradient and Newton dynamics.

    \begin{proof}
    Substitute \eqref{eq:u_w} into \eqref{eq:Vdot_general} to get
    \begin{equation}
        \dot V
        = \nabla V^\top \!\Big(-\,\tfrac{\sigma(V,t)}{\|\nabla V\|^2}\,\nabla V\Big)
        = -\,\sigma(V,t).
    \end{equation}
    The cases in~\eqref{eq:sigma_def} follow directly from Lemmas~\ref{lem:exp}--\ref{lem:prescribed}.
    \end{proof}

    \begin{exmp}[HGD for Unconstrained Optimization] \label{ex:hgd_unconstrained}
    Let $\mathbf{S}(\mathbf{x})=\nabla J(\mathbf{x})$ and $V(\mathbf{x})=\tfrac12\|\mathbf{S}(\mathbf{x})\|^2$. Define the Hessian-gradient product
    \begin{equation}
    \mathbf{w}(\mathbf{x}) \;=\; \nabla^2 J(\mathbf{x})\,\nabla J(\mathbf{x}).
    \label{eq:hgd_uc_w}
    \end{equation}
    Assume $\mathbf{w}(\mathbf{x})\neq\mathbf{0}$ whenever $\nabla J(\mathbf{x})\neq\mathbf{0}$, i.e., that Assumption~\ref{ass:OLF_reg} holds on $\mathcal{U}$. 
    Set the feedback
    \begin{equation}
    \mathbf{u}(\mathbf{x},t) \;=\; -\,\frac{\sigma\!\big(V(\mathbf{x}),t\big)}{\|\mathbf{w}(\mathbf{x})\|^2}\,\mathbf{w}(\mathbf{x}),
    \label{eq:hgd_uc_u}
    \end{equation}
    valid for any decay law $\sigma$ in~\eqref{eq:sigma_def}. 
    Then along trajectories,
    \begin{equation}
    \dot V \;=\; \nabla J^\top \nabla^2 J \,\mathbf{u}
    \;=\; -\,\sigma\!\big(V(\mathbf{x}),t\big),
    \label{eq:hgd_uc_Vdot}
    \end{equation}
    so the \rev{decay} law is enforced with equality.
    
    If $J(\mathbf{x})=\tfrac12\mathbf{x}^\top \mathbf{A}\mathbf{x}-\mathbf{b}^\top\mathbf{x}$ with $\mathbf{A}\succeq \mathbf{0}$ and $\mathbf{b}\in\operatorname{range}(\mathbf{A})$, then $\nabla J=\mathbf{A}\mathbf{x}-\mathbf{b}$ and $\mathbf{w}=\mathbf{A}\nabla J\neq\mathbf{0}$ whenever $\nabla J\neq\mathbf{0}$, so \eqref{eq:hgd_uc_u} yields \eqref{eq:hgd_uc_Vdot} under any $\sigma$ in \eqref{eq:sigma_def}.
    \end{exmp}
    
        Lemma~\ref{lem:w_pseudo} applies to any OLF. The two remaining realizations, and the analysis in the rest of the paper, adopt the standard quadratic candidate
    \begin{equation}
        V(\mathbf{z}) \;=\; \tfrac{1}{2}\big\|\mathbf{S}(\mathbf{z})\big\|^2.
        \label{eq:V_quad}
    \end{equation}
    
    \begin{lemma}[Newton Dynamics (ND)]
    \label{lem:pseudo}
    Suppose $\nabla \mathbf{S}(\mathbf{z})$ is invertible on $\mathcal{U}$. 
    Then for every $\mathbf{z}\in\mathcal{U}\rev{\setminus\mathcal{S}}$,
    \begin{equation}
        \mathbf{u}(\mathbf{z},t) \;=\; -\,\rev{\frac{\sigma\!\big(V(\mathbf{z}),t\big)}{2V(\mathbf{z})}}\,\nabla \mathbf{S}(\mathbf{z})^{-1} \,\mathbf{S}(\mathbf{z})
        \label{eq:u_pseudo}
    \end{equation}
    satisfies~\eqref{eq:V_decay} with equality, 
    where $\sigma(\cdot,\cdot)$ is chosen as presented in~\eqref{eq:sigma_def}.
    \end{lemma}
    The term Newton dynamics (ND) reflects that in the unconstrained case $\mathbf{S}(\mathbf{z}) = \nabla J(\mathbf{x})$ with a quadratic Lyapunov function~\eqref{eq:lyap}, the feedback law reduces to the continuous-time Newton method~\cite{brown_effective_1989} $\dot{\mathbf{x}} = -\nabla^2 J(\mathbf{x})^{-1}\nabla J(\mathbf{x})$ that is known to guarantee exponential convergence.

    \begin{proof}
    With $V=\tfrac12\|\mathbf{S}\|^2$ and $\nabla V=\nabla \mathbf{S}^\top \mathbf{S}$,
    \begin{equation}
    \dot V \;=\; -\,\rev{\frac{\sigma(V,t)}{2V}}\,\mathbf{S}^\top \nabla \mathbf{S}\,\nabla \mathbf{S}^{-1} \mathbf{S}.
    \end{equation}
    Since $\nabla \mathbf{S}$ is invertible, $\dot V=-\rev{\frac{\sigma(V,t)}{2V}}\,\mathbf{S}^\top\mathbf{S}=\rev{-\sigma(V,t)}$.
    \end{proof}

    \begin{exmp}[ND for Unconstrained Optimization]
    \label{ex:nd_unconstrained}
    Let $\mathbf{S}(\mathbf{x})=\nabla J(\mathbf{x})$ and $V(\mathbf{x})=\tfrac12\|\mathbf{S}(\mathbf{x})\|^2$. Assume $\nabla^2 J(\mathbf{x})\succ\mathbf{0}$ for all $\mathbf{x}$, a sufficient condition for strict convexity. Set
    \begin{equation}
    \mathbf{u}(\mathbf{x},t)\;=\;-\big[\nabla^2 J(\mathbf{x})\big]^{-1}\,\nabla J(\mathbf{x})\,\rev{\frac{\sigma\!\big(V(\mathbf{x}),t\big)}{2V(\mathbf{x})}},
    \label{eq:nd_uc_u}
    \end{equation}
    valid for any decay law $\sigma$ in~\eqref{eq:sigma_def}. 
    Then
    \begin{equation}
    \dot V = \nabla J^\top \nabla^2 J\,\mathbf{u} = -\,\|\nabla J\|^2\,\rev{\frac{\sigma(V,t)}{2V}} = -\,\sigma\!\big(V(\mathbf{x}),t\big),
    \label{eq:nd_uc_Vdot}
    \end{equation}
    so the \rev{decay} law is enforced with equality.
    
    If $J(\mathbf{x})=\tfrac12\mathbf{x}^\top \mathbf{A}\mathbf{x}-\mathbf{b}^\top\mathbf{x}$ with $\mathbf{A}\succ \mathbf{0}$, then $\nabla J=\mathbf{A}\mathbf{x}-\mathbf{b}$ and $\mathbf{u}=-\mathbf{A}^{-1}\nabla J\,\rev{\sigma(V,t)/(2V)}$ satisfies \eqref{eq:nd_uc_Vdot} for any $\sigma$ in \eqref{eq:sigma_def}.
    \end{exmp}

    \begin{lemma}[Gradient Dynamics (GD)]
    \label{lem:svb}
    Suppose there exists $m>0$ such that for all $\mathbf{z}$$\,\in\mathcal{U}$,
    \begin{equation}
    \rev{\operatorname{sym}(\nabla\mathbf{S}(\mathbf{z}))\coloneqq{}}\frac{\nabla \mathbf{S}(\mathbf{z})+\nabla \mathbf{S}(\mathbf{z})^\top}{2} \succeq m \mathbf{I} .
    \label{eq:symm_part_cond}
    \end{equation}
    Then for every $\mathbf{z}\in\mathcal{U}\rev{\setminus\mathcal{S}}$ the feedback
    \begin{equation}
    \mathbf{u}(\mathbf{z},t) \;=\; -\,\rev{\frac{\sigma\!\big(V(\mathbf{z}),t\big)}{2mV(\mathbf{z})}}\,\mathbf{S}(\mathbf{z})
    \label{eq:u_gamma}
    \end{equation}
    satisfies
    \begin{equation}
    \dot V \;\le\; -\,\sigma(V,t),
    \label{eq:gd_vdot_bound}
    \end{equation}
    where $\sigma(\cdot,\cdot)$ is chosen as in~\eqref{eq:sigma_def}.
    \end{lemma}
    The term gradient dynamics (GD) is chosen in analogy with the unconstrained case: when $\mathbf{S}(\mathbf{z}) = \nabla J(\mathbf{x})$ with a quadratic Lyapunov function~\eqref{eq:lyap}, the feedback reduces to the classical gradient descent dynamics~\cite{brown_effective_1989} $\dot{\mathbf{x}} = -\nabla J(\mathbf{x})$ that is known to guarantee exponential convergence. \rev{Several known gradient and Newton flows with FT, FxT, and PT convergence are special cases of GD and ND under the corresponding decay laws, up to constant gains (see \iffullversion Appendix~\ref{app:known_flows}\fi).}
    The proof and an example of GD for unconstrained optimization are deferred to \iffullversion Appendix~\ref{app:gd_proof}\fi.

    The three realizations are nested in restrictiveness: Lemma~\ref{lem:w_pseudo} (HGD) requires only Assumption~\ref{ass:OLF_reg} and applies to any OLF; Lemma~\ref{lem:pseudo} (ND) requires invertibility of $\nabla\mathbf{S}$ (which implies Assumption~\ref{ass:OLF_reg}) and uses the quadratic OLF~\eqref{eq:V_quad}; Lemma~\ref{lem:svb} (GD) is strongest, requiring $\operatorname{sym}(\nabla\mathbf{S})\succeq m\mathbf{I}$. \rev{With a dense Jacobian, one evaluation of the feedback costs $O(N^2)$ operations for HGD and $O(N^3)$ for ND, and both store the $N^2$ entries of $\nabla\mathbf{S}$, whereas GD costs and stores $O(N)$. HGD, however, uses $\nabla\mathbf{S}$ only through the vector $\nabla\mathbf{S}^\top\mathbf{S}$, which can be computed without forming $\nabla\mathbf{S}$, for instance by automatic differentiation, and then stores only vectors (see \iffullversion Appendix~\ref{app:scaling}\fi).} The \rev{decay} law $\sigma(\cdot,\cdot)$ is left open as a design choice.
    \iffullversion
    \begin{remark}[Relation to Sontag’s Universal Construction]
    Sontag’s classical work~\cite{sontag_universal_1989} considered nonlinear systems affine in the control and showed that, whenever a control-Lyapunov function exists, one can explicitly construct a feedback law that ensures asymptotic stabilization. The key idea is that the feedback recipe is universal: it applies to any system once a Lyapunov function is available. Our framework plays an analogous role in optimization: given an optimization Lyapunov function, the feedback designs in Lemmas~\ref{lem:w_pseudo},~\ref{lem:pseudo}, and~\ref{lem:svb} provide universal constructions that not only ensure convergence but also shape its timing with exponential, FT, FxT, or PT guarantees.
    \end{remark}
    \fi

\subsection{Forward Invariance of Sublevel Sets}
\label{subsec:regularity}

The convergence theorems that follow require the closed-loop trajectory to remain in the set $\mathcal{U}$ on which Assumption~\ref{ass:OLF_reg} holds. This subsection removes that requirement: under a growth condition relating $\|\nabla V\|$ to $V$, sublevel sets of $V$ below a threshold are forward invariant, so initializing within one keeps the trajectory in $\mathcal{U}$. The same condition bounds the length of the gradient-flow paths of $V$; the HGD feedback~\eqref{eq:u_w} rescales $-\nabla V$ by a positive scalar, so its trajectories traverse these same paths and inherit the bound independently of the \rev{decay} law, as established in \iffullversion Appendix~\ref{app:fwd_inv_proof}\fi. We state the condition as an explicit assumption; when it cannot be verified, a regularized variant of the HGD feedback retains convergence.

\begin{assumption}[\L{}ojasiewicz Inequality]
\label{ass:lojasiewicz}
There exists a neighborhood $\mathcal{N}\subseteq\mathcal{U}$ of $\mathcal{S}\cap\mathcal{U}$ and constants $c_{\mathrm{L}}>0$ and $\alpha\in\rev{[1/2,1)}$ such that
\begin{equation}
\|\nabla V(\mathbf{z})\| \;\ge\; c_{\mathrm{L}}\,V(\mathbf{z})^{\alpha},
\qquad \forall\,\mathbf{z}\in\mathcal{N},
\label{eq:lojasiewicz}
\end{equation}
where $\mathcal{U}$ is the open set of Assumption~\ref{ass:OLF_reg}. The case $\alpha=1/2$ is the Polyak--\L{}ojasiewicz (PL) condition.
\end{assumption}
\rev{For the quadratic OLF~\eqref{eq:V_quad}, $\|\nabla V\|\le\sqrt{2}\,\|\nabla\mathbf{S}\|\,V^{1/2}$, so~\eqref{eq:lojasiewicz} cannot hold with $\alpha<1/2$ where $\nabla\mathbf{S}$ is bounded.}
In the structured problems treated below, Assumption~\ref{ass:lojasiewicz} need not be posited at the outset. PL via a nonsingular Jacobian makes this precise: on a compact sublevel set $\Omega_c\subseteq\mathcal{U}$ on which $\nabla\mathbf{S}$ is nonsingular, $\|\nabla V(\mathbf{z})\|^2\ge 2\sigma_c^{\,2}V(\mathbf{z})$ with $\sigma_c\coloneqq\min_{\mathbf{z}\in\Omega_c}\sigma_{\min}(\nabla\mathbf{S}(\mathbf{z}))>0$, so Assumption~\ref{ass:lojasiewicz} holds on $\Omega_c$ with $\alpha=1/2$ and $c_{\mathrm{L}}=\sigma_c\sqrt{2}$; this is \iffullversion stated and proved as Lemma~\ref{lem:lin_PL} in Appendix~\ref{app:lin_PL}\fi.

\begin{lemma}[Forward Invariance of Sublevel Sets]
\label{lem:fwd_inv}
Let $V$ be an OLF satisfying Assumptions~\ref{ass:OLF_reg} and~\ref{ass:lojasiewicz}\rev{, with $V$ bounded away from zero on $\overline{\mathcal{U}}\setminus\mathcal{N}$}, and let $\mathbf{u}(\mathbf{z},t)$ be any of the feedbacks of Lemmas~\ref{lem:w_pseudo}, \ref{lem:pseudo}, or \ref{lem:svb}, so that $\dot V\le 0$ on $\mathcal{U}\setminus\mathcal{S}$. Then there exists $\rho^\ast>0$ such that for every $c\in(0,\rho^\ast]$ the sublevel set
\(
\Omega_c\coloneqq\{\mathbf{z}\in\mathcal{U}:V(\mathbf{z})\le c\}
\)
is closed in $\mathcal{U}$ and forward invariant under the closed-loop dynamics; in particular, $\mathbf{z}(0)\in\Omega_c$ implies $\mathbf{z}(t)\in\Omega_c\subseteq\mathcal{U}$ \rev{until it reaches $\mathcal{S}$}. If additionally $\mathcal{S}\cap\mathcal{U}$ is bounded, then $\Omega_c$ is bounded.
\end{lemma}

\begin{proof}[Proof sketch]
Continuity of $V$ gives closedness, and an auxiliary gradient flow $\dot{\boldsymbol{\xi}}=-\nabla V(\boldsymbol{\xi})$ with Assumption~\ref{ass:lojasiewicz} bounds $\mathrm{dist}(\mathbf{z},\mathcal{S}\cap\mathcal{U})$ uniformly on $\Omega_c$, placing $\Omega_c$ in a tube of finite radius $r_c\coloneqq c^{1-\alpha}/[c_{\mathrm{L}}(1-\alpha)]$ (bounded when $\mathcal{S}\cap\mathcal{U}$ is). Choosing $\rho^\ast>0$ with $\Omega_{\rho^\ast}\subseteq\mathcal{N}$ and tube in $\mathcal{U}$, the inequality $\dot V\le 0$ on $\Omega_c\setminus\mathcal{S}$ with~\cite[Thm.~4.1]{khalil_nonlinear_2002} gives forward invariance for every $c\in(0,\rho^\ast]$. The complete argument is given in \iffullversion Appendix~\ref{app:fwd_inv_proof}\fi.
\end{proof}

\rev{The condition on $V$ outside $\mathcal{N}$ fails only if $V$ approaches zero along a sequence with $\|\mathbf{z}\|\to\infty$, in which case every sublevel set of $V$ is unbounded. It holds whenever $\mathcal{U}$ is bounded, and in the regional and global results below, where $\mathcal{N}$ is a sublevel set $\{V\le c'\}$, outside which $V>c'$.} Lemma~\ref{lem:fwd_inv} thus shows that initialization in a compact sublevel set suffices. In certain structured problems discussed below $\nabla\mathbf{S}$ is nonsingular and Assumption~\ref{ass:OLF_reg} holds globally on $\mathbb{R}^N$, so $\mathcal{U}=\mathbb{R}^N$ and $\rho^\ast$ imposes no restriction; what remains is compactness of $\Omega_c$, which is secured by a separate threshold $c^\ast$ introduced \rev{below}.

\begin{remark}[Regularized HGD]
\label{rem:regularizers}
When Assumption~\ref{ass:lojasiewicz} cannot be verified, the HGD feedback~\eqref{eq:u_w} can be replaced by the Tikhonov-regularized variant
\begin{equation}
\mathbf{u}_\eta(\mathbf{z},t)
\;=\;-\,\frac{\sigma\!\big(V(\mathbf{z}),t\big)}{\|\nabla V(\mathbf{z})\|^2+\eta}\,\nabla V(\mathbf{z}),
\qquad \eta>0,
\label{eq:HGD_reg}
\end{equation}
which\rev{, for the Exp, FT, and FxT laws,} yields uniformly bounded feedback, and hence forward-complete trajectories (no finite-time escape), without invoking Assumption~\ref{ass:lojasiewicz}; its equilibrium set equals $\mathcal{S}$ exactly, at the cost of a rate slowdown near $\mathcal{S}$ (so the FT/FxT/PT settling time is no longer met, though asymptotic convergence is preserved \rev{on compact sublevel sets of $V$ contained in $\mathcal{U}$}). The bound on $\|\mathbf{u}_\eta\|$, the rate-slowdown formula, and the analogous PT time regularization~\cite{song_prescribed-time_2023} are provided in \iffullversion Appendix~\ref{app:regularized_hgd}\fi.
\end{remark}
\begin{remark}[\rev{Singular and Bounded PT Gains}]
\label{rem:pt_gain}
\rev{The theorems concern the PT gain $\mu/(T-t)$ in~\eqref{eq:sigma_def}, which is unbounded as $t\to T$. Under Assumption~\ref{ass:lojasiewicz} and for $\mathbf{z}(0)\in\Omega_c$ as in Lemma~\ref{lem:fwd_inv}, HGD trajectories nevertheless satisfy}
\begin{equation}
\rev{\|\dot{\mathbf{z}}(t)\|\le(\mu/c_{\mathrm{L}})\,V(\mathbf{z}(0))^{1-\alpha}\,T^{-\mu(1-\alpha)}(T-t)^{\mu(1-\alpha)-1},}
\label{eq:pt_velocity}
\end{equation}
\rev{so the state velocity stays bounded when $\mu\ge 1/(1-\alpha)$, that is, $\mu\ge 2$ under PL; since the trajectories also have finite length, $\mathbf{z}(T)\coloneqq\lim_{t\to T}\mathbf{z}(t)$ exists and lies in $\mathcal{S}$. The bounded gain $\mu/(T+\bar\varepsilon-t)$, $\bar\varepsilon>0$, avoids the singularity at the cost of a terminal residual: under HGD and ND,}
\begin{equation}
\rev{V(\mathbf{z}(T))=V(\mathbf{z}(0))\,\big(\bar\varepsilon/(T+\bar\varepsilon)\big)^{\mu}.}
\label{eq:pt_residual}
\end{equation}
\rev{The derivations are in \iffullversion Appendix~\ref{app:regularized_hgd}\fi.}
\end{remark}

\subsection{Numerical Example: Unconstrained Problem}
\label{subsec:logsumexp_example}

We illustrate the three feedback laws on an $n$-dimensional, smooth, strongly convex model.
We consider the classical log-sum-exp plus quadratic term function~\cite{castera_continuous_2024}, for any $\mathbf{x}\in\mathbb{R}^n$
\begin{equation}
    J(\mathbf{x})\;=\;\log\!\Bigg(\sum_{i=1}^n \big(e^{x_i}+e^{-x_i}\big)\Bigg)\;+\;\frac{1}{2}\,\|\mathbf{x}\|^2,
    \label{eq:logsumexp_cost_nd}
\end{equation}
which has the unique minimizer $\mathbf{x}^\star=\mathbf{0}$.
We compare the Hessian-gradient dynamics (HGD), the Newton dynamics (ND), and the gradient dynamics (GD). 
Each dynamics enforces the same Lyapunov decay \rev{law, $\dot V=-\sigma(V,t)$ under HGD and ND and $\dot V\le-\sigma(V,t)$ under GD}, with $V(\mathbf{x})=\tfrac{1}{2}\|\nabla J(\mathbf{x})\|^2$. 
We test four \rev{decay} laws: exponential, FT, FxT, and PT, resulting in a total of \rev{12} cases.

We run these simulations on a consumer-grade computer equipped with a single 6-core Intel i7 CPU running at 2.6 GHz.
All simulations are carried out in MATLAB using \texttt{ode15s} with 
\texttt{RelTol}$=10^{-9}$ and \texttt{AbsTol}$=10^{-12}$, 
terminating when $\|\nabla J(\mathbf{x}(t))\|\le 10^{-6}$. 
For all simulations, we set $n=50$ and initialize at $\mathbf{x}\rev{(0)=\mathbf{1}}\in\mathbb{R}^n$.
\rev{For PT, the simulations use the bounded gain of Remark~\ref{rem:pt_gain}, here with $T=1$~s, and set $\bar\varepsilon$ before each run so that $\|\mathbf{S}(\mathbf{z}(T))\|\le10^{-6}$; the parameters of all runs are in \iffullversion Appendix~\ref{app:sim_params}\fi.}

Figure~\ref{fig:logsumexp} confirms that all 12 cases realize the decay patterns induced by $\sigma(V,t)$. \rev{Under HGD and ND, which enforce the law with equality, the FT runs settle at $10.15$~s, the bound of Lemma~\ref{lem:finite}, the FxT runs settle at $2.38$~s, within the bound of $3$~s of Lemma~\ref{lem:fixed}, and the PT runs reach $\|\nabla J\|=10^{-6}$ at $T$; GD, which enforces the law as an inequality, settles no later.} Table~\ref{tab:cpu_times} reports the corresponding wall-clock CPU costs (the time axis in Fig.~\ref{fig:logsumexp} corresponds to the continuous dynamics of~\eqref{eq:plant}). \rev{GD and HGD have similar CPU times, and ND, which solves a linear system at every evaluation, is the slowest under all four laws.}
\begin{figure}
  \begin{centering}
    \includegraphics[width=0.9\linewidth]{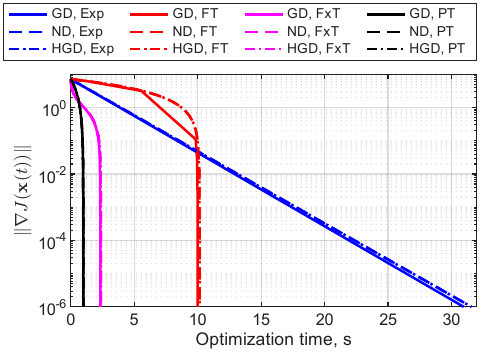}
    \par\end{centering}
    \caption{Gradient norm trajectories $\|\nabla J(\mathbf{x})\|$ for $n=50$ 
    under the exponential (Exp), finite-time (FT), fixed-time (FxT), 
    and prescribed-time (PT) laws, denoted by blue, red, magenta, and black lines, respectively. 
    Three realizations are compared: GD, ND, and HGD, denoted by solid, dashed, and dashed-dotted lines, respectively. 
    The test function is the log-sum-exp plus quadratic model \eqref{eq:logsumexp_cost_nd}, 
    initialized at $\mathbf{x}\rev{(0) = \mathbf{1}}$. \rev{HGD and ND enforce the decay law with equality, so their curves coincide.}}
  \label{fig:logsumexp}
\end{figure}
\begin{table}
  \centering
  \setlength{\tabcolsep}{5pt}
  \caption{Wall-clock CPU times (ms) for the $n=50$ log-sum-exp example
under the four \rev{decay} laws and three feedback realizations\rev{, median of 20 runs}.}
  \label{tab:cpu_times}
  \begin{tabular}{lcccc}
    \hline
    \textbf{Case} & \textbf{Exp} & \textbf{FT} & \textbf{FxT} & \textbf{PT} \\
    \hline
    ND  & \rev{33} & \rev{30} & \rev{58} & \rev{39} \\
    GD  & \rev{19} & \rev{9} & \rev{28} & \rev{21} \\
    HGD & \rev{20} & \rev{11} & \rev{30} & \rev{22} \\
    \hline
  \end{tabular}
\end{table}

All three feedback laws achieve convergence with the selected \rev{decay} law, but differ in their computational profiles. 
The HGD design avoids matrix inversion at the cost of additional multiplications; \rev{ND requires a nonsingular Hessian, and GD requires strong convexity.}
\rev{This example demonstrates} how the OLF framework accommodates different structural assumptions while unifying exponential, FT, FxT, and PT guarantees.

\subsection{On Discrete-Time Implementation}
\label{subsec:discrete_impl}

Our simulations use MATLAB's \texttt{ode15s}, which integrates the continuous-time dynamics without exploiting their Lyapunov structure. More principled discrete-time realizations instead preserve the convergence guarantees by following Lyapunov inequalities rather than a fixed numerical scheme~\cite{garg_fixed-time_2023,hernandez-solano_numerical_2023,guedes_preservation_2025,sanchez_lyapunov-based_2021,ushiyama_unified_2023}; a survey is provided in \iffullversion Appendix~\ref{app:discrete_notes}\fi.

\section{Constrained Optimization}
\label{sec:constrained}
We now extend the control-centric framework to constrained optimization problems that include equality and inequality constraints, covering exponential, FT, FxT, and PT \rev{decay} laws.
We use HGD throughout this and the following sections, since it requires only the OLF regularity of Assumption~\ref{ass:OLF_reg} and applies to any OLF. ND remains available wherever $\nabla\mathbf{S}$ is invertible, but GD does not apply as the symmetric part of $\nabla\mathbf{S}$ is indefinite.

\subsection{Problem Setup and Preliminaries}
Consider the constrained problem
\begin{subequations}
\begin{align}
\min_{\mathbf{x}\in\mathbb{R}^n}~&J(\mathbf{x}) \\ 
\text{s.t.}~& \mathbf{g}(\mathbf{x})\le \mathbf{0}, \label{eq:ineq-constraints}\\
&\mathbf{h}(\mathbf{x})=\mathbf{0}, \label{eq:eq-constraints}
\end{align}
\label{eq:constrained_problem_general}
\end{subequations}
with the associated Lagrangian
\begin{equation}
L(\mathbf{x},\boldsymbol{\lambda},\boldsymbol{\mu})
=
J(\mathbf{x})
+\boldsymbol{\lambda}^\top\mathbf{g}(\mathbf{x})
+\boldsymbol{\mu}^\top\mathbf{h}(\mathbf{x}),
\label{eq:Lagrangian_general}
\end{equation}
where $J:\mathbb{R}^n\to\mathbb{R}$, $\mathbf{g}:\mathbb{R}^n\to\mathbb{R}^p$, $\mathbf{h}:\mathbb{R}^n\to\mathbb{R}^q$, $\boldsymbol{\lambda}\in\mathbb{R}^p_{\ge 0}$, and $\boldsymbol{\mu}\in\mathbb{R}^q$.

\begin{definition}[KKT Conditions]
\label{def:KKT_general}
A point $(\mathbf{x}^\star,\boldsymbol{\lambda}^\star,\boldsymbol{\mu}^\star)$ satisfies the KKT conditions if
\begin{subequations}
\begin{align}
&\nabla J(\mathbf{x}^\star)+\nabla\mathbf{g}(\mathbf{x}^\star)^\top\boldsymbol{\lambda}^\star+\nabla\mathbf{h}(\mathbf{x}^\star)^\top\boldsymbol{\mu}^\star=\mathbf{0}, \label{eq:KKT_stationarity_general}\\
&\mathbf{g}(\mathbf{x}^\star)\le \mathbf{0},~~\boldsymbol{\lambda}^\star\ge \mathbf{0},~~\langle\boldsymbol{\lambda}^\star,\mathbf{g}(\mathbf{x}^\star)\rangle=0, \label{eq:KKT_ineq_general}\\
&\mathbf{h}(\mathbf{x}^\star)=\mathbf{0}. \label{eq:KKT_eq_general}
\end{align}
\end{subequations}
\end{definition}

\subsection{Fischer--Burmeister Function and Its Usage for Stationarity}

We encode KKT conditions associated with the inequality constraints in~\eqref{eq:KKT_ineq_general} via the Fischer--Burmeister (FB) function~\cite{fischer_special_1992}:
\begin{equation}
\phi(a,b)\coloneqq\sqrt{a^2+b^2}-(a+b).
\label{eq:FB_def_gen}
\end{equation}
Note that $\phi(a,b)=0$ if and only if $a\ge 0$, $b\ge 0$, and $ab=0$~\cite{fischer_special_1992}. To encode the KKT conditions in~\eqref{eq:KKT_ineq_general}, which require $\boldsymbol{\lambda}\ge \mathbf{0}$, $\mathbf{g}(\mathbf{x})\le \mathbf{0}$, and $\boldsymbol{\lambda}^\top \mathbf{g}(\mathbf{x})=0$, we apply $\phi$ to the pair $(\boldsymbol{\lambda},-\mathbf{g}(\mathbf{x}))$ elementwise. Then $\boldsymbol{\phi}\!\big(\boldsymbol{\lambda},-\mathbf{g}(\mathbf{x})\big)=\mathbf{0}$ if and only if $\boldsymbol{\lambda}\ge \mathbf{0}$, $\mathbf{g}(\mathbf{x})\le \mathbf{0}$, and $\boldsymbol{\lambda}^\top \mathbf{g}(\mathbf{x})=0$. 

The FB function is nonsmooth near the origin, so we use its smoothed version~\cite{liao-mcpherson_regularized_2019}
\begin{equation}
\phi_{\varepsilon}(a,b)\coloneqq\sqrt{a^2+b^2+\varepsilon^2}-(a+b),
\quad \varepsilon>0,
\label{eq:FB_smooth_gen}
\end{equation}
where the smoothing satisfies
\begin{equation}
\left|\phi_{\varepsilon}(a,b) - \phi(a,b)\right|\le \varepsilon,
\end{equation}
for any $\varepsilon > 0$.
The smoothed FB mapping enforces the KKT conditions associated with the inequality constraints up to $\varepsilon > 0$. One may select $\varepsilon$ arbitrarily small to approach the exact KKT conditions while preserving differentiability and good numerical properties~\cite{liao-mcpherson_regularized_2019}. 
A complementary projection-free approach that instead renders the feasible set forward invariant via control barrier functions, yielding anytime feasibility under asymptotic convergence, is developed in~\cite{allibhoy_control-barrier-function-based_2024}.

To ensure that the modified Lagrangian Hessian $\mathbf{H}_{\mathcal{L}} \coloneqq \nabla^2 J + \sum_i \widetilde{\lambda}_i \nabla^2 g_i$ remains positive definite on $\mathcal{U}$, and not only where $\boldsymbol{\lambda}\ge\mathbf{0}$, we replace $\boldsymbol{\lambda}$ in the Lagrangian gradient block of the stationarity \rev{mapping} with
\begin{equation}
\widetilde{\lambda}_i \;\coloneqq\; \tfrac{1}{2}\!\Big(\!\sqrt{\lambda_i^2+\varepsilon^2}+\lambda_i\Big),
\qquad i=1,\dots,p,
\label{eq:lambda_tilde}
\end{equation}
a smooth approximation of the positive-part operator $\max(\lambda_i,0) = \tfrac{1}{2}(|\lambda_i|+\lambda_i)$, obtained by smoothing the absolute value as $|\lambda_i|\approx\sqrt{\lambda_i^2+\varepsilon^2}$. By construction $\widetilde{\lambda}_i>0$ for all $\lambda_i\in\mathbb{R}$, and for $\lambda_i\ge 0$ the map $\lambda_i\mapsto\widetilde{\lambda}_i-\lambda_i=\tfrac{1}{2}(\sqrt{\lambda_i^2+\varepsilon^2}-\lambda_i)$ is decreasing, so $0<\widetilde{\lambda}_i-\lambda_i\le\varepsilon/2$. The inequality block retains the original~$\boldsymbol{\lambda}$. \rev{The FB smoothing and the $\widetilde{\lambda}$ modification share the parameter $\varepsilon$ and enter different blocks of $\mathbf{S}$.}

Define the stacked variable $\mathbf{z}\coloneqq\operatorname{col}(\mathbf{x},\boldsymbol{\lambda},\boldsymbol{\mu})\in\mathbb{R}^{n+p+q}$ and the stationarity \rev{mapping}
\begin{equation}
\mathbf{S}(\mathbf{z})
=
\begin{bmatrix}
\nabla J(\mathbf{x})+\nabla\mathbf{g}(\mathbf{x})^\top\widetilde{\boldsymbol{\lambda}}+\nabla\mathbf{h}(\mathbf{x})^\top\boldsymbol{\mu}\\[2pt]
\boldsymbol{\phi}_{\varepsilon}\!\big(\boldsymbol{\lambda},-\mathbf{g}(\mathbf{x})\big)\\[2pt]
\mathbf{h}(\mathbf{x})
\end{bmatrix}.
\label{eq:S_FB}
\end{equation}
Here $\widetilde{\boldsymbol{\lambda}}=\operatorname{col}(\widetilde{\lambda}_1,\dots,\widetilde{\lambda}_p)$ with each $\widetilde{\lambda}_i$ defined in~\eqref{eq:lambda_tilde}; the inequality block retains the original $\boldsymbol{\lambda}$.
Then $\mathbf{S}(\mathbf{z})=\mathbf{0}$ enforces 
\eqref{eq:KKT_stationarity_general}--\eqref{eq:KKT_eq_general} up to the 
smoothing tolerance $\varepsilon$.

\rev{A} zero of $\mathbf{S}$ is \rev{a} smoothed stationarity point $\mathbf{z}^\star_\varepsilon$\rev{; under the hypotheses of Lemma~\ref{lem:smoothed_point} below it is unique, and the lemma bounds its distance to} the exact KKT point $\mathbf{z}^\star=(\mathbf{x}^\star,\boldsymbol{\lambda}^\star,\boldsymbol{\mu}^\star)$.

\begin{remark}[Alternative Stationarity Without Smoothing]
\label{rem:alt_stat}
If \rev{the exact KKT conditions are} required, one may use the alternative stationarity \rev{mapping}
\begin{equation}
\rev{\mathbf{S}_0}(\mathbf{z})=
\operatorname{col}\left(
\nabla_{\mathbf{x}} L(\mathbf{x},\boldsymbol{\lambda},\boldsymbol{\mu}),
\boldsymbol{\lambda}^\top \mathbf{g}(\mathbf{x}),
\mathbf{g}_+(\mathbf{x}),
\boldsymbol{\lambda}_-,
\mathbf{h}(\mathbf{x})
\right)
\label{eq:S_exact}
\end{equation}
with $\mathbf{g}_+(\mathbf{x})=\max(\mathbf{g}(\mathbf{x}),\mathbf{0})$ and $\boldsymbol{\lambda}_-=\max(-\boldsymbol{\lambda},\mathbf{0})$ applied componentwise.
In this case, $\rev{\mathbf{S}_0}(\mathbf{z})$ involves max operators and is therefore nonsmooth. However, the quadratic OLF $V(\mathbf{z}) = \tfrac{1}{2}\|\rev{\mathbf{S}_0}(\mathbf{z})\|^2$ remains continuously differentiable, since the max terms enter squared and $(\max(0,\cdot))^2$ is $\mathcal{C}^1$. 
\end{remark}

We also include the following classical regularity assumptions for constrained optimization~\cite{beck_introduction_2014}.
\begin{assumption}[LICQ]
\label{ass:LICQ_general}
Let $\mathcal{A}(\mathbf{x}) \coloneqq \{\, i \in \{1,\ldots,p\}\mid g_i(\mathbf{x}) = 0 \,\}$ denote the active set. We say the linear independence constraint qualification (LICQ) holds at $\mathbf{x}$ if the set
\begin{equation}
\Big\{\,\nabla h_j(\mathbf{x})~:~ j=1,\ldots,q \,\Big\}\cup \Big\{\, \nabla g_i(\mathbf{x})~:~ i\in\mathcal{A}(\mathbf{x}) \,\Big\}
\end{equation}
is linearly independent.
\end{assumption}

\begin{assumption}[Generalized Slater]
\label{ass:Slater_general}
For the inequality constraints $\mathbf{g}(\mathbf{x})\le \mathbf{0}$, generalized Slater's condition holds if there exists $\bar{\mathbf{x}}$ such that $\mathbf{h}(\bar{\mathbf{x}})=\mathbf{0}$ and $\mathbf{g}(\bar{\mathbf{x}})< \mathbf{0}$.
\end{assumption}
\begin{lemma}[\rev{Smoothed Stationarity Point}]
\label{lem:smoothed_point}
\rev{Let the $p$ inequality constraints $\mathbf{g}\in\mathcal{C}^1$ be convex, let $\mathbf{h}(\mathbf{x})=\mathbf{A}\mathbf{x}-\mathbf{b}$ with $\mathbf{A}$ of full row rank, let Assumption~\ref{ass:Slater_general} hold, and let $\nabla J$ be continuous and strongly monotone: for some $m>0$ and all $\mathbf{x},\mathbf{x}'$,}
\begin{equation}
\rev{\langle\nabla J(\mathbf{x})-\nabla J(\mathbf{x}'),\mathbf{x}-\mathbf{x}'\rangle\ge m\|\mathbf{x}-\mathbf{x}'\|^2.}
\label{eq:strong_mono_kerA}
\end{equation}
\rev{Then~\eqref{eq:constrained_problem_general} has a unique minimizer $\mathbf{x}^\star$, and for every $\varepsilon>0$: (i) $\mathbf{S}$ has a unique zero $\mathbf{z}^\star_\varepsilon=(\mathbf{x}^\star_\varepsilon,\boldsymbol{\lambda}^\star_\varepsilon,\boldsymbol{\mu}^\star_\varepsilon)$; (ii) $\mathbf{g}(\mathbf{x}^\star_\varepsilon)<\mathbf{0}$ and $\boldsymbol{\lambda}^\star_\varepsilon>\mathbf{0}$; (iii) with the spectral norm of $\nabla\mathbf{g}$,}
\begin{equation}
\rev{\|\mathbf{x}^\star_\varepsilon-\mathbf{x}^\star\|\le\varepsilon\sqrt{p}\,\big(\|\nabla\mathbf{g}(\mathbf{x}^\star_\varepsilon)\|/(2m)+1/\sqrt{2m}\big);}
\label{eq:eps_bound_iii}
\end{equation}
\rev{(iv) if, in addition, $\nabla J\in\mathcal{C}^1$, $\mathbf{g}\in\mathcal{C}^2$, and Assumption~\ref{ass:LICQ_general} holds at $\mathbf{x}^\star$, then the KKT point $\mathbf{z}^\star$ is unique and there exist $\varepsilon_0,\kappa>0$ such that $\|\mathbf{z}^\star_\varepsilon-\mathbf{z}^\star\|\le\kappa\,\varepsilon$ for all $\varepsilon\in(0,\varepsilon_0]$.}
\end{lemma}
\rev{In (iv), $\varepsilon_0$ and $\kappa$ depend only on $m$, $p$, and the problem data at and near $\mathbf{x}^\star$; their expressions and the proof are given in \iffullversion Appendix~\ref{app:eps_displacement}, and Appendix~\ref{app:eps_sweep} checks the rate numerically\fi.}
\subsection{Results}
We implement the Hessian-gradient dynamics (HGD):
\begin{equation}
\dot{\mathbf{z}} =
\mathbf{u}(\mathbf{z},t) = 
-\frac{ \nabla \mathbf{S}(\mathbf{z})^\top \mathbf{S}(\mathbf{z}) }{\left\|\nabla \mathbf{S}(\mathbf{z})^\top \mathbf{S}(\mathbf{z})\right\|^2}\;
\sigma\!\big(V(\mathbf{z}),t\big),
\label{eq:HGD_general}
\end{equation}
where $\sigma(\cdot,\cdot)$ is the \rev{decay law}. 
This feedback yields
\begin{equation}
\dot V = -\,\sigma\!\big(V(\mathbf{z}),t\big),
\label{eq:Vdot_scalar_template}
\end{equation}
on the domain where $\nabla \mathbf{S}(\mathbf{z})^\top \mathbf{S}(\mathbf{z})\neq \mathbf{0}$.

Under Assumption~\ref{ass:OLF_reg}, $\nabla\mathbf{S}(\mathbf{z})^\top\mathbf{S}(\mathbf{z})\neq\mathbf{0}$ for every $\mathbf{z}\in\mathcal{U}$ with $\mathbf{S}(\mathbf{z})\neq\mathbf{0}$, so~\eqref{eq:HGD_general} is well-defined on $\mathcal{U}\setminus\mathcal{S}$. The scalar coefficient in~\eqref{eq:HGD_general} may grow unbounded as $\mathbf{z}\to\mathcal{S}$; this non-Lipschitz behavior at the \rev{stationarity} set is necessary for convergence in finite time~\cite{bhat_finite-time_2000}, and \rev{Remarks~\ref{rem:regularizers} and~\ref{rem:pt_gain} provide} a bounded feedback when required.

We now present the theorem for the constrained case, which covers both equality and inequality constraints using the HGD law. 
It recovers the desired convergence by using~\eqref{eq:sigma_def}. 

\begin{theorem}[HGD for Constrained Optimization]
\label{thm:unified_constrained}
Consider~\eqref{eq:constrained_problem_general} with $J,\mathbf{g},\mathbf{h}\in\mathcal{C}^2$, and define $\mathbf{S}$ via~\eqref{eq:S_FB}. 
Fix $\varepsilon>0$ and let $(\mathbf{x}^\star,\boldsymbol{\lambda}^\star,\boldsymbol{\mu}^\star)$ be a KKT point at which Assumption~\ref{ass:LICQ_general} holds. Suppose Assumptions~\ref{ass:OLF_reg} and~\ref{ass:lojasiewicz} hold on an open neighborhood $\mathcal{U}$ of $(\mathbf{x}^\star,\boldsymbol{\lambda}^\star,\boldsymbol{\mu}^\star)$ with $\mathcal{S}\cap\mathcal{U}=\{\mathbf{z}^\star_\varepsilon\}$, and let $\Omega_c\subseteq\mathcal{U}$ be a forward-invariant sublevel set as constructed in Lemma~\ref{lem:fwd_inv}. Then for every $\mathbf{z}(0)\in\Omega_c$, the trajectory $\mathbf{z}(t)$ of~\eqref{eq:HGD_general} satisfies $\mathbf{z}(t)\in\Omega_c$ \rev{until it reaches $\mathcal{S}$}, and 
the KKT conditions are enforced up to the smoothing tolerance $\varepsilon$ with the \rev{decay} law $\sigma(\cdot,\cdot)$ selected in~\eqref{eq:sigma_def}.
\end{theorem}
\begin{proof}[Proof sketch]
Assumption~\ref{ass:LICQ_general} ensures the multipliers $(\boldsymbol{\lambda}^\star,\boldsymbol{\mu}^\star)$ are unique~\cite[Thm.~11.12]{beck_introduction_2014}; the hypothesis $\mathcal{S}\cap\mathcal{U}=\{\mathbf{z}^\star_\varepsilon\}$ makes $\mathcal{S}\cap\mathcal{U}$ a bounded singleton, so Lemma~\ref{lem:fwd_inv} yields a forward-invariant $\Omega_c\subseteq\mathcal{U}$ on which~\eqref{eq:HGD_general} is well-defined, and the chain rule gives $\dot V=-\sigma(V,t)$; since $V=\tfrac12\|\mathbf{S}\|^2$, this drives $\mathbf{S}(\mathbf{z}(t))\to\mathbf{0}$ with the indicated timing, enforcing the KKT conditions up to the smoothing tolerance $\varepsilon$. The full argument, including well-definedness of~\eqref{eq:HGD_general} along the trajectory and the chain-rule computation, is given in \iffullversion Appendix~\ref{app:thm_constrained_proof}\fi.
\end{proof}

Theorem~\ref{thm:unified_constrained} delivers the full timing guarantee for any OLF meeting Assumptions~\ref{ass:OLF_reg} and~\ref{ass:lojasiewicz} on a forward-invariant $\Omega_c$. Under stronger structural hypotheses, these conditions hold automatically and the solution is unique, with a quantitative basin obtained from compactness of the sublevel sets of $V$.

\begin{proposition}[Sublevel-Set Compactness for Constrained Problems]
\label{prop:c_star}
Consider~\eqref{eq:constrained_problem_general} with $\mathbf{S}$ defined by~\eqref{eq:S_FB}, and assume $J$ is strongly convex, $\mathbf{h}(\mathbf{x})=\mathbf{A}\mathbf{x}-\mathbf{b}$ with $\mathbf{A}$ full row rank, each $g_j$ is convex, Slater's condition holds (there exists $\bar{\mathbf{x}}$ with $\mathbf{A}\bar{\mathbf{x}}=\mathbf{b}$ and $\mathbf{g}(\bar{\mathbf{x}})<\mathbf{0}$), and the feasible set $\{\mathbf{x}:\mathbf{A}\mathbf{x}=\mathbf{b},\,\mathbf{g}(\mathbf{x})\le\mathbf{0}\}$ is bounded. Then there exists $c^\ast>0$ such that for every $c\in(0,c^\ast)$ the sublevel set $\Omega_c=\{\mathbf{z}\in\mathbb{R}^{n+p+q}:V(\mathbf{z})\le c\}$ is compact.
\end{proposition}

\begin{proof}[Proof sketch]
Closedness of $\Omega_c$ is immediate from continuity of $V$. Boundedness is a coercivity property of $V$, established by ruling out escape to infinity in $\mathbf{z}=(\mathbf{x},\boldsymbol{\lambda},\boldsymbol{\mu})$ on $\{V\le c\}$; the dominant mechanisms are: (A) $\|\mathbf{x}\|\to\infty$ is ruled out by strong convexity of $J$, which makes the Lagrangian-gradient block of $\mathbf{S}$ coercive in $\mathbf{x}$; (B) $\|\boldsymbol{\mu}\|\to\infty$ is ruled out by full row rank of $\mathbf{A}$, which forces $\|\mathbf{A}^\top\boldsymbol{\mu}\|\to\infty$ in the same block; (C) $\|\boldsymbol{\lambda}\|\to\infty$: projecting the Lagrangian-gradient block onto $\ker(\mathbf{A})$ removes the $\mathbf{A}^\top\boldsymbol{\mu}$ term irrespective of $\|\boldsymbol{\mu}\|$, so the diverging multipliers force a convex combination of the active constraint gradients to lie in $\operatorname{range}(\mathbf{A}^\top)$ at the limit point; Slater's condition together with a Cauchy--Schwarz step then bounds $\tfrac{1}{2}\|\mathbf{S}\|^2$ below by a strictly positive constant along the ray, so the escape is excluded for every $c$ below that constant. The full coercivity argument, including the joint-escape modes and the use of boundedness of the feasible set, is provided in \iffullversion Appendix~\ref{app:c_star_proof}\fi.
\end{proof}

The condition $V(\mathbf{z}(0))<c^\ast$ restricts only the initial residual $\|\mathbf{S}(\mathbf{z}(0))\|$; in particular, $\mathbf{x}(0)$ does not have to satisfy the constraints of~\eqref{eq:constrained_problem_general}.

\begin{proposition}[Global Convergence Under Affine Constraints]
\label{prop:affine_global}
Consider~\eqref{eq:constrained_problem_general} with $\mathbf{S}$ defined by~\eqref{eq:S_FB}, $J$ strongly convex \rev{with Lipschitz gradient}, $\mathbf{h}(\mathbf{x})=\mathbf{A}\mathbf{x}-\mathbf{b}$, and affine inequalities $\mathbf{g}(\mathbf{x})=\mathbf{C}\mathbf{x}-\mathbf{d}$. If $\big[\begin{smallmatrix}\mathbf{A}\\ \mathbf{C}\end{smallmatrix}\big]$ has full row rank, then $\sigma_{\min}(\nabla\mathbf{S})\ge\bar\sigma>0$ on all of $\mathbb{R}^{n+p+q}$, so the global PL inequality $\|\nabla V\|^2\ge 2\bar\sigma^2 V$ holds and the closed-loop trajectory~\eqref{eq:HGD_general} converges globally, irrespective of $c^\ast$.
\end{proposition}
\begin{proof}[Proof sketch]
With affine $\mathbf{g}$ the modified Lagrangian Hessian reduces to $\mathbf{H}_{\mathcal{L}}=\nabla^2 J\succeq m\mathbf{I}$\rev{, bounded above because $\nabla J$ is Lipschitz}, and full row rank of $\big[\begin{smallmatrix}\mathbf{A}\\ \mathbf{C}\end{smallmatrix}\big]$ rules out the $\boldsymbol{\lambda}$-escape mode, so $\sigma_{\min}(\nabla\mathbf{S})\ge\bar\sigma>0$ uniformly on $\mathbb{R}^{n+p+q}$. The PL inequality $\|\nabla V\|^2=\|\nabla\mathbf{S}^\top\mathbf{S}\|^2\ge\bar\sigma^2\|\mathbf{S}\|^2=2\bar\sigma^2 V$ then holds globally, \rev{so HGD trajectories have finite length (Sec.~\ref{subsec:regularity}) and converge to a zero of $\mathbf{S}$, which is unique by Lemma~\ref{lem:smoothed_point}(i), since full row rank gives Slater's condition}. The escape-mode analysis is detailed in \iffullversion Appendix~\ref{app:coro_proof}\fi.
\end{proof}

\begin{corollary}
\label{cor:auto_nonsing_general}
Let the hypotheses of Prop.~\ref{prop:c_star} hold with $J,\mathbf{g},\mathbf{h}\in\mathcal{C}^2$, fix $\varepsilon>0$, and let additionally Assumption~\ref{ass:LICQ_general} hold \rev{at the minimizer $\mathbf{x}^\star$}. Then $\mathbf{x}^\star$ is the unique global minimizer with unique KKT multipliers $(\boldsymbol{\lambda}^\star,\boldsymbol{\mu}^\star)$. Moreover, there exists $c^\ast>0$ such that every trajectory~\eqref{eq:HGD_general} from $V(\mathbf{z}(0))<c^\ast$ converges with the \rev{chosen decay} law to the smoothed stationarity point $\mathbf{z}^\star_\varepsilon$\rev{, which satisfies $\|\mathbf{z}^\star_\varepsilon-\mathbf{z}^\star\|\le\kappa\,\varepsilon$ whenever $\varepsilon\le\varepsilon_0$ (Lemma~\ref{lem:smoothed_point}(iv))}.
\end{corollary}
\begin{proof}[Proof sketch]
Uniqueness follows from strong convexity, Slater, and LICQ~\cite[Thm.~11.12]{beck_introduction_2014}. A block-elimination argument on the smoothed-FB Jacobian~\cite{fischer_special_1992,liao-mcpherson_regularized_2019} makes $\nabla\mathbf{S}$ nonsingular, so on the compact $\Omega_c$ of Prop.~\ref{prop:c_star} ($c\in(0,c^\ast)$) one has $\|\nabla V\|^2\ge 2\sigma_c^2 V$ with $\sigma_c=\min_{\Omega_c}\sigma_{\min}(\nabla\mathbf{S})>0$, i.e., Assumption~\ref{ass:lojasiewicz} with $\alpha=1/2$. Lemma~\ref{lem:fwd_inv} then yields convergence. The block-elimination computation is detailed in \iffullversion Appendix~\ref{app:coro_proof}\fi.
\end{proof}

\subsection{Illustrative Example: Network Utility Maximization}
\label{subsec:NUM}

Consider the classical network utility maximization (NUM) problem, widely used in the study of continuous-time primal-dual dynamics~\cite{feijer_stability_2010}.
Let $\mathbf{x}\in\mathbb{R}_{\ge 0}^{\rev{n_s}}$ denote the vector of source rates, $\mathbf{R}\in\{0,1\}^{\rev{n_l\times n_s}}$ the routing matrix (with $R_{ij}=1$ if source $j$ uses link $i$), and $\mathbf{c}\in\mathbb{R}^{\rev{n_l}}_{>0}$ the vector of link capacities. Each source $j$ has a strictly concave, \rev{twice} continuously differentiable utility $U_j:\mathbb{R}_{>0}\to\mathbb{R}$ associated with it. The NUM problem reads
\begin{equation}
\label{eq:NUM}
\max_{\mathbf{x}\,>\,\mathbf{0}} \;\; \sum_{j=1}^{\rev{n_s}} U_j(x_j)
\quad \text{s.t.} \quad \mathbf{R}\mathbf{x} \,\le\, \mathbf{c}.
\end{equation}
Under standard assumptions (strict concavity of $U_j$, \rev{every source using at least one link,} and Slater’s condition),~\eqref{eq:NUM} has a unique primal optimum.
We set $U_{j} = \rev{\beta_j} \log(x_j)$, a strictly concave function for $x_j > 0$, where $\rev{\beta_j} > 0$.
\rev{We take $n_s=50$ sources and $n_l=20$ links and a pseudo-random routing matrix with full row rank; at the optimum, 18 of the 20 links are saturated. The data are in \iffullversion Appendix~\ref{app:sim_params}\fi.}

Problem~\eqref{eq:NUM} fits the constrained setup of~\eqref{eq:constrained_problem_general} with the strictly (but not strongly) concave utilities $U_j=\rev{\beta_j}\log(x_j)$, so the example falls outside the strong-convexity hypothesis of Cor.~\ref{cor:auto_nonsing_general} and we invoke instead the conditional template of Thm.~\ref{thm:unified_constrained}. Assumptions~\ref{ass:OLF_reg} and~\ref{ass:lojasiewicz} are verified numerically along the simulated trajectories rather than guaranteed a priori; this is the practical regime in which Thm.~\ref{thm:unified_constrained} is meant to be applied. \rev{Along every run below, the smallest singular value of $\nabla\mathbf{S}$ stays above $0.21$, which gives Assumption~\ref{ass:OLF_reg} and, since $\|\nabla V\|^2\ge 2\,\sigma_{\min}(\nabla\mathbf{S})^2\,V$, Assumption~\ref{ass:lojasiewicz} with $\alpha=1/2$ along the trajectories.}
We use the stationarity \rev{mapping} from~\eqref{eq:S_FB} with $\varepsilon = 10^{-6}$ for the smoothed FB.

Classical primal-dual dynamics for~\eqref{eq:NUM} exhibit asymptotic convergence under strict concavity~\cite{feijer_stability_2010}\rev{, and FT and FxT variants exist for linear constraints~\cite{ozaslan_exponential_2024,shi_finite-time_2023}}.
\rev{Within our framework, a single design yields all four regimes, including PT, under the conditions verified above.}

\rev{We run HGD and ND from $\mathbf{x}(0)=0.2\cdot\mathbf{1}$ and $\boldsymbol{\lambda}(0)=\mathbf{0}$, where $V(\mathbf{z}(0))=1168.6$, with $T=10$~s. Figure~\ref{fig:NUM_V_decay} shows that the two realizations produce the same decay, since both enforce the law with equality. The Exp runs reach $\|\mathbf{S}\|=10^{-6}$ at $35.4$~s, the FT runs settle at $68.37$~s, the bound of Lemma~\ref{lem:finite}, the FxT runs at $2.42$~s, below the bound of Lemma~\ref{lem:fixed}, and the PT runs reach $\|\mathbf{S}\|=10^{-6}$ at $T$. All runs end within $3\times10^{-6}$ of the exact KKT point. The wall-clock CPU times\iffullversion{} (Table~\ref{tab:NUM_cpu})\fi{} range from 297 to 375~ms for HGD and from 705 to 753~ms for ND, which solves a linear system with $\nabla\mathbf{S}$ at every evaluation.}

\begin{figure}
  \centering
  \includegraphics[width=0.9\linewidth]{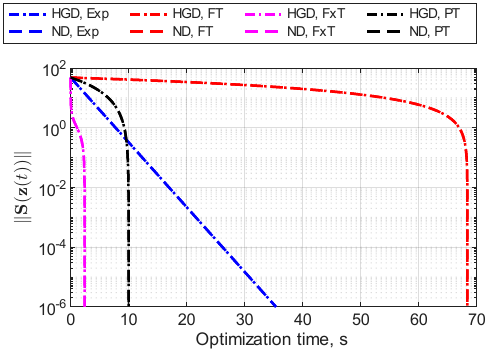}
    \caption{\rev{Norm of the stationarity mapping $\|\mathbf{S}(\mathbf{z}(t))\|$ for the NUM problem under the four decay laws (colors as in Fig.~\ref{fig:logsumexp}), for HGD (dash-dotted) and ND (dashed), whose curves coincide.}} 
  \label{fig:NUM_V_decay}
\end{figure}

\iffullversion
\begin{table}
  \centering
  \caption{Wall-clock CPU time (ms) for the NUM problem\rev{, median of 20 runs}.}
  \label{tab:NUM_cpu}
  \begin{tabular}{lcccc}
    \hline \rev{\textbf{Case}} & \textbf{Exp} & \textbf{FT} & \textbf{FxT} & \textbf{PT} \\
    \hline \rev{HGD} & \rev{297} & \rev{352} & \rev{375} & \rev{347} \\
    \rev{ND} & \rev{705} & \rev{719} & \rev{753} & \rev{722} \\
    \hline
  \end{tabular}
\end{table}
\fi

\begin{remark}[Stationarity Monitoring]
\label{rem:stat_mon}
The stationarity \rev{mapping} $\mathbf{S}(\mathbf{z})$ in~\eqref{eq:S_FB} stacks the modified Lagrangian gradient (with smoothed multipliers $\widetilde{\boldsymbol{\lambda}}$) and the smoothed FB:
\begin{equation}
\mathbf{S}(\mathbf{z}) = 
\begin{bmatrix}
\nabla J(\mathbf{x})+\mathbf{R}^\top\widetilde{\boldsymbol{\lambda}} \\[2pt]
\boldsymbol{\phi}_\varepsilon(\boldsymbol{\lambda}, \mathbf{c} - \mathbf{R}\mathbf{x})
\end{bmatrix}.
\label{eq:stacked_S}
\end{equation}
Since the norm of a stacked vector dominates that of each block,
\begin{equation}
\|\mathbf{S}(\mathbf{z})\|
\;\ge\;
\max\!\Big\{\,\|\nabla J(\mathbf{x})+\mathbf{R}^\top\widetilde{\boldsymbol{\lambda}}\|,\;
\|\boldsymbol{\phi}_\varepsilon(\boldsymbol{\lambda}, \mathbf{c} - \mathbf{R}\mathbf{x})\|\,\Big\}.
\label{eq:stat_mon_bound}
\end{equation}
Hence, driving $V(\mathbf{z}) = \tfrac{1}{2}\|\mathbf{S}(\mathbf{z})\|^2$ below a threshold guarantees that each optimality component is satisfied to the same numerical accuracy, so a single scalar Lyapunov function monitors convergence of all optimality conditions.
\end{remark}


\section{Constrained Minimax Problems}
\label{sec:minmax}
Constrained minimax problems are central in robust control, adversarial machine learning, and game-theoretic models of resource allocation, where two agents with conflicting objectives interact under shared feasibility constraints. 
Classical saddle-point dynamics~\cite{cherukuri_saddle-point_2017,cherukuri_role_2018} establish asymptotic convergence for convex-concave problems, but only in either the unconstrained minimax setting or the primal-dual reformulation of constrained optimization, where the saddle structure is on $(\mathbf{x},\boldsymbol{\lambda},\boldsymbol{\mu})$, not on a pair of decision variables $(\mathbf{x},\mathbf{y})$ subject to shared constraints.
More recent works demonstrate FxT \rev{and PT} convergence in unconstrained minimax settings~\cite{garg_fixed-time_2022}\rev{, \cite{song_predefined_2026}}.
\rev{To our knowledge, no projection-free FxT or PT method has been documented for constrained minimax with shared constraints.}
This section addresses this gap using our control-centric framework with the HGD realization.

\subsection{Problem Statement and Existence}
We consider the constrained minimax problem
\begin{subequations}\label{eq:minimax}
\begin{align}
\min_{\mathbf{x}\in\mathbb{R}^{n_x}}\ \max_{\mathbf{y}\in\mathbb{R}^{n_y}}\quad & J(\mathbf{x},\mathbf{y}) \\[-2pt]
\text{s.t.}\qquad 
& \mathbf{A}\,\mathbf{x} + \mathbf{B}\,\mathbf{y} - \mathbf{b} = \mathbf{0}, \label{eq:minimax_eq}\\[-2pt]
& \mathbf{G}(\mathbf{x},\mathbf{y}) \le \mathbf{0}, \label{eq:minimax_ineq}
\end{align}
\end{subequations}
where $J(\cdot,\mathbf{y})$ is convex in $\mathbf{x}$ for each fixed $\mathbf{y}$, and $J(\mathbf{x},\cdot)$ is concave in $\mathbf{y}$ for each fixed $\mathbf{x}$. 
The matrices $\mathbf{A}\in\mathbb{R}^{q\times n_x}$ and $\mathbf{B}\in\mathbb{R}^{q\times n_y}$ define $q$ affine equality constraints, and $\mathbf{G}:\mathbb{R}^{n_x+n_y}\to\mathbb{R}^p$ collects $p$ convex inequality constraints \rev{of class $\mathcal{C}^2$}. We assume the stacked equality Jacobian $[\mathbf{A}\ \mathbf{B}]\in\mathbb{R}^{q\times(n_x+n_y)}$ has full row rank $q$; this is a direct structural assumption, not a consequence of LICQ at the solution.

We impose the regularity condition of Assumption~\ref{ass:Slater_general}.
For convex-concave saddle problems on product compact convex domains, classical existence and the equality $\min\max=\max\min$ are given by Sion's theorem~\cite{sion_general_1958}. Problem~\eqref{eq:minimax}, however, has a coupled feasible set \rev{$\mathcal{W}$, on which the order of minimization and maximization matters: the sequential problem $\min_{\mathbf{x}}\max_{\mathbf{y}:(\mathbf{x},\mathbf{y})\in\mathcal{W}}J$ can violate the max-min inequality and is NP-hard even for strongly convex-strongly concave $J$~\cite{tsaknakis_minimax_2023}. We therefore adopt the variational notion: $(\mathbf{x}^\star,\mathbf{y}^\star)$ is a saddle point of~\eqref{eq:minimax} if it solves} the \rev{variational inequality (VI)} on \rev{$\mathcal{W}$} with operator $\mathcal{F}(\mathbf{w})\coloneqq\operatorname{col}(\nabla_\mathbf{x} J,-\nabla_\mathbf{y} J)$, $\mathbf{w}\coloneqq\operatorname{col}(\mathbf{x},\mathbf{y})$, which is monotone under convex-concavity of $J$ and strongly monotone under strong convex-concavity. \rev{Such a point satisfies $J(\mathbf{x}^\star,\mathbf{y})\le J(\mathbf{x}^\star,\mathbf{y}^\star)\le J(\mathbf{x},\mathbf{y}^\star)$ whenever $(\mathbf{x}^\star,\mathbf{y}),(\mathbf{x},\mathbf{y}^\star)\in\mathcal{W}$, and it is the classical saddle point when the constraints do not couple $\mathbf{x}$ and $\mathbf{y}$.} Existence of a saddle point on the compact convex feasible set then follows from VI solvability~\cite[Cor.~2.2.5]{facchinei_finite-dimensional_2003}.
We also use the OLF regularity condition stated in Assumption~\ref{ass:OLF_reg}.

The KKT conditions for~\eqref{eq:minimax} introduce multipliers 
$\boldsymbol{\mu}\in\mathbb{R}^q$ for the equalities~\eqref{eq:minimax_eq} 
and $\boldsymbol{\lambda}\in\mathbb{R}^p_{\ge 0}$ for the inequalities~\eqref{eq:minimax_ineq}. 
The KKT system reads
\begin{subequations}\label{eq:KKT_minimax}
\begin{align}
&\nabla_{\mathbf{x}} J(\mathbf{x}^\star,\mathbf{y}^\star) 
   + \mathbf{A}^\top \boldsymbol{\mu}^\star 
   + \nabla_{\mathbf{x}} \mathbf{G}(\mathbf{x}^\star,\mathbf{y}^\star)^\top \boldsymbol{\lambda}^\star = \mathbf{0}, \label{eq:KKT_x}\\
-&\nabla_{\mathbf{y}} J(\mathbf{x}^\star,\mathbf{y}^\star) 
   + \mathbf{B}^\top \boldsymbol{\mu}^\star 
   + \nabla_{\mathbf{y}} \mathbf{G}(\mathbf{x}^\star,\mathbf{y}^\star)^\top \boldsymbol{\lambda}^\star = \mathbf{0}, \label{eq:KKT_y}\\
& \mathbf{G}(\mathbf{x}^\star,\mathbf{y}^\star) \le \mathbf{0},\quad 
\boldsymbol{\lambda}^\star \ge \mathbf{0},\quad 
\langle \boldsymbol{\lambda}^\star, \mathbf{G}  (\mathbf{x}^\star,\mathbf{y}^\star)\rangle = 0, \label{eq:KKT_comp}\\
& \mathbf{A} \mathbf{x}^\star + \mathbf{B} \mathbf{y}^\star - \mathbf{b} = \mathbf{0}. \label{eq:KKT_eq}
\end{align}
\end{subequations}

Let $\mathbf{z}\coloneqq\operatorname{col}(\mathbf{x},\mathbf{y},\boldsymbol{\lambda},\boldsymbol{\mu})$, where $\boldsymbol{\lambda}$ and $\boldsymbol{\mu}$ are the Lagrange multipliers for the inequality and equality constraints in~\eqref{eq:minimax}.  
We define the stationarity \rev{mapping} $\mathbf{S}(\mathbf{z})$ as
\begin{equation}\label{eq:S_minimax}
\mathbf{S}(\mathbf{z}) =
\begin{bmatrix}
\nabla_{\mathbf{x}} J(\mathbf{x},\mathbf{y}) + \mathbf{A}^\top \boldsymbol{\mu} + \nabla_{\mathbf{x}} \mathbf{G}(\mathbf{x},\mathbf{y})^\top \widetilde{\boldsymbol{\lambda}} \\[2pt]
-\nabla_{\mathbf{y}} J(\mathbf{x},\mathbf{y}) + \mathbf{B}^\top \boldsymbol{\mu} + \nabla_{\mathbf{y}} \mathbf{G}(\mathbf{x},\mathbf{y})^\top \widetilde{\boldsymbol{\lambda}} \\[2pt]
\boldsymbol{\phi}_{\varepsilon}\!\big(\boldsymbol{\lambda},\,-\mathbf{G}(\mathbf{x},\mathbf{y})\big) \\[2pt]
\mathbf{A}\mathbf{x}+\mathbf{B}\mathbf{y}-\mathbf{b} 
\end{bmatrix}.
\end{equation}
Here $\widetilde{\boldsymbol{\lambda}}$ is defined componentwise via~\eqref{eq:lambda_tilde}.
The condition $\mathbf{S}(\mathbf{z})=\mathbf{0}$ enforces the KKT system~\eqref{eq:KKT_minimax} up to the smoothing tolerance $\varepsilon$, encompassing stationarity in both $\mathbf{x}$ and $\mathbf{y}$. \rev{With $\operatorname{col}(\mathbf{x},\mathbf{y})$, $\operatorname{col}(\nabla_{\mathbf{x}}J,-\nabla_{\mathbf{y}}J)$, $\mathbf{G}$, and $[\,\mathbf{A}\ \mathbf{B}\,]$ in place of $\mathbf{x}$, $\nabla J$, $\mathbf{g}$, and $\mathbf{A}$, and the saddle point in place of the minimizer, Lemma~\ref{lem:smoothed_point} and its proof extend to~\eqref{eq:minimax}.}

Following our framework, we associate to $\mathbf{S}$ the quadratic OLF~\eqref{eq:V_quad}, and define the dynamics
\begin{equation}
\label{eq:HGD_minimax}
\dot{\mathbf{z}} \;=\; \mathbf{u}(\mathbf{z},t) \;=\; -\,\frac{\nabla \mathbf{S}(\mathbf{z})^\top \mathbf{S}(\mathbf{z})}{\big\|\nabla \mathbf{S}(\mathbf{z})^\top \mathbf{S}(\mathbf{z})\big\|^2}\;\sigma\!\big(V(\mathbf{z}),t\big).
\end{equation}
This ensures that the derivative of $V$ along trajectories is shaped by the selected decay law.

\subsection{Convex-Concave Case}
\begin{theorem}[Convex-Concave Minimax]\label{thm:minimax_main}
Suppose $J(\mathbf{x},\mathbf{y}) \in \mathcal{C}^2$ is convex in $\mathbf{x}$ and concave in $\mathbf{y}$, and the feasible set of~\eqref{eq:minimax} is compact. 
Suppose Assumptions~\ref{ass:OLF_reg} and~\ref{ass:lojasiewicz} hold on an open \rev{set} $\mathcal{U}$, and let $\Omega_c\subseteq\mathcal{U}$ be a forward-invariant sublevel set as in Lemma~\ref{lem:fwd_inv}.
Assume that there exists $(\bar{\mathbf{x}},\bar{\mathbf{y}})$ satisfying Assumption~\ref{ass:Slater_general}. 
Then, a saddle point $(\mathbf{x}^\star,\mathbf{y}^\star)$ exists, and for every $\mathbf{z}(0)\in\Omega_c$ the trajectory of~\eqref{eq:HGD_minimax} remains in the compact set $\Omega_c$ \rev{until it reaches $\mathcal{S}$} and satisfies $\mathbf{S}(\mathbf{z}(t))\to\mathbf{0}$ with the \rev{decay} law selected in~\eqref{eq:sigma_def}; in particular, any cluster point of $\mathbf{z}(t)$ (which exists by compactness of $\Omega_c$) is \rev{a zero of $\mathbf{S}$ and thus satisfies} the KKT system~\eqref{eq:KKT_x}--\eqref{eq:KKT_eq} up to the smoothing tolerance $\varepsilon$.
\end{theorem}

\begin{proof}[Proof sketch]
The VI form of~\eqref{eq:minimax} admits a saddle point on the compact convex feasible set~\cite[Cor.~2.2.5]{facchinei_finite-dimensional_2003}; Slater gives existence of $(\boldsymbol{\lambda}^\star,\boldsymbol{\mu}^\star)$\rev{. Every zero of $\mathbf{S}$ is strictly feasible, by the argument for Lemma~\ref{lem:smoothed_point}(ii), which uses no monotonicity, and a Slater argument bounds its multipliers, so $\mathcal{S}\cap\mathcal{U}$ is bounded}. Lemma~\ref{lem:fwd_inv} then yields a forward-invariant $\Omega_c\subseteq\mathcal{U}$, and the chain rule gives $\dot V=-\sigma(V,t)$; since $V=\tfrac12\|\mathbf{S}\|^2$, this drives $\mathbf{S}(\mathbf{z}(t))\to\mathbf{0}$ with the chosen timing, enforcing the KKT conditions up to the smoothing tolerance $\varepsilon$, and by compactness of $\Omega_c$ and continuity of $\mathbf{S}$ every cluster point is \rev{a zero of $\mathbf{S}$}. The full argument, including well-definedness of~\eqref{eq:HGD_minimax} along the trajectory and the chain-rule computation, is given in \iffullversion Appendix~\ref{app:thm_minimax_proof}\fi.
\end{proof}

Saddle-point dynamics under convex-concave assumptions have been shown to achieve asymptotic convergence~\cite{cherukuri_saddle-point_2017,cherukuri_role_2018}, and this analysis applies in particular to unconstrained minimax problems. 
Our framework recovers these results as special cases and extends them to the shared constrained case while guaranteeing FT, FxT, or PT convergence by an appropriate choice of the decay law.

\subsection{Strongly Convex-Strongly Concave Case}\label{subsec:scsc_case}
We next consider the standard strengthening that yields uniqueness.

\begin{corollary}\label{cor:scsc}
Suppose $J\in\mathcal{C}^2$, $J(\cdot,\mathbf{y})$ is \mbox{$m_x$-strongly} convex in $\mathbf{x}$, and $J(\mathbf{x},\cdot)$ is \mbox{$m_y$-strongly} concave in $\mathbf{y}$ with $m_x,m_y>0$.
Assume Assumption~\ref{ass:LICQ_general} holds at \rev{the saddle point $(\mathbf{x}^\star,\mathbf{y}^\star)$} and that there exists $(\bar{\mathbf{x}},\bar{\mathbf{y}})$ satisfying Assumption~\ref{ass:Slater_general}. 
Then the saddle point $(\mathbf{x}^\star,\mathbf{y}^\star)$ of~\eqref{eq:minimax} \rev{exists and} is unique. \rev{If} the feasible set of~\eqref{eq:minimax} is bounded, \rev{then} there exists $c^\ast>0$ such that for every $\mathbf{z}(0)$ with $V(\mathbf{z}(0))<c^\ast$, the trajectory of~\eqref{eq:HGD_minimax} converges with the \rev{decay} law selected in~\eqref{eq:sigma_def} to the smoothed stationarity point $\mathbf{z}^\star_\varepsilon$, $O(\varepsilon)$-close to the exact KKT point $\mathbf{z}^\star=(\mathbf{x}^\star,\mathbf{y}^\star,\boldsymbol{\lambda}^\star,\boldsymbol{\mu}^\star)$ \rev{for $\varepsilon\le\varepsilon_0$, by Lemma~\ref{lem:smoothed_point}(iv)}. When $\mathbf{G}$ is affine in $(\mathbf{x},\mathbf{y})$, $\mathbf{G}(\mathbf{x},\mathbf{y})=\mathbf{C}\,\operatorname{col}(\mathbf{x},\mathbf{y})-\mathbf{d}$, $\big[\begin{smallmatrix}[\,\mathbf{A}\ \mathbf{B}\,]\\ \mathbf{C}\end{smallmatrix}\big]$ has full row rank, and \rev{$\nabla J$ is Lipschitz}, the conclusion \rev{holds for every $\mathbf{z}(0)$, without the boundedness hypothesis}.
\end{corollary}
\begin{proof}[Proof sketch]
Strong convex-concavity makes the saddle operator $\mathcal{F}(\mathbf{w})\coloneqq\operatorname{col}(\nabla_{\mathbf{x}}J,-\nabla_{\mathbf{y}}J)$ $m$-strongly monotone with $m=\min(m_x,m_y)$, which gives existence and uniqueness of the saddle point without feasible-set compactness; boundedness of the feasible set is still used to obtain compactness of the sublevel sets $\Omega_c$ for $c<c^\ast$\rev{, and in the global clause the Lipschitz bound on $\nabla J$ gives a uniform lower bound on $\sigma_{\min}(\nabla\mathbf{S})$, as in Prop.~\ref{prop:affine_global}}. A block-elimination argument on the smoothed-FB Jacobian makes $\nabla\mathbf{S}$ nonsingular, so Assumption~\ref{ass:lojasiewicz} holds on $\Omega_c$ with $\alpha=1/2$ and Lemma~\ref{lem:fwd_inv} yields convergence. The block-elimination computation and the sublevel-set compactness argument are given in \iffullversion Appendix~\ref{app:scsc_proof}\fi.
\end{proof}

\rev{In} the unconstrained minimax case ($p=q=0$)\rev{, $\mathbf{S}=\mathcal{F}$ and $\operatorname{sym}(\nabla\mathbf{S})\succeq m\mathbf{I}$, so $\sigma_{\min}(\nabla\mathbf{S})\ge m$ on all of $\mathbb{R}^{n_x+n_y}$, and} convergence is global \rev{with no boundedness or Lipschitz hypothesis}.

\subsection{\rev{Numerical Example: Constrained Minimax}}
\label{subsec:minimax_example}

\rev{We illustrate Cor.~\ref{cor:scsc} on~\eqref{eq:minimax} with $n_x=3$, $n_y=2$, and}
\begin{equation}
\begin{split}
\rev{J(\mathbf{x},\mathbf{y})}&\rev{{}=\tfrac12\mathbf{x}^\top\mathbf{Q}_x\mathbf{x}+\mathbf{c}_x^\top\mathbf{x}+\mathbf{x}^\top\mathbf{C}_{xy}\mathbf{y}}\\
&\quad\rev{{}-\tfrac12\mathbf{y}^\top\mathbf{Q}_y\mathbf{y}-\mathbf{d}_y^\top\mathbf{y},}
\end{split}
\label{eq:minimax_cost}
\end{equation}
\rev{where $\mathbf{Q}_x$ and $\mathbf{Q}_y$ are positive definite. The constraints couple the two players: one equality and two inequalities, all affine. Their three rows are linearly independent and $\nabla J$ is affine, so the global statement of Cor.~\ref{cor:scsc} applies. At the saddle point, one of the two inequalities is active. The data and the saddle point are in \iffullversion Appendix~\ref{app:sim_params}\fi.}

\rev{We run HGD and ND from $\mathbf{x}(0)=\mathbf{1}$, $\mathbf{y}(0)=\mathbf{1}$, $\boldsymbol{\lambda}(0)=0.1\cdot\mathbf{1}$, and $\boldsymbol{\mu}(0)=\mathbf{0}$, where $V(\mathbf{z}(0))=23.5$, with $\varepsilon=10^{-6}$ and $T=10$~s. Both realizations enforce the decay law with equality, so Fig.~\ref{fig:minimax_decay} shows HGD only: the FT run settles at $9.69$~s as predicted by Lemma~\ref{lem:finite}, the FxT run at $2.38$~s, and the PT run reaches $\|\mathbf{S}\|=10^{-6}$ at $T$. All runs end within $1.1\times10^{-6}$ of $(\mathbf{x}^\star,\mathbf{y}^\star)$. The CPU times range from 193 to 213~ms for HGD and from 18 to 49~ms for ND.}

\begin{figure}[t]
  \centering
  \includegraphics[width=0.9\linewidth]{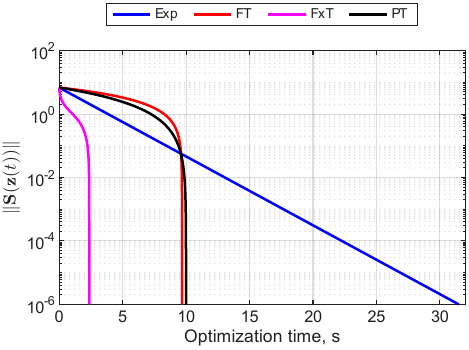}
  \caption{\rev{Norm of the stationarity mapping $\|\mathbf{S}(\mathbf{z}(t))\|$ for the constrained minimax example under the four decay laws (HGD; the ND curves coincide).}}
  \label{fig:minimax_decay}
\end{figure}

\section{Generalized Nash Equilibrium Seeking Problems}
\label{sec:gne}
Generalized Nash equilibrium (GNE) seeking problems arise when agents share coupling constraints, such as common resources or safety requirements, while pursuing individual objectives. 
These problems have been extensively studied from both the variational inequality and optimization perspectives; see the comprehensive survey in~\cite{facchinei_generalized_2010}. 
Classical reformulations based on Nikaido--Isoda functions and their regularized variants~\cite{von_heusinger_optimization_2009} provide optimization-based characterizations of (normalized) GNE. Algorithmic approaches have primarily relied on penalty and augmented Lagrangian methods, which embed the constraints into the agents’ costs. 
Representative contributions include augmented Lagrangian frameworks~\cite{kanzow_augmented_2016}, and, more recently, continuous-time penalty dynamics that establish asymptotic or exponential convergence~\cite{sun_continuous-time_2021}, as well as hierarchical-game extensions with \rev{asymptotic convergence}~\cite{huang_distributed_2024} \rev{and distributed seeking over switching networks~\cite{he_distributed_2025}, both based on projections}.

\rev{Without projections, faster regimes have been reached only for restricted constraint structures: FT convergence, guaranteed a priori for linear coupled constraints~\cite{shi_finite-time_2026}, FxT convergence for coupled equalities~\cite{hua_optimizing_2025}, and predefined-time convergence for aggregative games with coupled equality constraints~\cite{chen_predefined_2025,guo_distributed_2025}; a projected method for variational inequalities reaches FxT convergence~\cite{garg_fixed-time_2023}. To our knowledge, projection-free FxT or PT guarantees for the variational GNE (v-GNE) problem with coupled inequality constraints remain unavailable.}
In contrast, the control-centric OLF framework developed in this paper provides explicit convergence rate guarantees that extend beyond asymptotic or exponential stability. 
In what follows, we specialize this framework to strongly monotone games with convex shared constraints. \rev{The resulting dynamics are a centralized GNE solver: the feedback uses the full pseudogradient, the shared constraints, and the Jacobian of the stationarity mapping, which contains the Jacobian of the pseudogradient and the constraint Hessians. A distributed realization, in which each player accesses only its own cost and communicates with neighbors, is beyond the scope of this paper.}

\subsection{Problem Statement and Stationarity \rev{Mapping}}

Consider $N$ players and let $\mathbf{x}_{-i}$ denote the stacked decisions of all players other than player $i$. Each player $i \in \{1,\dots,N\}$ chooses a decision vector $\mathbf{x}_i \in \mathbb{R}^{n_i}$ subject to a constraint set $\mathcal{X}_i(\mathbf{x}_{-i})$ that depends on $\mathbf{x}_{-i}$ via the shared coupling constraints (defined below). 
Then, each player seeks to minimize its local cost function
\begin{equation}
    \min_{\mathbf{x}_i \in \mathcal{X}_i} J_i(\mathbf{x}_i, \mathbf{x}_{-i}),
    \label{eq:gne_problem_general}
\end{equation}
where $J_i:\mathbb{R}^{n}\!\to\!\mathbb{R}$ is \rev{twice} continuously differentiable and depends on both 
the player's own decision and those of the others.

A profile $\rev{\mathbf{x}^\star}=\operatorname{col}(\rev{\mathbf{x}^\star_1},\dots,\rev{\mathbf{x}^\star_N})$ with $n=\sum_i n_i$ is a GNE if it is feasible for the shared constraints and no player can reduce its cost by a unilateral deviation that also respects those shared constraints; namely,
\begin{equation}
J_i(\rev{\mathbf{x}^\star_i},\rev{\mathbf{x}^\star_{-i}}) \le J_i(\mathbf{x}_i,\rev{\mathbf{x}^\star_{-i}}) \quad \forall\, \mathbf{x}_i \in \mathcal{X}_i(\rev{\mathbf{x}^\star_{-i}}),
\label{eq:gne_ineq_def}
\end{equation}
with $\mathcal{X}_i(\mathbf{x}_{-i}) \coloneqq \{\mathbf{x}_i \mid \mathbf{A}\mathbf{x}=\mathbf{b},\ \mathbf{g}(\mathbf{x})\le \mathbf{0}\}$ where $\mathbf{A}\in\mathbb{R}^{q\times n}$, $\mathbf{b}\in\mathbb{R}^q$, and $\mathbf{g}:\mathbb{R}^n\to\mathbb{R}^p$ is \rev{$\mathcal{C}^2$}. 
This differs from a standard Nash equilibrium because each player's feasible set depends on the others via the shared constraints.
In convex settings, equilibria for which all players share the same multipliers $(\boldsymbol{\lambda},\boldsymbol{\mu})$ for the joint constraints are called variational GNE (v-GNE); these coincide with solutions of a single KKT system~\cite{dreves_solution_2011} and are the natural target for our OLF design.

Denote the (stacked) pseudogradient mapping
\begin{align}
\mathcal{G}(\mathbf{x}) \coloneqq \operatorname{col}\!\big(\nabla_{\mathbf{x}_1}J_1(\mathbf{x}_1,\mathbf{x}_{-1}),\dots,\nabla_{\mathbf{x}_N}J_N(\mathbf{x}_N,\mathbf{x}_{-N})\big). \label{eq:pseudogradient}
\end{align}
We adopt the variational formulation of GNE, using shared Lagrange multipliers $\boldsymbol{\lambda}\in\mathbb{R}^p_{\ge 0}$ for $\mathbf{g}(\mathbf{x})\le\mathbf{0}$ and $\boldsymbol{\mu}\in\mathbb{R}^q$ for $\mathbf{A}\mathbf{x}=\mathbf{b}$. The KKT system reads
\begin{subequations}
\begin{align}
&\mathcal{G}(\mathbf{x}^{\rev{\star}}) + \nabla \mathbf{g}(\mathbf{x}^{\rev{\star}})^\top \boldsymbol{\lambda}^{\rev{\star}} + \mathbf{A}^\top \boldsymbol{\mu}^{\rev{\star}} = \mathbf{0}, \quad \mathbf{A}\mathbf{x}^{\rev{\star}}-\mathbf{b}=\mathbf{0},\label{eq:KKT_stationarity}\\
&\mathbf{g}(\mathbf{x}^{\rev{\star}})\le \mathbf{0},\ \boldsymbol{\lambda}^{\rev{\star}}\ge \mathbf{0},\ \langle\boldsymbol{\lambda}^{\rev{\star}}, \mathbf{g}(\mathbf{x}^{\rev{\star}})\rangle=0. \label{eq:KKT_feas_comp}
\end{align}
\end{subequations}

Following our convention of encoding inequality constraints via the smoothed FB function, we define
\begin{align}
&\mathbf{r}_{\mathrm{eq}}(\mathbf{x}) \coloneqq \mathbf{A}\mathbf{x}-\mathbf{b}, 
\label{eq:req} \\
&\mathbf{r}_{\mathrm{ineq}}(\mathbf{x},\boldsymbol{\lambda}) \coloneqq \boldsymbol{\phi}_\varepsilon\!\big(\boldsymbol{\lambda},\,-\mathbf{g}(\mathbf{x})\big).
\label{eq:rineq}
\end{align}
With the stacked primal-dual variable $\mathbf{z}\coloneqq \operatorname{col}(\mathbf{x},\boldsymbol{\lambda},\boldsymbol{\mu})$, the stationarity \rev{mapping} is
\begin{align}
\mathbf{S}(\mathbf{z}) \coloneqq \operatorname{col}\!\Big(
\mathcal{G}(\mathbf{x})+\nabla \mathbf{g}(\mathbf{x})^\top \widetilde{\boldsymbol{\lambda}}+\mathbf{A}^\top \boldsymbol{\mu},
\mathbf{r}_{\mathrm{eq}}(\mathbf{x}),
\mathbf{r}_{\mathrm{ineq}}(\mathbf{x},\boldsymbol{\lambda})
\Big).
\label{eq:S_def}
\end{align}
Here $\widetilde{\boldsymbol{\lambda}}$ is defined componentwise via~\eqref{eq:lambda_tilde}; $\mathbf{r}_{\mathrm{ineq}}$ retains the original $\boldsymbol{\lambda}$.
Then \rev{$\mathbf{S}(\mathbf{z})=\mathbf{0}$ enforces~\eqref{eq:KKT_stationarity}--\eqref{eq:KKT_feas_comp} up to the smoothing tolerance $\varepsilon$}.

We define the following underlying assumption for the GNE problem.
\begin{assumption}[Strong Monotonicity]\label{ass:S1_strong_mono}
There exists $m>0$ such that for all $\mathbf{x},\mathbf{y}\in\mathbb{R}^n$
\begin{align}
\langle \mathcal{G}(\mathbf{x})-\mathcal{G}(\mathbf{y}),\,\mathbf{x}-\mathbf{y}\rangle \ge m\|\mathbf{x}-\mathbf{y}\|^2. \label{eq:strong_mono}
\end{align}
Moreover, $\mathbf{A}$ has full row rank, each component $g_j$ is convex, and the shared feasible set $\{\mathbf{x}:\mathbf{A}\mathbf{x}=\mathbf{b},\,\mathbf{g}(\mathbf{x})\le\mathbf{0}\}$ is nonempty, closed, and convex, with a Slater point.
\end{assumption}

\rev{Under Assumption~\ref{ass:S1_strong_mono}, Lemma~\ref{lem:smoothed_point} holds with the pseudogradient $\mathcal{G}$ in place of $\nabla J$ and the v-GNE in place of the minimizer, its proof covering this case, so $\mathbf{S}$ has a unique zero, whose primal part is $O(\varepsilon)$-close to the v-GNE.}

\subsection{Main Result}\label{subsec:gne_main}
We adopt the quadratic OLF~\eqref{eq:V_quad} and optimizer $\mathbf{u}(\mathbf{z},t)$ as in~\eqref{eq:u_w} to impose the chosen decay law.

\begin{theorem}[Strongly Monotone Games]\label{thm:S1}
Under Assumption~\ref{ass:S1_strong_mono}, the v-GNE \rev{$\mathbf{x}^\star$} exists and is unique. 
\rev{If the shared feasible set is bounded, there} exists $c^\ast>0$ such that for every $\mathbf{z}(0)$ with $V(\mathbf{z}(0))<c^\ast$, the closed-loop dynamics in~\eqref{eq:u_w}, instantiated with $\mathbf{S}$ from~\eqref{eq:S_def}, converge with the \rev{decay} law selected in~\eqref{eq:sigma_def} to the smoothed stationarity point $\mathbf{z}^\star_\varepsilon$\rev{, whose primal part is $O(\varepsilon)$-close to $\mathbf{x}^\star$ by Lemma~\ref{lem:smoothed_point}(iii)}. When $\mathbf{g}(\mathbf{x})=\mathbf{C}\mathbf{x}-\mathbf{d}$ is affine, $\big[\begin{smallmatrix}\mathbf{A}\\ \mathbf{C}\end{smallmatrix}\big]$ has full row rank, and \rev{$\mathcal{G}$ is Lipschitz}, the conclusion \rev{holds for every $\mathbf{z}(0)$, without the boundedness hypothesis}.
\end{theorem}
\begin{proof}[Proof sketch]
Strong monotonicity of $\mathcal{G}$ on $\ker(\mathbf{A})$ gives, in place of strong convexity, boundedness of the sublevel sets $\Omega_c$ for $c<c^\ast$ \rev{when the shared feasible set is bounded}, with full row rank of $\mathbf{A}$ controlling the equality block; together they also give existence and uniqueness of the v-GNE\rev{, and in the global clause the Lipschitz bound on $\mathcal{G}$ gives a uniform lower bound on $\sigma_{\min}(\nabla\mathbf{S})$, as in Prop.~\ref{prop:affine_global}}. A block-elimination argument on the smoothed-FB Jacobian makes $\nabla\mathbf{S}$ nonsingular, so Assumption~\ref{ass:lojasiewicz} holds on $\Omega_c$ with $\alpha=1/2$ and Lemma~\ref{lem:fwd_inv} yields convergence. The block-elimination computation and the sublevel-set compactness argument are given in \iffullversion Appendix~\ref{app:nonsing_full}\fi.
\end{proof}

The classical (unconstrained) Nash equilibrium and unconstrained minimax cases are recovered by setting $p=q=0$, which eliminates the FB and equality blocks of $\mathbf{S}$. \rev{Then $\mathbf{S}=\mathcal{G}$, and strong monotonicity~\eqref{eq:strong_mono} gives $\operatorname{sym}(\nabla\mathbf{S})\succeq m\mathbf{I}$, hence $\sigma_{\min}(\nabla\mathbf{S})\ge m$ on all of $\mathbb{R}^n$, so convergence is global with no boundedness or Lipschitz hypothesis, and the ND and GD realizations of Lemmas~\ref{lem:pseudo} and~\ref{lem:svb} apply as well.}

\subsection{Numerical Example: Cournot GNE with Shared Constraints}\label{subsec:cournot_gne}

We illustrate Thm.~\ref{thm:S1} with a Cournot competition game~\cite{martinez-piazuelo_payoff_2022}.
Consider $N=4$ firms and $M=2$ markets. Each firm selects a production vector $\rev{\mathbf{x}_i}\in\rev{\mathbb{R}^2}$, and the aggregate supply is $\rev{\mathbf{E}}\mathbf{x}$. 
Market prices follow a linear inverse demand
\begin{equation}
\rev{\boldsymbol{\pi}}(\rev{\mathbf{E}}\mathbf{x}) = \rev{\bar{\boldsymbol{\pi}}} - \mathbf{D}\,\rev{\mathbf{E}}\mathbf{x},
\end{equation}
and each firm incurs a quadratic production cost $\rev{Q_i(\mathbf{x}_i) = \tfrac{1}{2}\|\mathbf{x}_i\|^2}$. 
The shared constraints consist of one linear equality (a target supply in the first market) and \rev{a capacity of the second market}.


\rev{These shared constraints are affine with linearly independent rows and the pseudogradient is affine, so the global statement of Thm.~\ref{thm:S1} applies. We omit nonnegativity constraints, which are inactive at the v-GNE, and a capacity of the first market, which the target supply would make redundant. At the v-GNE, the market-2 capacity is active; the data and the v-GNE are in \iffullversion Appendix~\ref{app:sim_params}\fi. We run HGD and ND from $\mathbf{x}(0)=0.2\cdot\mathbf{1}$, $\boldsymbol{\lambda}(0)=\mathbf{0}$, and $\boldsymbol{\mu}(0)=\mathbf{0}$, where the equality residual is $11.2$ and $V(\mathbf{z}(0))=310.1$, with $\varepsilon=10^{-6}$ and $T=10$~s. Both realizations enforce the decay law with equality, so Fig.~\ref{fig:Cournot_decay} shows HGD only: the FT run settles at $35.22$~s as predicted by Lemma~\ref{lem:finite}, the FxT run at $2.42$~s, and the PT run reaches $\|\mathbf{S}\|=10^{-6}$ at $T$. All runs end within $5\times10^{-7}$ of $\mathbf{x}^\star$. The CPU times\iffullversion{} (Table~\ref{tab:Cournot_cpu})\fi{} range from 72 to 108~ms for HGD and from 54 to 94~ms for ND.}
\begin{figure}[t]
  \centering
  \includegraphics[width=0.9\linewidth]{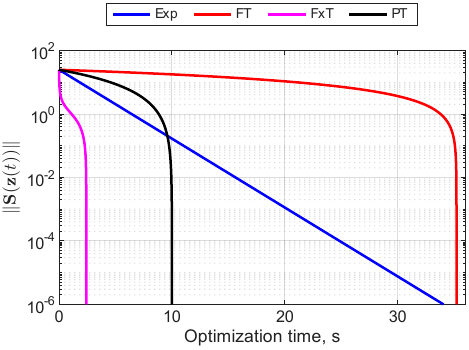}
    \caption{\rev{Norm of the stationarity mapping $\|\mathbf{S}(\mathbf{z}(t))\|$ for the Cournot game under the four decay laws (HGD; the ND curves coincide).}}
  \label{fig:Cournot_decay}
\end{figure}

\iffullversion
\begin{table}[t]
  \centering
  \caption{Wall-clock CPU time (ms) for the Cournot GNE\rev{, median of 20 runs}.}
  \label{tab:Cournot_cpu}
  \begin{tabular}{lcccc}
    \hline \rev{\textbf{Case}} & \textbf{Exp} & \textbf{FT} & \textbf{FxT} & \textbf{PT} \\
    \hline \rev{HGD} & \rev{72} & \rev{85} & \rev{108} & \rev{81} \\
    \rev{ND} & \rev{59} & \rev{54} & \rev{94} & \rev{57} \\
    \hline
  \end{tabular}
\end{table}
\fi

As in Remark~\ref{rem:stat_mon}, since the Lyapunov function is built directly from the stacked stationarity \rev{mapping}, driving $V(\mathbf{z})$ below a given tolerance automatically certifies that each optimality component is satisfied to the same numerical accuracy.

The Cournot game thus illustrates how the dynamics extend beyond convex optimization to multi-agent equilibria with \rev{coupled equality and inequality constraints, reaching feasibility in the limit}.

\rev{In \iffullversion Appendix~\ref{app:comparisons}\fi, OLF, in its fastest configuration for each example, is compared on the three constrained examples with the projected primal-dual and saddle-point flows of~\cite{feijer_stability_2010,cherukuri_role_2018} and, on NUM, with the generalized-Lagrangian primal-dual flow of~\cite{ozaslan_exponential_2024} with its FT and FxT versions, from ten random initial conditions and with the same solver, tolerances, and stopping rule. From each of these initial conditions, OLF needs fewer right-hand-side evaluations than every baseline, and it has the smallest median CPU time in all three examples: by a factor of 2.5 on minimax, by 6 percent on Cournot, and by 2 percent on NUM. Among the compared methods, only OLF has a PT guarantee.}

    \section{Conclusions}
    \label{sec:conc}
    
    This paper introduced a control-centric framework for the systematic design of continuous-time optimization algorithms. The framework is founded on optimization Lyapunov functions (OLFs), which pair a Lyapunov candidate encoding the optimality conditions with a \rev{decay law} that specifies the desired convergence rate, yielding feedback dynamics with guaranteed exponential, FT, FxT, or PT convergence. Three realizations were developed: the Hessian-gradient, Newton, and gradient dynamics. They differ in structure and information requirements but share the same Lyapunov-based synthesis principle, ranging from a normalized second-order descent that avoids explicit matrix inversion (HGD), through full Hessian feedback (ND), to first-order gradient feedback (GD). The framework was applied to unconstrained, constrained, minimax, and GNE seeking problems, recovering known results as special cases and \rev{providing FT, FxT, and PT convergence for constrained problems with nonlinear convex constraints without projections onto the feasible set}. This perspective offers systematic design tools that complement existing algorithmic approaches and may inform future developments in intelligent optimization and cyber-physical decision-making.

\iffullversion
    \appendix
\ifsupplement\setcounter{table}{2}\fi 
    \section{Proofs}\label{app:proofs}

    \subsection{Proof of Lemma~\ref{lem:svb}}
    \label{app:gd_proof}

    \begin{proof}
    From~\eqref{eq:V_quad} and~\eqref{eq:Vdot_general}, $\dot V=-\rev{\frac{\sigma(V,t)}{2mV}}\,\mathbf{S}^\top\nabla\mathbf{S}\,\mathbf{S}$. Decompose
    \begin{equation}
    \nabla\mathbf{S}=\mathbf{H}+\mathbf{K},\qquad\begin{aligned}\mathbf{H}&\coloneqq\tfrac12(\nabla\mathbf{S}+\nabla\mathbf{S}^\top),\\ \mathbf{K}&\coloneqq\tfrac12(\nabla\mathbf{S}-\nabla\mathbf{S}^\top).\end{aligned}
    \end{equation}
    Since $\mathbf{K}^\top=-\mathbf{K}$, $\mathbf{S}^\top \mathbf{K}\mathbf{S}=0$, so by~\eqref{eq:symm_part_cond}
    \begin{equation}
    \mathbf{S}^\top\nabla\mathbf{S}\,\mathbf{S}=\mathbf{S}^\top \mathbf{H}\mathbf{S}\ge m\|\mathbf{S}\|^2.
    \end{equation}
    Substituting yields $\dot V\le-\rev{\frac{\sigma(V,t)}{2mV}}\,m\|\mathbf{S}\|^2=-\sigma(V,t)$, which is~\eqref{eq:gd_vdot_bound}.
    \end{proof}

    The following example, deferred from Sec.~\ref{sec:control_centric_olf}, illustrates the GD feedback in \rev{a} setting where the symmetric-part hypothesis~\eqref{eq:symm_part_cond} holds globally.

    \begin{exmp}[GD for Unconstrained Optimization]
    \label{ex:gd_unconstrained}
    Let $\mathbf{S}(\mathbf{x})=\nabla J(\mathbf{x})$ and $V(\mathbf{x})=\tfrac{1}{2}\|\mathbf{S}(\mathbf{x})\|^2$. Assume $J$ is strongly convex, so that $\nabla^2 J(\mathbf{x})\succeq m\mathbf{I}$ for all $\mathbf{x}\in\mathbb{R}^n$ and some $m>0$. The gradient feedback
    \begin{equation}
    \mathbf{u}(\mathbf{x},t)=-\rev{\frac{\sigma(V(\mathbf{x}),t)}{2mV(\mathbf{x})}}\nabla J(\mathbf{x})
    \end{equation}
    then yields
    \begin{equation}
    \begin{split}
    \dot V&=-\rev{\frac{\sigma}{2mV}}\nabla J^\top\nabla^2 J\nabla J\\
    &\le-\rev{\frac{\sigma}{2mV}}\,m\|\nabla J\|^2=-\sigma(V(\mathbf{x}),t)
    \end{split}
    \end{equation}
    for any decay law $\sigma$ as in~\eqref{eq:sigma_def}, so the \rev{decay} law is enforced in inequality form under strong convexity.
    \end{exmp}

    \subsection{Proof of Lemma~\ref{lem:fwd_inv}}
    \label{app:fwd_inv_proof}

    \begin{proof}
    The proof has two parts. Steps 1--3 are feedback-agnostic: they use $V$ alone to establish closedness and a geometric tube bound on $\Omega_c$. Step 4 uses $\dot V\le 0$ along the closed-loop trajectory to conclude forward invariance, and is therefore the only step that depends on the specific feedback law.

    Step 1 (closedness). Continuity of $V$ implies that $\Omega_c$ is closed in $\mathcal{U}$.

    Step 2 (tube bound via an auxiliary gradient flow on $V$). Fix any $\mathbf{z}\in\Omega_c\setminus\mathcal{S}$ and consider the auxiliary flow
    \begin{equation}
    \dot{\boldsymbol{\xi}}(s)\;=\;-\nabla V(\boldsymbol{\xi}(s)),\qquad \boldsymbol{\xi}(0)=\mathbf{z}.
    \label{eq:aux_flow}
    \end{equation}
    The flow~\eqref{eq:aux_flow} depends on $V$ alone and is unrelated to the closed-loop dynamics of Lemmas~\ref{lem:pseudo} and~\ref{lem:svb}. For the HGD feedback~\eqref{eq:u_w}, $\dot{\mathbf{z}}$ is parallel to $-\nabla V$ with $\|\dot{\mathbf{z}}\|=\sigma/\|\nabla V\|$ and $\dot V=-\sigma$, so the closed-loop orbit coincides with that of~\eqref{eq:aux_flow} up to a time reparameterization; its arc length equals $\int_0^{V(\mathbf{z})}\!dV/\|\nabla V\|$ and is therefore subject to the bound~\eqref{eq:tube_radius} below, independently of the \rev{decay} law. Along~\eqref{eq:aux_flow}, $V$ is non-increasing, so $\boldsymbol{\xi}(s)$ remains in $\Omega_c$. By Assumption~\ref{ass:lojasiewicz} applied to $\boldsymbol{\xi}(s)$,
    \begin{equation}
    \dot V(\boldsymbol{\xi}) = -\|\nabla V(\boldsymbol{\xi})\|^2 \le -\,c_{\mathrm{L}}\,V(\boldsymbol{\xi})^\alpha\,\|\nabla V(\boldsymbol{\xi})\|,
    \end{equation}
    hence $\|\dot{\boldsymbol{\xi}}(s)\|\le -\dot V(\boldsymbol{\xi}(s))/[c_{\mathrm{L}}\,V(\boldsymbol{\xi}(s))^\alpha]$. Integrating until $V(\boldsymbol{\xi}(s))$ vanishes gives the path-length bound
    \begin{align}
    \int_0^{\infty}\!\|\dot{\boldsymbol{\xi}}(s)\|\,ds
    &\;\le\;
    \int_0^{V(\mathbf{z})}\!\frac{dV}{c_{\mathrm{L}}\,V^\alpha}
    \;=\;\frac{V(\mathbf{z})^{1-\alpha}}{c_{\mathrm{L}}(1-\alpha)} \notag \\
    &\;\le\;\frac{c^{1-\alpha}}{c_{\mathrm{L}}(1-\alpha)}\;\eqqcolon\;r_c.
    \label{eq:tube_radius}
    \end{align}
    Since $\boldsymbol{\xi}(s)$ accumulates at a point of $\mathcal{S}\cap\mathcal{U}$, the Euclidean distance from $\mathbf{z}$ to that point is bounded by~\eqref{eq:tube_radius}. Hence $\mathrm{dist}(\mathbf{z},\mathcal{S}\cap\mathcal{U})\le r_c$ for every $\mathbf{z}\in\Omega_c$, and $\Omega_c$ is contained in the closed tube
    \begin{equation}
    T_c\coloneqq\{\mathbf{z}\in\mathcal{U}\,:\,\mathrm{dist}(\mathbf{z},\mathcal{S}\cap\mathcal{U})\le r_c\}.
    \end{equation}

    Step 3 (conditional boundedness). Boundedness of $\Omega_c$ is not unconditional. Suppose additionally that $\mathcal{S}\cap\mathcal{U}$ is bounded. Then $T_c$ is a finite-radius tube around a bounded set and is itself bounded; consequently, so is $\Omega_c$.

    Step 4 (forward invariance under any of the three feedbacks). For any feedback of Lemmas~\ref{lem:w_pseudo}, \ref{lem:pseudo}, or~\ref{lem:svb}, $\dot V\le 0$ on $\mathcal{U}\setminus\mathcal{S}$. \rev{Choose $\rho^\ast>0$ below the infimum of $V$ on $\overline{\mathcal{U}}\setminus\mathcal{N}$; for $c\in(0,\rho^\ast]$, every $\mathbf{z}\in\overline{\mathcal{U}}$ with $V(\mathbf{z})\le c$ lies in $\mathcal{N}$, so a path from $\Omega_c$ along which $V\le c$ stays in $\Omega_c\subseteq\mathcal{N}$, and Steps~2 and~3 apply.} For every $c\in(0,\rho^\ast]$, $\Omega_c\subseteq T_c\subseteq\mathcal{U}$, and $\dot V\le 0$ on $\Omega_c\setminus\mathcal{S}$. By~\cite[Thm.~4.1]{khalil_nonlinear_2002}, $\Omega_c$ is forward invariant under the closed-loop dynamics: $\mathbf{z}(0)\in\Omega_c$ implies $\mathbf{z}(t)\in\Omega_c$ \rev{until it reaches $\mathcal{S}$}.
    \end{proof}

    \subsection{PL via a Nonsingular Jacobian on Compact Sublevel Sets}
    \label{app:lin_PL}

    \begin{lemma}[PL via a Nonsingular Jacobian on Compact Sublevel Sets]
    \label{lem:lin_PL}
    Let $V(\mathbf{z})=\tfrac{1}{2}\|\mathbf{S}(\mathbf{z})\|^2$ with $\mathbf{S}\in\mathcal{C}^1(\mathcal{U};\mathbb{R}^N)$ and $\mathcal{U}\subseteq\mathbb{R}^N$ (so that $\nabla\mathbf{S}$ is square), and let $\Omega_c=\{\mathbf{z}\in\mathcal{U}:V(\mathbf{z})\le c\}$ be a compact sublevel set on which $\nabla\mathbf{S}(\mathbf{z})$ is nonsingular. Define
    \begin{equation}
    \sigma_c\;\coloneqq\;\min_{\mathbf{z}\in\Omega_c}\sigma_{\min}\!\big(\nabla\mathbf{S}(\mathbf{z})\big)\;>\;0.
    \label{eq:sigma_c}
    \end{equation}
    Then $V$ satisfies the PL inequality on $\Omega_c$:
    \begin{equation}
    \|\nabla V(\mathbf{z})\|^2 \;\ge\; 2\,\sigma_c^{\,2}\,V(\mathbf{z}),\qquad \mathbf{z}\in\Omega_c.
    \label{eq:PL_compact}
    \end{equation}
    In particular, Assumption~\ref{ass:lojasiewicz} holds on $\Omega_c$ with $\alpha=1/2$ and $c_{\mathrm{L}}=\sigma_c\sqrt{2}$.
    \end{lemma}
    \begin{proof}
    By the chain rule, $\nabla V(\mathbf{z})=\nabla\mathbf{S}(\mathbf{z})^\top\mathbf{S}(\mathbf{z})$, so $\|\nabla V(\mathbf{z})\|\ge\sigma_{\min}(\nabla\mathbf{S}(\mathbf{z}))\,\|\mathbf{S}(\mathbf{z})\|\ge\sigma_c\,\|\mathbf{S}(\mathbf{z})\|$. Squaring and using $\|\mathbf{S}\|^2=2V$ yields~\eqref{eq:PL_compact}. The minimum in~\eqref{eq:sigma_c} is attained and positive: $\sigma_{\min}(\nabla\mathbf{S}(\cdot))$ is continuous, $\Omega_c$ is compact, and $\nabla\mathbf{S}$ is nonsingular on $\Omega_c$ by hypothesis.
    \end{proof}

    \subsection{Analysis of the Regularized HGD}
    \label{app:regularized_hgd}

    This appendix provides the quantitative details deferred from Remark~\ref{rem:regularizers}: the bound on the regularized feedback, the rate-slowdown formula, and the time regularization for the PT divergent gain.

    Bound on the feedback. The regularized feedback~\eqref{eq:HGD_reg} satisfies
    \begin{equation}
    \|\mathbf{u}_\eta(\mathbf{z},t)\| \;\le\; \frac{\sigma(V,t)}{2\sqrt{\eta}},
    \end{equation}
    which is finite for every $\eta>0$, so trajectories are forward complete (no finite-time escape) without invoking Assumption~\ref{ass:lojasiewicz}.

    Rate slowdown. Substituting~\eqref{eq:HGD_reg} into the chain rule for $V=\tfrac{1}{2}\|\mathbf{S}\|^2$ gives
    \begin{equation}
    \dot V \;=\; -\,\sigma(V,t)\,\frac{\|\nabla V\|^2}{\|\nabla V\|^2+\eta},
    \end{equation}
    which matches the \rev{chosen decay law} only on $\{\|\nabla V\|^2\gg\eta\}$. In the FT, FxT, and PT cases, the settling time \rev{of the chosen law} is therefore no longer met. \rev{For trajectories in a compact sublevel set of $V$ contained in $\mathcal{U}$, which is forward invariant since $\dot V\le 0$,} asymptotic convergence to $\mathcal{S}$ is preserved because, under Assumption~\ref{ass:OLF_reg}, $\nabla V=\mathbf{0}$ implies $\mathbf{S}=\mathbf{0}$, and the regularizer introduces no terminal residual on $V$.

    Time regularization for the PT divergent gain. The PT gain $\mu/(T-t)$ admits the regularization $(T-t)^{-1}\mapsto(T-t+\bar\varepsilon)^{-1}$~\cite{song_prescribed-time_2023}, which bounds the feedback on $[0,T]$ at the cost of a small terminal residual on $V$. The trade-off exchanges exact convergence at $t=T$ for a bounded gain on $[0,T]$. \rev{Under HGD and ND the law holds with equality, so
    \begin{equation}
    \begin{aligned}
    V(\mathbf{z}(t))&=V(\mathbf{z}(0))\big((T+\bar\varepsilon-t)/(T+\bar\varepsilon)\big)^{\mu},\\
    V(\mathbf{z}(T))&=V(\mathbf{z}(0))\big(\bar\varepsilon/(T+\bar\varepsilon)\big)^{\mu};
    \end{aligned}
    \end{equation}
    under GD, $V(\mathbf{z}(T))$ is at most this value. For the singular gain, HGD gives $V(\mathbf{z}(t))=V(\mathbf{z}(0))(1-t/T)^{\mu}$ and, by Assumption~\ref{ass:lojasiewicz},
    \begin{equation}
    \|\dot{\mathbf{z}}\|=\sigma/\|\nabla V\|\le\mu V^{1-\alpha}/[c_{\mathrm{L}}(T-t)],
    \end{equation}
    which gives the bound of Remark~\ref{rem:pt_gain}; the same step gives $\|\dot{\mathbf{z}}\|\le(k/c_{\mathrm{L}})V^{\gamma-\alpha}$ for the FT law and $\|\dot{\mathbf{z}}\|\le(a/c_{\mathrm{L}})V^{\gamma-\alpha}+(b/c_{\mathrm{L}})V^{\delta-\alpha}$ for the FxT law. The limit $\mathbf{z}(T)\in\mathcal{S}$ follows from the arc-length bound~\eqref{eq:tube_radius} and $V(\mathbf{z}(t))\to 0$.}

    \rev{We check these results numerically for the function of Sec.~\ref{subsec:logsumexp_example} with $n=50$, $T=1$~s, and HGD. With the singular gain, integrated up to $10^{-2}$, $10^{-4}$, and $10^{-6}$~s before $T$, the largest state speed over the last tenth of each run grows as the inverse square root of that margin for $\mu=1$ (35, 352, and 3519), stays at $7.04$ for $\mu=2$, and stays between $3.3$ and $3.5$ for $\mu=3$, in line with Remark~\ref{rem:pt_gain}, since this strongly convex problem satisfies the PL condition; at the end of every run, $V$ is within a relative $10^{-4}$ of $V(\mathbf{z}(0))(1-t/T)^{\mu}$. With the bounded gain, $\bar\varepsilon=1.39\times10^{-7}$ for $\mu=2$ and $\bar\varepsilon=2.69\times10^{-5}$ for $\mu=3$ give $\|\mathbf{S}(\mathbf{z}(T))\|=10^{-6}$, with largest gains $1.4\times10^{7}$ and $1.1\times10^{5}$.}

\subsection{\rev{Relation to Known Flows}}
\label{app:known_flows}
\rev{Let $\mathbf{S}=\nabla J$, let $V=\tfrac12\|\nabla J\|^2$ be the quadratic OLF, and write $\mathbf{g}\coloneqq\nabla J(\mathbf{x})$, $\mathbf{H}\coloneqq\nabla^2 J(\mathbf{x})$, and $r_s\coloneqq 2^{-s}\|\mathbf{g}\|^{2s-2}$. Then GD~\eqref{eq:u_gamma} and ND~\eqref{eq:u_pseudo} read $\dot{\mathbf{x}}=-\sigma(V,t)\,\mathbf{g}/(2mV)$ and $\dot{\mathbf{x}}=-\sigma(V,t)\,\mathbf{H}^{-1}\mathbf{g}/(2V)$, and the laws~\eqref{eq:sigma_def} give the flows below, which recover known flows up to constant gains; in the FxT rows, $1/2<\gamma<1<\delta$.}
\begin{center}\footnotesize
\rev{\begin{tabular}{@{}l@{\hspace{5pt}}l@{\hspace{5pt}}l@{}}
Realization, law & $\dot{\mathbf{x}}$ & Known flow\\ \hline
GD, Exp & $-\frac{c}{2m}\,\mathbf{g}$ & gradient~\cite{brown_effective_1989}\\
GD, FT, $\gamma=\frac12$ & $-\frac{k}{\sqrt2\,m}\,\mathbf{g}/\|\mathbf{g}\|$ & normalized gradient~\cite{cortes_finite-time_2006}\\
GD, FT, $\frac12<\gamma<1$ & $-\frac{k}{m}\,r_\gamma\,\mathbf{g}$ & rescaled gradient~\cite{romero_finite-time_2020}\\
GD, FxT & $-\frac1m(a\,r_\gamma+b\,r_\delta)\,\mathbf{g}$ & FxT gradient~\cite{garg_fixed-time_2021}\\
GD, PT & $-\frac{\mu}{2m(T-t)}\,\mathbf{g}$ & time-varying gain~\cite{aal_controlled_2025}\\
ND, Exp & $-\frac{c}{2}\,\mathbf{H}^{-1}\mathbf{g}$ & Newton~\cite{brown_effective_1989}\\
ND, FT, $\gamma=\frac12$ & $-\frac{k}{\sqrt2}\,\mathbf{H}^{-1}\mathbf{g}/\|\mathbf{g}\|$ & rescaled Newton~\cite{romero_finite-time_2020}\\
ND, FxT & $-(a\,r_\gamma+b\,r_\delta)\,\mathbf{H}^{-1}\mathbf{g}$ & FxT Newton~\cite{garg_fixed-time_2021}\\
\end{tabular}}
\end{center}
\rev{The time-varying-gain flow of~\cite{aal_controlled_2025} is the one with unit exponent. Under the FT law, HGD~\eqref{eq:u_w} reads $\dot{\mathbf{x}}=-k\,2^{-\gamma}\|\mathbf{g}\|^{2\gamma}\mathbf{H}\mathbf{g}/\|\mathbf{H}\mathbf{g}\|^2$, which is the flow of~\cite[Eq.~(27)]{romero_finite-time_2020} with $r=1$; under the other laws, it recovers none of these flows.}

    \subsection{Proof of Theorem~\ref{thm:unified_constrained}}
    \label{app:thm_constrained_proof}

    \begin{proof}
    The argument proceeds in three steps. (i) Uniqueness of the multipliers and singleton stationarity set. Assumption~\ref{ass:LICQ_general} ensures uniqueness of the multipliers $(\boldsymbol{\lambda}^\star,\boldsymbol{\mu}^\star)$ associated with $\mathbf{x}^\star$~\cite[Thm.~11.12]{beck_introduction_2014}; by the hypothesis $\mathcal{S}\cap\mathcal{U}=\{\mathbf{z}^\star_\varepsilon\}$ of Thm.~\ref{thm:unified_constrained}, the smoothed stationarity point is the unique zero of $\mathbf{S}$ in $\mathcal{U}$, hence $\mathcal{S}\cap\mathcal{U}$ is bounded.

    (ii) Forward invariance of $\Omega_c$. Under Assumptions~\ref{ass:OLF_reg} and~\ref{ass:lojasiewicz}, and given the boundedness of $\mathcal{S}\cap\mathcal{U}$ from step (i), Lemma~\ref{lem:fwd_inv} provides $\rho^\ast>0$ such that $\Omega_c\subseteq\mathcal{U}$ is closed, bounded, and forward invariant under~\eqref{eq:HGD_general} for every $c\in(0,\rho^\ast]$; hence for $\mathbf{z}(0)\in\Omega_c$, the trajectory $\mathbf{z}(t)\in\Omega_c\subseteq\mathcal{U}$ \rev{until it reaches $\mathcal{S}$}, and Assumption~\ref{ass:OLF_reg} yields $\nabla\mathbf{S}(\mathbf{z}(t))^\top\mathbf{S}(\mathbf{z}(t))\neq\mathbf{0}$ whenever $\mathbf{S}(\mathbf{z}(t))\neq\mathbf{0}$, so the denominator in~\eqref{eq:HGD_general} is well-defined.

    (iii) Convergence with the \rev{chosen decay law}. Substituting~\eqref{eq:HGD_general} into the chain rule for $V=\tfrac{1}{2}\|\mathbf{S}\|^2$ gives
    \begin{equation}
    \dot V
    =
    \mathbf{S}^\top \nabla\mathbf{S}\left(-\frac{\nabla \mathbf{S}^\top \mathbf{S}}{\|\nabla \mathbf{S}^\top \mathbf{S}\|^2}\sigma(V,t)\right)
    = -\,\sigma(V,t),
    \end{equation}
    which is~\eqref{eq:Vdot_scalar_template}.
    Since $V=\frac12\|\mathbf{S}\|^2$, the decay of $V$ to zero implies $\mathbf{S}(\mathbf{z}(t))\to \mathbf{0}$ with the indicated timing; the KKT conditions associated with the inequality constraints are enforced up to the smoothing tolerance $\varepsilon$ of~\eqref{eq:FB_smooth_gen}.
    \end{proof}

    \subsection{Proof of Proposition~\ref{prop:c_star}}
    \label{app:c_star_proof}

    \begin{proof}
    Closedness of $\Omega_c$ in $\mathbb{R}^{n+p+q}$ is immediate from continuity of $V$. We establish boundedness ($V$ is coercive on $\{V\le c\}$, $c<c^\ast$) in \rev{three} stages. Stage~1 reduces the dual block to the primal block: $\boldsymbol{\mu}$ is controlled by $(\mathbf{x},\boldsymbol{\lambda})$, and \rev{$g_j(\mathbf{x})$ stays bounded along any sequence on which $\lambda_j\to+\infty$}. Stage~2 rules out escape of $\mathbf{x}$, treating the simultaneous divergence of $\mathbf{x}$ and the multipliers via the FB structure and boundedness of the feasible set. \rev{Stage~3 rules out escape of $\boldsymbol{\lambda}$ with $\mathbf{x}$ bounded.} Throughout, $V\le c$ bounds every block of $\mathbf{S}$ by $\sqrt{2c}$.

    \emph{Stage 1a ($\boldsymbol{\mu}$ via $(\mathbf{x},\boldsymbol{\lambda})$).} From the first block, $\mathbf{A}^\top\boldsymbol{\mu}=\mathbf{s}_1-\nabla J(\mathbf{x})-\nabla\mathbf{g}(\mathbf{x})^\top\widetilde{\boldsymbol{\lambda}}$, so full row rank of $\mathbf{A}$ gives
    \begin{equation}
    \|\boldsymbol{\mu}\|\le\sigma_{\min}(\mathbf{A})^{-1}\big(\sqrt{2c}+\|\nabla J(\mathbf{x})\|+\|\nabla\mathbf{g}(\mathbf{x})^\top\widetilde{\boldsymbol{\lambda}}\|\big);
    \end{equation}
    $\boldsymbol{\mu}$ is bounded once $(\mathbf{x},\boldsymbol{\lambda})$ are. It therefore suffices to bound $\mathbf{x}$ and $\boldsymbol{\lambda}$.

    \emph{Stage 1b (FB localization, with or without $\mathbf{x}$ bounded).} For $a,b\ge 0$, $\phi_\varepsilon(a,b)\le\varepsilon-ab/(a+b)$, since
    \begin{equation}
    \begin{split}
    \sqrt{a^2+b^2}-(a+b)&=-2ab/(\sqrt{a^2+b^2}+a+b)\\
    &\le -ab/(a+b).
    \end{split}
    \end{equation}
    Let $\mathcal{J}_+\coloneqq\{j:\lambda_j^k\to+\infty\}$ along an escape sequence. For $j\in\mathcal{J}_+$, taking $a=\lambda_j^k$, $b=-g_j(\mathbf{x}^k)$, boundedness of the FB block excludes $-g_j(\mathbf{x}^k)\to+\infty$ (else $\phi_\varepsilon\to-\infty$), and the symmetric counterpart of the sign screen (C1) in Stage~3 below excludes $g_j(\mathbf{x}^k)\to+\infty$. Hence $g_j(\mathbf{x}^k)$ is bounded for every $j\in\mathcal{J}_+$, even when $\|\mathbf{x}^k\|\to\infty$.

    \emph{Stage 2 ($\mathbf{x}$ cannot escape).} Suppose $\|\mathbf{x}^k\|\to\infty$ on $\{V\le c\}$. The equality block $\mathbf{h}=\mathbf{A}\mathbf{x}-\mathbf{b}$ bounded and full row rank of $\mathbf{A}$ bound the component of $\mathbf{x}^k$ in $\operatorname{range}(\mathbf{A}^\top)$, so the component $\mathbf{x}_\perp^k$ in $\ker(\mathbf{A})$ diverges; write $\mathbf{x}^k=\mathbf{x}_\parallel^k+\mathbf{x}_\perp^k$ with $\mathbf{x}_\parallel^k$ bounded. If the multipliers $\widetilde{\boldsymbol{\lambda}}^k$ are bounded, testing $\mathbf{s}_1^k$ against $\mathbf{x}_\perp^k\in\ker(\mathbf{A})$ (which annihilates $\mathbf{A}^\top\boldsymbol{\mu}^k$) gives, by strong convexity along $\ker(\mathbf{A})$ and convexity of each $g_i$,
    \begin{align}
    \langle\mathbf{s}_1^k,\mathbf{x}_\perp^k\rangle&\ge m\|\mathbf{x}_\perp^k\|^2\rev{+\langle\nabla J(\mathbf{x}_\parallel^k),\mathbf{x}_\perp^k\rangle} \notag \\
    &\quad+\textstyle\sum_i\widetilde{\lambda}_i^k\big(g_i(\mathbf{x}^k)-g_i(\mathbf{x}_\parallel^k)\big) \notag \\
    &\ge m\|\mathbf{x}_\perp^k\|^2-C\|\mathbf{x}_\perp^k\|,
    \end{align}
    using
    \begin{equation}
    \begin{split}
    g_i(\mathbf{x}^k)-g_i(\mathbf{x}_\parallel^k)&\ge\langle\nabla g_i(\mathbf{x}_\parallel^k),\mathbf{x}_\perp^k\rangle\\
    &\ge-\|\nabla g_i(\mathbf{x}_\parallel^k)\|\,\|\mathbf{x}_\perp^k\|
    \end{split}
    \end{equation}
    \rev{and $\langle\nabla J(\mathbf{x}_\parallel^k),\mathbf{x}_\perp^k\rangle\ge-\|\nabla J(\mathbf{x}_\parallel^k)\|\,\|\mathbf{x}_\perp^k\|$,} with $\mathbf{x}_\parallel^k$ bounded. Since $\langle\mathbf{s}_1^k,\mathbf{x}_\perp^k\rangle\le\sqrt{2c}\,\|\mathbf{x}_\perp^k\|$, this forces $\|\mathbf{x}_\perp^k\|\le(\sqrt{2c}+C)/m$, a contradiction. Hence bounded multipliers preclude $\mathbf{x}$-escape.

    It remains to exclude simultaneous divergence of $\mathbf{x}^k$ and the multipliers $\widetilde{\lambda}_j^k$ ($j\in\mathcal{J}_+$). \rev{Since $\phi_\varepsilon(a,b)\ge|a|-a-b\ge-b$, the FB block gives $g_j(\mathbf{x}^k)\le\sqrt{2c}$ for every $j$}; passing to a subsequence, $\hat{\mathbf{x}}\coloneqq\lim\mathbf{x}^k/\|\mathbf{x}^k\|\in\ker(\mathbf{A})$ satisfies $g_j^\infty(\hat{\mathbf{x}})\le 0$ for \rev{every $j$}, where $g_j^\infty$ is the recession function of the convex $g_j$. Thus $\hat{\mathbf{x}}$ is a nonzero recession direction of the feasible set $\{\mathbf{A}\mathbf{x}=\mathbf{b},\,\mathbf{g}(\mathbf{x})\le\mathbf{0}\}$. By hypothesis this set is bounded, so it has no nonzero recession direction; the simultaneous-divergence mode is therefore empty.

    \emph{Stage 3 ($\boldsymbol{\lambda}$ cannot escape with $\mathbf{x}$ bounded).} With $\mathbf{x}$ bounded by Stage~2, suppose, for contradiction, a sequence $\mathbf{z}^k\in\Omega_c$ (so $V(\mathbf{z}^k)\le c<c^\ast$) has $\|\boldsymbol{\lambda}^k\|\to\infty$. Since $V\le c$ bounds every block of $\mathbf{S}$, we have $\|\mathbf{s}_1^k\|\le\sqrt{2c}$ and $|\phi_\varepsilon(\lambda_j^k,-g_j(\mathbf{x}^k))|\le\sqrt{2c}$ for all $j,k$; $\|\mathbf{A}\mathbf{x}^k-\mathbf{b}\|\le\sqrt{2c}$; pass to a subsequence with $\mathbf{x}^k\to\rev{\mathbf{x}_\infty}$ (bounded by Stage~2). The multipliers $\boldsymbol{\mu}^k$ are not assumed bounded here.

    \emph{(C1) Sign screen.} If $\lambda_j^k\to-\infty$ then
    \begin{equation}
    \begin{split}
    \phi_\varepsilon(\lambda_j^k,-g_j(\mathbf{x}^k))&=\sqrt{(\lambda_j^k)^2+g_j^2+\varepsilon^2}-\lambda_j^k+g_j\\
    &=-2\lambda_j^k+O(1)\to+\infty,
    \end{split}
    \end{equation}
    contradicting boundedness of the FB block. Hence on $\Omega_c$ every $\lambda_j$ is bounded below, and $\|\boldsymbol{\lambda}^k\|\to\infty$ forces the set $\mathcal{J}_+\coloneqq\{j:\lambda_j^k\to+\infty\}$ to be nonempty.

    \emph{(C2) FB limit.} For $j\in\mathcal{J}_+$, expanding~\eqref{eq:FB_smooth_gen} as $\lambda_j^k\to+\infty$ gives $\phi_\varepsilon(\lambda_j^k,-g_j(\mathbf{x}^k))\to g_j(\rev{\mathbf{x}_\infty})$.

    \emph{(C3) Gradient cancellation modulo $\operatorname{range}(\mathbf{A}^\top)$.} Let $\mathbf{P}$ be the orthogonal projection onto $\ker(\mathbf{A})$, so $\mathbf{P}\mathbf{A}^\top=\mathbf{0}$. Applying $\mathbf{P}$ to the first block removes the $\mathbf{A}^\top\boldsymbol{\mu}^k$ term irrespective of $\|\boldsymbol{\mu}^k\|$; boundedness of $\mathbf{P}\mathbf{s}_1^k$, $\nabla J(\mathbf{x}^k)$, and the terms $\nabla g_i(\mathbf{x}^k)\widetilde{\lambda}_i^k$ for $i\notin\mathcal{J}_+$ (whose multipliers stay bounded) then forces $\mathbf{P}\mathbf{D}^k$ to be bounded, where $\mathbf{D}^k\coloneqq\sum_{i\in\mathcal{J}_+}\widetilde{\lambda}_i^k\nabla g_i(\mathbf{x}^k)$. With $t^k\coloneqq\sum_{i\in\mathcal{J}_+}\widetilde{\lambda}_i^k\to+\infty$ and weights $\rev{\omega_i}^k\coloneqq\widetilde{\lambda}_i^k/t^k$ (a probability vector on $\mathcal{J}_+$), dividing by $t^k$ and passing to a further subsequence $\rev{\omega_i}^k\to \rev{\omega_i}\ge 0$, $\sum_{i\in\mathcal{J}_+}\rev{\omega_i}=1$, yields
    \begin{equation}
    \mathbf{P}\sum_{i\in\mathcal{J}_+}\rev{\omega_i}\nabla g_i(\rev{\mathbf{x}_\infty})=\mathbf{0},
    \end{equation}
    i.e., $\sum_{i\in\mathcal{J}_+}\rev{\omega_i}\nabla g_i(\rev{\mathbf{x}_\infty})=\mathbf{A}^\top\boldsymbol{\nu}$ for some $\boldsymbol{\nu}$. \rev{Fix $c_0>0$; for $c\le c_0$,} full row rank of $\mathbf{A}$ and continuity of the $\nabla g_i$ on the bounded $\mathbf{x}$-range of \rev{$\{V\le c_0\}$ given by} Stage~2 bound $\|\boldsymbol{\nu}\|\le\ell<\infty$\rev{, with $\ell$ independent of $c$}.

    \emph{(C4) Slater lower bound on the equality slice.} The map $\rev{g_\omega}\coloneqq\sum_{i\in\mathcal{J}_+}\rev{\omega_i} g_i$ is convex with $\nabla \rev{g_\omega}(\rev{\mathbf{x}_\infty})=\sum_{i\in\mathcal{J}_+}\rev{\omega_i}\nabla g_i(\rev{\mathbf{x}_\infty})=\mathbf{A}^\top\boldsymbol{\nu}$ by (C3). Let $\bar{\mathbf{x}}_s$ be a deepest Slater point, $\mathbf{A}\bar{\mathbf{x}}_s=\mathbf{b}$ and $\max_j g_j(\bar{\mathbf{x}}_s)=-d$, with $d\coloneqq-\min_{\mathbf{A}\mathbf{x}=\mathbf{b}}\max_j g_j(\mathbf{x})>0$ the feasibility margin. Convexity of $\rev{g_\omega}$ and $\mathbf{A}\bar{\mathbf{x}}_s=\mathbf{b}$ give
    \begin{align}
    &\rev{g_\omega}(\bar{\mathbf{x}}_s)-\rev{g_\omega}(\rev{\mathbf{x}_\infty})\ge\langle\nabla \rev{g_\omega}(\rev{\mathbf{x}_\infty}),\bar{\mathbf{x}}_s-\rev{\mathbf{x}_\infty}\rangle \notag\\ &\qquad=\langle\boldsymbol{\nu},\mathbf{A}(\bar{\mathbf{x}}_s-\rev{\mathbf{x}_\infty})\rangle=-\langle\boldsymbol{\nu},\rev{\mathbf{r}_\infty}\rangle,
    \end{align}
    where $\rev{\mathbf{r}_\infty}\coloneqq\mathbf{A}\rev{\mathbf{x}_\infty}-\mathbf{b}$.
    Since $\rev{g_\omega}(\bar{\mathbf{x}}_s)\le\max_{i\in\mathcal{J}_+}g_i(\bar{\mathbf{x}}_s)\le -d$, writing $\rho\coloneqq\|\rev{\mathbf{r}_\infty}\|$ and using $\|\boldsymbol{\nu}\|\le\ell$,
    \begin{equation}
    \rev{g_\omega}(\rev{\mathbf{x}_\infty})\le -d+\langle\boldsymbol{\nu},\rev{\mathbf{r}_\infty}\rangle\le -d+\ell\rho.
    \end{equation}
    \emph{(C5) Positive floor.} By (C2) the inequality block satisfies $\sum_{i\in\mathcal{J}_+}\phi_\varepsilon^2\to\sum_{i\in\mathcal{J}_+}g_i(\rev{\mathbf{x}_\infty})^2$, and Cauchy--Schwarz with the weights $\rev{\boldsymbol{\omega}}$ gives $\sum_{i\in\mathcal{J}_+}g_i(\rev{\mathbf{x}_\infty})^2\ge \rev{g_\omega}(\rev{\mathbf{x}_\infty})^2$. Retaining also the equality block $\mathbf{s}_3=\mathbf{A}\mathbf{x}-\mathbf{b}$ with $\|\rev{\mathbf{r}_\infty}\|=\rho$,
    \begin{align}
    \liminf_k V(\mathbf{z}^k)&\ge\tfrac{1}{2}\big(\rev{g_\omega}(\rev{\mathbf{x}_\infty})^2+\rho^2\big) \notag\\
    &\ge\tfrac{1}{2}\big((d-\ell\rho)_+^2+\rho^2\big)\\
    &\ge\tfrac{1}{2}\min_{\rho\ge 0}\big[(d-\ell\rho)^2+\rho^2\big]=\frac{d^2}{2(1+\ell^2)}>0. \notag
    \end{align}
    Hence there exists $c^\ast>0$ (any value in \rev{$(0,\,\min\{c_0,\,d^2/(2(1+\ell^2))\}]$}) such that no $\boldsymbol{\lambda}$-escape sequence satisfies $V\le c<c^\ast$; on $\Omega_c$ with $c<c^\ast$, $\|\boldsymbol{\lambda}\|$ cannot escape to infinity.

    Combining Stages~1--3, no escape sequence exists on $\{V\le c\}$, so $\Omega_c$ is bounded. Closedness plus boundedness in $\mathbb{R}^{n+p+q}$ yields compactness.
    \end{proof}

    \subsection{Proofs of Proposition~\ref{prop:affine_global} and Corollary~\ref{cor:auto_nonsing_general}}
    \label{app:coro_proof}

    \begin{proof}
    \emph{Uniqueness.} Under strong convexity of $J$ and Slater's condition, $\mathbf{x}^\star$ is the unique global minimizer of~\eqref{eq:constrained_problem_general}~\cite[Thm.~11.12]{beck_introduction_2014}; LICQ further yields uniqueness of the associated multipliers $(\boldsymbol{\lambda}^\star,\boldsymbol{\mu}^\star)$. Hence the KKT triple $(\mathbf{x}^\star,\boldsymbol{\lambda}^\star,\boldsymbol{\mu}^\star)$ is unique.
    
    \emph{Nonsingularity of $\nabla\mathbf{S}$ via block elimination.}

    Step 1: FB sign convention.
    For each $i$ and every $\varepsilon>0$, the partial derivatives of $\phi_\varepsilon$ at
    $(a,b)=(\lambda_i,-g_i(\mathbf{x}))$,
    \begin{equation}
    \begin{aligned}
    \rev{d_{a,i}}&=\frac{\lambda_i}{\sqrt{\lambda_i^2+g_i^2+\varepsilon^2}}-1,\\
    \rev{d_{b,i}}&=\frac{-g_i}{\sqrt{\lambda_i^2+g_i^2+\varepsilon^2}}-1,
    \end{aligned}
    \end{equation}
    lie strictly in $(-2,0)$. Hence $\mathbf{D}_a\coloneqq\operatorname{diag}(\rev{d_{a,i}})\prec\mathbf{0}$ and
    $\mathbf{D}_b\coloneqq\operatorname{diag}(\rev{d_{b,i}})\prec\mathbf{0}$ are strictly negative diagonal, and both
    are invertible.

    Step 2: Jacobian block structure.
    Differentiating~\eqref{eq:S_FB} blockwise gives
    \begin{equation}
    \nabla\mathbf{S}(\mathbf{z})=\begin{bmatrix}
    \mathbf{H}_{\mathcal{L}}(\mathbf{x},\boldsymbol{\lambda}) & \nabla\mathbf{g}(\mathbf{x})^\top\mathbf{D}_{\widetilde{\lambda}'} & \mathbf{A}^\top\\[2pt]
    -\mathbf{D}_b\,\nabla\mathbf{g}(\mathbf{x}) & \mathbf{D}_a & \mathbf{0}\\[2pt]
    \mathbf{A} & \mathbf{0} & \mathbf{0}
    \end{bmatrix},
    \end{equation}
    where $\mathbf{D}_{\widetilde{\lambda}'}=\operatorname{diag}(\widetilde{\lambda}_i')$ with
    $\widetilde{\lambda}_i'\in(0,1)$ from~\eqref{eq:lambda_tilde}, and
    $\mathbf{H}_{\mathcal{L}}(\mathbf{x},\boldsymbol{\lambda})\coloneqq\nabla^2 J(\mathbf{x})+\sum_{i=1}^{p}\widetilde{\lambda}_i\nabla^2 g_i(\mathbf{x})$.
    Strong convexity gives $\nabla^2 J\succeq m\mathbf{I}$; convexity of each $g_i$ gives
    $\nabla^2 g_i\succeq\mathbf{0}$, and $\widetilde{\lambda}_i>0$ by construction~\eqref{eq:lambda_tilde},
    so $\sum_i\widetilde{\lambda}_i\nabla^2 g_i\succeq\mathbf{0}$ unconditionally. Hence
    \begin{equation}
    \mathbf{v}_1^{\!\top}\mathbf{H}_{\mathcal{L}}(\mathbf{x},\boldsymbol{\lambda})\,\mathbf{v}_1\;\ge\;m\,\|\mathbf{v}_1\|^2,
    \qquad\forall \mathbf{v}_1\in\mathbb{R}^{n},
    \label{eq:HL_coercive_con}
    \end{equation}
    with no sign assumption on $\boldsymbol{\lambda}$.

    Step 3: Block elimination.
    Let $\mathbf{v}=\operatorname{col}(\mathbf{v}_1,\mathbf{v}_2,\mathbf{v}_3)\in\mathbb{R}^{n}\times\mathbb{R}^{p}\times\mathbb{R}^{q}$
    satisfy $\nabla\mathbf{S}(\mathbf{z})\,\mathbf{v}=\mathbf{0}$. The three block equations are
    \begin{subequations}
    \begin{align}
    \mathbf{H}_{\mathcal{L}}\,\mathbf{v}_1+\nabla\mathbf{g}^\top\mathbf{D}_{\widetilde{\lambda}'} \mathbf{v}_2+\mathbf{A}^\top \mathbf{v}_3 &=\mathbf{0},\label{eq:cbk1}\\
    -\mathbf{D}_b\,\nabla\mathbf{g}\,\mathbf{v}_1+\mathbf{D}_a\,\mathbf{v}_2 &=\mathbf{0},\label{eq:cbk2}\\
    \mathbf{A}\,\mathbf{v}_1 &=\mathbf{0}.\label{eq:cbk3}
    \end{align}
    \end{subequations}
    Equation~\eqref{eq:cbk2}, using invertibility of $\mathbf{D}_a$, yields
    $\mathbf{v}_2=\mathbf{D}_a^{-1}\mathbf{D}_b\,\nabla\mathbf{g}\,\mathbf{v}_1$. Substituting into~\eqref{eq:cbk1} and
    taking the inner product with $\mathbf{v}_1$, while using~\eqref{eq:cbk3} to eliminate the
    $\mathbf{A}^\top \mathbf{v}_3$ term, gives
    \begin{equation}
    \mathbf{v}_1^{\!\top}\Bigl(\mathbf{H}_{\mathcal{L}}+\nabla\mathbf{g}^\top\mathbf{D}_{\widetilde{\lambda}'}\mathbf{D}_a^{-1}\mathbf{D}_b\,\nabla\mathbf{g}\Bigr)\mathbf{v}_1=0.
    \label{eq:cHhat}
    \end{equation}

    Step 4: Positive-semidefinite argument.
    Since $\mathbf{D}_a,\mathbf{D}_b\prec\mathbf{0}$ are strictly negative diagonal,
    $\mathbf{D}_a^{-1}\mathbf{D}_b$ is strictly positive diagonal; with
    $\mathbf{D}_{\widetilde{\lambda}'}\succ\mathbf{0}$ this gives
    $\mathbf{D}_{\widetilde{\lambda}'}\mathbf{D}_a^{-1}\mathbf{D}_b\succ\mathbf{0}$, hence
    $\nabla\mathbf{g}^\top\mathbf{D}_{\widetilde{\lambda}'}\mathbf{D}_a^{-1}\mathbf{D}_b\,\nabla\mathbf{g}\succeq\mathbf{0}$.
    Combined with~\eqref{eq:HL_coercive_con},~\eqref{eq:cHhat} yields
    \begin{equation}
    \begin{aligned}0&=\mathbf{v}_1^{\!\top}\Bigl(\mathbf{H}_{\mathcal{L}}+\nabla\mathbf{g}^\top\mathbf{D}_{\widetilde{\lambda}'}\mathbf{D}_a^{-1}\mathbf{D}_b\,\nabla\mathbf{g}\Bigr)\mathbf{v}_1\\
    &\ge \mathbf{v}_1^{\!\top}\mathbf{H}_{\mathcal{L}}\mathbf{v}_1\ge m\|\mathbf{v}_1\|^2.\end{aligned}
    \end{equation}
    Since $m>0$, this forces $\mathbf{v}_1=\mathbf{0}$.

    Step 5: Back-substitution.
    With $\mathbf{v}_1=\mathbf{0}$, \eqref{eq:cbk2} gives $\mathbf{D}_a \mathbf{v}_2=\mathbf{0}$, hence
    $\mathbf{v}_2=\mathbf{0}$ by invertibility of $\mathbf{D}_a$; \eqref{eq:cbk1} then reduces to
    $\mathbf{A}^\top \mathbf{v}_3=\mathbf{0}$, hence $\mathbf{v}_3=\mathbf{0}$ by full row rank of $\mathbf{A}$. Therefore
    the only solution of $\nabla\mathbf{S}(\mathbf{z})\,\mathbf{v}=\mathbf{0}$ is $\mathbf{v}=\mathbf{0}$, so
    $\nabla\mathbf{S}(\mathbf{z})$ is nonsingular at every $\mathbf{z}\in\mathbb{R}^{n+p+q}$. In particular
    $\nabla V(\mathbf{z})=\nabla\mathbf{S}(\mathbf{z})^\top\mathbf{S}(\mathbf{z})\neq\mathbf{0}$ whenever
    $\mathbf{S}(\mathbf{z})\neq\mathbf{0}$, so Assumption~\ref{ass:OLF_reg} holds automatically with
    $\mathcal{U}=\mathbb{R}^{n+p+q}$.

    \emph{Pointwise nonsingularity versus a uniform bound.} The argument shows
    $\sigma_{\min}(\nabla\mathbf{S}(\mathbf{z}))>0$ at every finite $\mathbf{z}$, but does not provide a
    uniform bound $\inf_{\mathbf{z}}\sigma_{\min}(\nabla\mathbf{S}(\mathbf{z}))>0$. As any
    $\lambda_j\to+\infty$ the derivative $\rev{d_{a,j}}\to 0$, so $\mathbf{D}_a$ approaches singularity and
    $\sigma_{\min}(\nabla\mathbf{S}(\mathbf{z}))$ \rev{can approach $0$ along such sequences}. The two are not in
    conflict: nonsingularity is a pointwise property at each finite state, whereas the absence of a
    uniform bound concerns the limit at infinity. This is why the constant in
    Lemma~\ref{lem:lin_PL} is taken as $\sigma_c=\min_{\mathbf{z}\in\Omega_c}\sigma_{\min}(\nabla\mathbf{S}(\mathbf{z}))>0$
    over the compact $\Omega_c$ rather than over $\mathbb{R}^{n+p+q}$: nonsingularity supplies the PL
    constant on each $\Omega_c$, while compactness of $\Omega_c$ is the separate property secured by
    the threshold $c^\ast$ of Prop.~\ref{prop:c_star}.
    
    It remains to verify that the hypotheses of Thm.~\ref{thm:unified_constrained} are met \rev{on $\Omega_c$ for $c<c^\ast$}. By Steps~1--5 above, $\nabla\mathbf{S}(\mathbf{z})$ is nonsingular at every $\mathbf{z}\in\mathbb{R}^{n+p+q}$, so Assumption~\ref{ass:OLF_reg} holds on $\mathbb{R}^{n+p+q}$. Fix any $c\in(0,c^\ast)$ with $c^\ast>0$ the threshold of Prop.~\ref{prop:c_star}, which yields compactness of $\Omega_c$, and Lemma~\ref{lem:lin_PL} then establishes Assumption~\ref{ass:lojasiewicz} on $\Omega_c$ with $\alpha=1/2$ and constant $c_{\mathrm{L}}=\sigma_c\sqrt{2}$ where $\sigma_c=\min_{\mathbf{z}\in\Omega_c}\sigma_{\min}(\nabla\mathbf{S}(\mathbf{z}))>0$. Hence Lemma~\ref{lem:fwd_inv} applies for every $c\in(0,c^\ast)$, yielding forward invariance of $\Omega_c$; for $\mathbf{z}(0)$ with $V(\mathbf{z}(0))\rev{\le c}<c^\ast$, the trajectory remains in $\Omega_c\subseteq\mathcal{U}=\mathbb{R}^{n+p+q}$ \rev{until it reaches $\mathcal{S}$}, so $V(\mathbf{z}(t))\to 0$ and hence $\mathbf{S}(\mathbf{z}(t))\to\mathbf{0}$ with the chosen \rev{decay} law.

    Since $\mathbf{z}(t)$ remains in the compact set $\Omega_c$, \rev{the set $\mathcal{E}$ of its limit points} is nonempty, compact, and connected. From $\dot V=-\sigma(V,t)\le 0$ and $V(\mathbf{z}(t))\to 0$, every point of $\rev{\mathcal{E}}$ is a zero of $\mathbf{S}$. By Steps~1--5, $\nabla\mathbf{S}$ is nonsingular, so the zeros of $\mathbf{S}$ are isolated; a connected subset of an isolated set is a singleton. Hence $\rev{\mathcal{E}}=\{\mathbf{z}^\star_\varepsilon\}$ and $\mathbf{z}(t)\to\mathbf{z}^\star_\varepsilon$, the zero of $\mathbf{S}$. By Lemma~\ref{lem:smoothed_point}\rev{(iv)}, $\|\mathbf{z}^\star_\varepsilon-(\mathbf{x}^\star,\boldsymbol{\lambda}^\star,\boldsymbol{\mu}^\star)\|=O(\varepsilon)$. 
    
    \emph{Affine specialization: uniform Jacobian regularity and global convergence.} When $\mathbf{g}(\mathbf{x})=\mathbf{C}\mathbf{x}-\mathbf{d}$ is affine with $\mathbf{C}\in\mathbb{R}^{p\times n}$ of full row rank, the pointwise-vs-uniform tradeoff above is broken in a structural way. Two effects combine. First, $\nabla^2 g_i\equiv\mathbf{0}$ for each $i$, so the Lagrangian-Hessian block collapses to $\mathbf{H}_{\mathcal{L}}=\nabla^2 J\succeq m\mathbf{I}$ uniformly on $\mathbb{R}^{n+p+q}$, with no multiplier-dependent contribution. Second, we show $\inf_{\mathbf{z}}\sigma_{\min}(\nabla\mathbf{S}(\mathbf{z}))\ge\bar\sigma>0$ when $\big[\begin{smallmatrix}\mathbf{A}\\ \mathbf{C}\end{smallmatrix}\big]$ has full row rank.

    Suppose not: take $\mathbf{z}^k$ and unit $\mathbf{v}^k=\operatorname{col}(\mathbf{v}_1^k,\mathbf{v}_2^k,\mathbf{v}_3^k)$ with $\nabla\mathbf{S}(\mathbf{z}^k)\mathbf{v}^k\to\mathbf{0}$. The diagonals $\rev{d_{a,j}},\rev{d_{b,j}}\in(-2,0)$ and $\widetilde{\lambda}_j'\in(0,1)$ lie in compact intervals; pass to a subsequence with $\rev{d_{a,j}^k}\to\rev{d_{a,j}^{\infty}}$, $\rev{d_{b,j}^k}\to\rev{d_{b,j}^{\infty}}$, $\widetilde{\lambda}_j'^k\to\ell_j$, and $\mathbf{v}^k\to\mathbf{v}$. Set $\mathcal{J}_0=\{j:\rev{d_{a,j}^{\infty}}=0\}$. The third block gives $\mathbf{A}\mathbf{v}_1=\mathbf{0}$. Taking the inner product of the first block with $\mathbf{v}_1^k$ (so $\langle\mathbf{A}^\top\mathbf{v}_3^k,\mathbf{v}_1^k\rangle=\langle\mathbf{v}_3^k,\mathbf{A}\mathbf{v}_1^k\rangle\to 0$, and $\mathbf{H}_{\mathcal{L}}\succeq m\mathbf{I}$): for $j\notin\mathcal{J}_0$ the second block gives a contribution $\ell_j(\rev{d_{b,j}^{\infty}}/\rev{d_{a,j}^{\infty}})(\mathbf{c}_j^\top\mathbf{v}_1)^2\ge 0$ (a ratio of two negatives); for $j\in\mathcal{J}_0$, where $\rev{d_{b,j}^{\infty}}=-1$ and $\ell_j=1$, the second block forces $\mathbf{c}_j^\top\mathbf{v}_1=0$, so its term vanishes. Hence $0\ge m\|\mathbf{v}_1\|^2$, so $\mathbf{v}_1=\mathbf{0}$\rev{; since $\|\nabla^2 J\|$ is bounded by the Lipschitz constant of $\nabla J$, also $\nabla^2 J(\mathbf{x}^k)\,\mathbf{v}_1^k\to\mathbf{0}$}. Then $v_{2,j}=0$ for $j\notin\mathcal{J}_0$, and the first block reduces to $\sum_{j\in\mathcal{J}_0}v_{2,j}\,\mathbf{c}_j+\mathbf{A}^\top\mathbf{v}_3=\mathbf{0}$. Full row rank of $\big[\begin{smallmatrix}\mathbf{A}\\ \mathbf{C}\end{smallmatrix}\big]$ makes the rows $\{\mathbf{c}_j\}_{j\in\mathcal{J}_0}$ together with the rows of $\mathbf{A}$ linearly independent, so $v_{2,j}=0$ and $\mathbf{v}_3=\mathbf{0}$; thus $\mathbf{v}=\mathbf{0}$, contradicting $\|\mathbf{v}\|=1$. At most one of $\rev{d_{a,j}},\rev{d_{b,j}}$ tends to $0$ since $\rev{d_{a,j}}+\rev{d_{b,j}}\le\sqrt{2}-2<0$, and $\rev{d_{b,j}}\to 0$ leaves $\rev{d_{a,j}}$ bounded away from $0$, so $\mathcal{J}_0$ is the only degenerate set. Hence $\sigma_{\min}(\nabla\mathbf{S})\ge\bar\sigma>0$ uniformly on $\mathbb{R}^{n+p+q}$. The condition $\big[\begin{smallmatrix}\mathbf{A}\\ \mathbf{C}\end{smallmatrix}\big]$ full row rank \rev{cannot simply be dropped: for instance,} if some $\mathbf{c}_j\in\operatorname{range}(\mathbf{A}^\top)$, the limit vector $\mathbf{v}=(\mathbf{0},\mathbf{v}_2,\mathbf{v}_3)$ above is nonzero, and $\sigma_{\min}(\nabla\mathbf{S})\to 0$ along the corresponding ray.

    Lemma~\ref{lem:lin_PL} then applies globally with $\sigma_c$ replaced by $\bar\sigma$, yielding $\|\nabla V(\mathbf{z})\|^2\ge 2\bar\sigma^2 V(\mathbf{z})$ for all $\mathbf{z}\in\mathbb{R}^{n+p+q}$. Combined with $\dot V=-\sigma(V,t)$ along the closed-loop dynamics and forward invariance of every sublevel set\rev{, and with the finite length of the trajectories that this inequality implies (Appendix~\ref{app:fwd_inv_proof})}, this gives global convergence $\mathbf{z}(t)\to\mathbf{z}^\star_\varepsilon$ for every initial condition, irrespective of $c^\ast$, with $\|\mathbf{z}^\star_\varepsilon-(\mathbf{x}^\star,\boldsymbol{\lambda}^\star,\boldsymbol{\mu}^\star)\|=O(\varepsilon)$ by Lemma~\ref{lem:smoothed_point}\rev{(i),(iv), since full row rank gives Slater's condition and LICQ}.
    \end{proof}

    \subsection{\rev{Proof of Lemma~\ref{lem:smoothed_point}}}
    \label{app:eps_displacement}

    \rev{The proof covers the three problem classes at once: $\mathbf{w}$, $\mathcal{F}$, and $\mathbf{g}$ stand for $\mathbf{x}$, $\nabla J$, and $\mathbf{g}$ in~\eqref{eq:constrained_problem_general}, for $\operatorname{col}(\mathbf{x},\mathbf{y})$, $\operatorname{col}(\nabla_{\mathbf{x}}J,-\nabla_{\mathbf{y}}J)$, and $\mathbf{G}$ in~\eqref{eq:minimax}, and for $\mathbf{x}$, $\mathcal{G}$, and $\mathbf{g}$ in~\eqref{eq:gne_problem_general}.} \rev{Write $\widetilde{\mathbf{A}}$ for the equality matrix ($\mathbf{A}$, or $[\,\mathbf{A}\ \mathbf{B}\,]$ in Sec.~\ref{sec:minmax}), $\mathcal{W}\coloneqq\{\mathbf{w}:\widetilde{\mathbf{A}}\mathbf{w}=\mathbf{b},\ \mathbf{g}(\mathbf{w})\le\mathbf{0}\}$, $\mathcal{W}^\circ\coloneqq\{\mathbf{w}\in\mathcal{W}:\mathbf{g}(\mathbf{w})<\mathbf{0}\}$, and $\widetilde{\lambda}(\lambda)\coloneqq\tfrac{1}{2}(\sqrt{\lambda^2+\varepsilon^2}+\lambda)$ as in~\eqref{eq:lambda_tilde}. Since $\mathcal{F}$ is strongly monotone, the VI with operator $\mathcal{F}$ on $\mathcal{W}$ has a unique solution $\mathbf{w}^\star$~\cite[Thm.~2.3.3]{facchinei_finite-dimensional_2003}: the minimizer of~\eqref{eq:constrained_problem_general}, the saddle point of Sec.~\ref{sec:minmax}, or the v-GNE of Sec.~\ref{sec:gne}.}

    \rev{(i) and (ii). For $a,b\in\mathbb{R}$, $\phi_\varepsilon(a,b)=0$ if and only if $a+b\ge 0$ and $a^2+b^2+\varepsilon^2=(a+b)^2$, that is, $a>0$, $b>0$, and $ab=\varepsilon^2/2$. Hence $(\mathbf{w},\boldsymbol{\lambda},\boldsymbol{\mu})$ is a zero of $\mathbf{S}$ if and only if $\mathbf{w}\in\mathcal{W}^\circ$, $\lambda_j=\varepsilon^2/(-2g_j(\mathbf{w}))$ for every $j$, and $\mathcal{F}_\varepsilon(\mathbf{w})+\widetilde{\mathbf{A}}^\top\boldsymbol{\mu}=\mathbf{0}$, where
    \begin{equation}
    \mathcal{F}_\varepsilon(\mathbf{w})\coloneqq\mathcal{F}(\mathbf{w})+\sum_j\psi(-g_j(\mathbf{w}))\nabla g_j(\mathbf{w})
    \end{equation}
    and $\psi(s)\coloneqq\widetilde{\lambda}(\varepsilon^2/(2s))$ for $s>0$. This gives (ii) at any zero. Since $\psi$ is positive and decreasing, $\sum_j\psi(-g_j)\nabla g_j$ is the gradient on $\mathcal{W}^\circ$ of the convex function $\sum_j\Psi(-g_j(\mathbf{w}))$, where $\Psi(s)\coloneqq\int_s^1\psi(\tau)\,d\tau$, so $\mathcal{F}_\varepsilon$ is $m$-strongly monotone on $\mathcal{W}^\circ$. Two zeros therefore satisfy
    \begin{equation}
    \begin{split}
    m\|\mathbf{w}^1-\mathbf{w}^2\|^2&\le\langle\mathcal{F}_\varepsilon(\mathbf{w}^1)-\mathcal{F}_\varepsilon(\mathbf{w}^2),\mathbf{w}^1-\mathbf{w}^2\rangle\\
    &=-\langle\boldsymbol{\mu}^1-\boldsymbol{\mu}^2,\widetilde{\mathbf{A}}(\mathbf{w}^1-\mathbf{w}^2)\rangle=0,
    \end{split}
    \end{equation}
    so $\mathbf{w}^1=\mathbf{w}^2$; then $\boldsymbol{\lambda}^1=\boldsymbol{\lambda}^2$ by the formula above, and $\boldsymbol{\mu}^1=\boldsymbol{\mu}^2$ by full row rank.}

    \rev{For existence, let $\bar{\mathbf{w}}\in\mathcal{W}^\circ$ be a Slater point with $\mathbf{g}(\bar{\mathbf{w}})\le-d\,\mathbf{1}$, $d>0$, and for $\delta\in(0,d)$ let $\mathbf{w}_\delta$ solve the VI with operator $\mathcal{F}_\varepsilon$ on the closed convex set $\mathcal{W}_\delta\coloneqq\{\mathbf{w}:\widetilde{\mathbf{A}}\mathbf{w}=\mathbf{b},\ \mathbf{g}(\mathbf{w})\le-\delta\mathbf{1}\}\ni\bar{\mathbf{w}}$, which exists by~\cite[Thm.~2.3.3]{facchinei_finite-dimensional_2003}. Testing the VI with $\bar{\mathbf{w}}$ and using strong monotonicity gives $\|\mathbf{w}_\delta-\bar{\mathbf{w}}\|\le\|\mathcal{F}_\varepsilon(\bar{\mathbf{w}})\|/m$, so $\mathbf{w}_\delta$ stays in a fixed ball, on which $\langle\mathcal{F}(\mathbf{w}_\delta),\bar{\mathbf{w}}-\mathbf{w}_\delta\rangle$ is bounded. With $s_i\coloneqq-g_i(\mathbf{w}_\delta)$ and $\bar s_i\coloneqq-g_i(\bar{\mathbf{w}})\ge d$, convexity gives
    \begin{equation}
    \psi(s_i)\langle\nabla g_i(\mathbf{w}_\delta),\bar{\mathbf{w}}-\mathbf{w}_\delta\rangle\le\psi(s_i)(s_i-\bar s_i);
    \end{equation}
    this is bounded above on the ball uniformly in $\delta$, because $s\psi(s)\le\varepsilon^2/2+\varepsilon s/2$, and it is at most $-\varepsilon^2(d-\delta)/(2\delta)$ if $s_i=\delta$, because $\psi(s)\ge\varepsilon^2/(2s)$. Hence, for small $\delta$, no constraint of $\mathcal{W}_\delta$ is active at $\mathbf{w}_\delta$, since otherwise $\langle\mathcal{F}_\varepsilon(\mathbf{w}_\delta),\bar{\mathbf{w}}-\mathbf{w}_\delta\rangle<0$, contradicting the VI. Then $\mathbf{w}_\delta$ is interior to $\mathcal{W}_\delta$ relative to $\{\widetilde{\mathbf{A}}\mathbf{w}=\mathbf{b}\}$, the VI gives $\mathcal{F}_\varepsilon(\mathbf{w}_\delta)\perp\ker(\widetilde{\mathbf{A}})$, so $\mathcal{F}_\varepsilon(\mathbf{w}_\delta)=-\widetilde{\mathbf{A}}^\top\boldsymbol{\mu}$ for some $\boldsymbol{\mu}$, and $\mathbf{w}_\delta$ yields a zero of $\mathbf{S}$. This proves (i).}

    \rev{(iii). Write $\mathbf{z}^\star_\varepsilon=(\mathbf{w}_\varepsilon,\boldsymbol{\lambda},\boldsymbol{\mu})$, $\Delta\coloneqq\|\mathbf{w}_\varepsilon-\mathbf{w}^\star\|$, and $\boldsymbol{\theta}\coloneqq\nabla\mathbf{g}(\mathbf{w}_\varepsilon)^\top(\widetilde{\boldsymbol{\lambda}}-\boldsymbol{\lambda})$, so $\|\boldsymbol{\theta}\|\le\|\nabla\mathbf{g}(\mathbf{w}_\varepsilon)\|\sqrt{p}\,\varepsilon/2$ because $0<\widetilde{\lambda}_j-\lambda_j\le\varepsilon/2$. Since $\mathbf{w}_\varepsilon-\mathbf{w}^\star\in\ker(\widetilde{\mathbf{A}})$, $\mathbf{w}_\varepsilon\in\mathcal{W}$, and $\mathbf{w}^\star$ solves the VI, the stationarity block of $\mathbf{S}$ gives
    \begin{align}
    m\Delta^2&\le\langle\mathcal{F}(\mathbf{w}_\varepsilon)-\mathcal{F}(\mathbf{w}^\star),\mathbf{w}_\varepsilon-\mathbf{w}^\star\rangle \notag\\
    &\le\textstyle\sum_j\lambda_j\langle\nabla g_j(\mathbf{w}_\varepsilon),\mathbf{w}^\star-\mathbf{w}_\varepsilon\rangle+\|\boldsymbol{\theta}\|\,\Delta \notag\\
    &\le p\varepsilon^2/2+\|\boldsymbol{\theta}\|\,\Delta,
    \end{align}
    where the last step uses convexity, $g_j(\mathbf{w}^\star)\le 0$, and $\lambda_j g_j(\mathbf{w}_\varepsilon)=-\varepsilon^2/2$. Solving the quadratic inequality gives
    \begin{equation}
    \Delta\le\|\boldsymbol{\theta}\|/m+\varepsilon\sqrt{p/(2m)},
    \end{equation}
    which is (iii). The same chain with $\widetilde{\lambda}_j$ in place of $\lambda_j$, together with $\widetilde{\lambda}_j(-g_j)\le\varepsilon^2/2+(\varepsilon/2)(-g_j)$ and $-g_j(\mathbf{w}_\varepsilon)\le-g_j(\mathbf{w}^\star)+\|\nabla g_j(\mathbf{w}^\star)\|\,\Delta$, gives
    \begin{equation}
    m\Delta^2\le(\varepsilon/2)\big(p\varepsilon+\|\mathbf{g}(\mathbf{w}^\star)\|_1+\kappa_1\Delta\big)
    \end{equation}
    with $\kappa_1\coloneqq\sum_j\|\nabla g_j(\mathbf{w}^\star)\|$. The positive root $\Delta_1(\varepsilon)$ of this quadratic increases with $\varepsilon$ and vanishes as $\varepsilon\to 0$, so $\mathbf{w}_\varepsilon\to\mathbf{w}^\star$; moreover, for $\varepsilon\le 1$, (iii) holds with $\|\nabla\mathbf{g}(\mathbf{w}_\varepsilon)\|$ replaced by the maximum of $\|\nabla\mathbf{g}\|$ over the ball of radius $\Delta_1(1)$ about $\mathbf{w}^\star$, so $\Delta=O(\varepsilon)$.}

    \rev{(iv). Under Assumption~\ref{ass:LICQ_general}, the multipliers $(\boldsymbol{\lambda}^\star,\boldsymbol{\mu}^\star)$ of $\mathbf{w}^\star$, which exist by Slater's condition, are unique. Let $\mathcal{A}$ and $\mathcal{I}$ be the active and inactive sets at $\mathbf{w}^\star$, $s_{\mathcal{I}}\coloneqq\min_{j\in\mathcal{I}}(-g_j(\mathbf{w}^\star))$, and $\sigma_{\mathcal{A}}\coloneqq\sigma_{\min}\big([\nabla\mathbf{g}_{\mathcal{A}}(\mathbf{w}^\star)^\top\ \widetilde{\mathbf{A}}^\top]\big)>0$, and choose $r>0$ such that, on the ball $\mathcal{B}$ of radius $r$ about $\mathbf{w}^\star$, $\mathcal{F}$ and $\nabla\mathbf{g}$ are Lipschitz with constants $L_F$ and $L_g$, $\|\nabla\mathbf{g}\|\le\kappa_g$, and $g_j\le-s_{\mathcal{I}}/2$ for $j\in\mathcal{I}$; set
    \begin{equation}
    \kappa_w\coloneqq\sqrt{p}\,\big(\kappa_g/(2m)+1/\sqrt{2m}\big),
    \end{equation}
    and let $\varepsilon_1$ be the largest $\varepsilon\le 1$ with $\Delta_1(\varepsilon)\le r$, so that $\mathbf{w}_\varepsilon\in\mathcal{B}$ for $\varepsilon\le\varepsilon_1$.}

    \rev{Let
    \begin{equation}
    \varepsilon\le\varepsilon_0\coloneqq\min\{1,\,s_{\mathcal{I}}/2,\,\varepsilon_1,\,\sigma_{\mathcal{A}}/(2L_g\kappa_w)\},
    \end{equation}
    dropping the terms with $\mathcal{I}=\emptyset$ or $L_g=0$; if $\mathcal{A}=\emptyset$ and there are no equality constraints, the matrix defining $\sigma_{\mathcal{A}}$ has no columns, and $\boldsymbol{\zeta}$ below and all terms with $\sigma_{\mathcal{A}}$ are dropped. Then $\Delta\le\kappa_w\varepsilon$ by (iii), and (ii) gives $\lambda_j\le\varepsilon^2/s_{\mathcal{I}}$ and $\widetilde{\lambda}_j\le\varepsilon$ for $j\in\mathcal{I}$. With $\boldsymbol{\zeta}\coloneqq(\widetilde{\boldsymbol{\lambda}}_{\mathcal{A}}-\boldsymbol{\lambda}^\star_{\mathcal{A}},\,\boldsymbol{\mu}-\boldsymbol{\mu}^\star)$, subtracting the stationarity blocks at $\mathbf{z}^\star_\varepsilon$ and $\mathbf{z}^\star$ gives
    \begin{align}
    [\nabla\mathbf{g}_{\mathcal{A}}(\mathbf{w}^\star)^\top\ \widetilde{\mathbf{A}}^\top]\,\boldsymbol{\zeta}={}&\mathcal{F}(\mathbf{w}^\star)-\mathcal{F}(\mathbf{w}_\varepsilon)-\nabla\mathbf{g}_{\mathcal{I}}(\mathbf{w}^\star)^\top\widetilde{\boldsymbol{\lambda}}_{\mathcal{I}} \notag\\
    &-\big(\nabla\mathbf{g}(\mathbf{w}_\varepsilon)-\nabla\mathbf{g}(\mathbf{w}^\star)\big)^\top\widetilde{\boldsymbol{\lambda}},
    \end{align}
    so
    \begin{equation}
    \begin{split}
    \sigma_{\mathcal{A}}\|\boldsymbol{\zeta}\|\le{}&L_F\Delta+L_g\Delta\,(\|\boldsymbol{\lambda}^\star\|+\|\boldsymbol{\zeta}\|+\sqrt{p}\,\varepsilon)\\
    &{}+\sqrt{p}\,\|\nabla\mathbf{g}(\mathbf{w}^\star)\|\,\varepsilon
    \end{split}
    \end{equation}
    and, since $L_g\Delta\le\sigma_{\mathcal{A}}/2$,
    \begin{equation}
    \begin{split}
    \|\boldsymbol{\zeta}\|\le{}&(2\varepsilon/\sigma_{\mathcal{A}})\big[\kappa_w\big(L_F+L_g(\|\boldsymbol{\lambda}^\star\|+\sqrt{p})\big)\\
    &{}+\sqrt{p}\,\|\nabla\mathbf{g}(\mathbf{w}^\star)\|\big].
    \end{split}
    \end{equation}
    Finally, $\|\boldsymbol{\lambda}-\boldsymbol{\lambda}^\star\|\le\sqrt{p}\,\varepsilon/2+\|\boldsymbol{\zeta}\|+\sqrt{p}\,\varepsilon$, so $\|\mathbf{z}^\star_\varepsilon-\mathbf{z}^\star\|\le\kappa\,\varepsilon$ with
    \begin{equation}
    \begin{split}
    \kappa\coloneqq{}&\kappa_w+\tfrac{3}{2}\sqrt{p}\\
    &{}+(4/\sigma_{\mathcal{A}})\big[\kappa_w\big(L_F+L_g(\|\boldsymbol{\lambda}^\star\|+\sqrt{p})\big)\\
    &{}+\sqrt{p}\,\|\nabla\mathbf{g}(\mathbf{w}^\star)\|\big].
    \end{split}
    \end{equation}}

    \subsection{Proof of Theorem~\ref{thm:minimax_main}}
    \label{app:thm_minimax_proof}

    \begin{proof}
    The argument proceeds in three steps. (i) Existence of a saddle point. The VI form of~\eqref{eq:minimax} with monotone operator $\mathcal{F}(\mathbf{w})=\operatorname{col}(\nabla_\mathbf{x} J,-\nabla_\mathbf{y} J)$ admits a saddle point on the compact convex feasible set by~\cite[Cor.~2.2.5]{facchinei_finite-dimensional_2003} (monotonicity of $\mathcal{F}$ follows from convex-concavity of $J$, and the feasible set is compact convex under the hypotheses of Thm.~\ref{thm:minimax_main}); Slater's condition (Assumption~\ref{ass:Slater_general}) ensures existence of associated KKT multipliers $(\boldsymbol{\lambda}^\star,\boldsymbol{\mu}^\star)$. \rev{Every zero $(\mathbf{w},\boldsymbol{\lambda},\boldsymbol{\mu})$ of $\mathbf{S}$ has $\mathbf{G}(\mathbf{w})<\mathbf{0}$ and $\lambda_j=\varepsilon^2/(-2G_j(\mathbf{w}))$, as in the proof of Lemma~\ref{lem:smoothed_point}(ii), so $\mathbf{w}$ lies in the compact feasible set. Testing its stationarity block with a Slater point $\bar{\mathbf{w}}$, $\mathbf{G}(\bar{\mathbf{w}})\le-d\,\mathbf{1}$, and using convexity of each $G_j$ and the bounds on $\psi$ from that proof gives
    \begin{equation}
    \min_j(-G_j(\mathbf{w}))\ge\min\{d/2,\,\varepsilon^2 d/(4C)\},
    \end{equation}
    where $C>0$ bounds $\langle\mathcal{F}(\mathbf{w}),\bar{\mathbf{w}}-\mathbf{w}\rangle+p\,(\varepsilon^2/2+\varepsilon\max_j(-G_j(\mathbf{w}))/2)$ on the feasible set. Hence $\boldsymbol{\lambda}$ is bounded, and so is $\boldsymbol{\mu}$ by full row rank of $[\mathbf{A}\ \mathbf{B}]$; thus $\mathcal{S}\cap\mathcal{U}$ is bounded.}

    (ii) Forward invariance of $\Omega_c$. Under Assumptions~\ref{ass:OLF_reg} and~\ref{ass:lojasiewicz}, and given the boundedness of $\mathcal{S}\cap\mathcal{U}$ from step (i), Lemma~\ref{lem:fwd_inv} ensures the existence of $\rho^\ast>0$ such that $\Omega_c\subseteq\mathcal{U}$ is closed, bounded, and forward invariant under~\eqref{eq:HGD_minimax} for every $c\in(0,\rho^\ast]$; hence for $\mathbf{z}(0)\in\Omega_c$, the trajectory $\mathbf{z}(t)\in\Omega_c\subseteq\mathcal{U}$ \rev{until it reaches $\mathcal{S}$}, and Assumption~\ref{ass:OLF_reg} guarantees $\nabla\mathbf{S}(\mathbf{z}(t))^\top\mathbf{S}(\mathbf{z}(t))\neq\mathbf{0}$ whenever $\mathbf{S}(\mathbf{z}(t))\neq\mathbf{0}$, so the denominator in~\eqref{eq:HGD_minimax} is well-defined.

    (iii) Convergence with the \rev{chosen decay law}. Substituting~\eqref{eq:HGD_minimax} into the chain rule for $V=\tfrac{1}{2}\|\mathbf{S}\|^2$ gives
    \begin{equation}\label{eq:Vdot_minimax}
    \dot V(\mathbf{z}(t)) = -\,\sigma\!\big(V(\mathbf{z}(t)),t\big),
    \end{equation}
    so $V(\mathbf{z}(t))\to 0$ with the \rev{decay} law selected via~\eqref{eq:sigma_def}; consequently $\mathbf{S}(\mathbf{z}(t))\to\mathbf{0}$, and the KKT conditions associated with the inequality constraints are enforced up to the smoothing tolerance $\varepsilon$. Since $\mathbf{z}(t)$ remains in the compact set $\Omega_c$, it has cluster points; by continuity of $\mathbf{S}$, every cluster point \rev{is a zero of $\mathbf{S}$ and therefore satisfies the KKT system~\eqref{eq:KKT_x}--\eqref{eq:KKT_eq}} up to the smoothing tolerance $\varepsilon$.
    \end{proof}

\subsection{Proof of Corollary~\ref{cor:scsc}}
\label{app:scsc_proof}

This appendix provides the complete block-elimination argument behind the nonsingularity step used in Cor.~\ref{cor:scsc}, together with the adaptation of the sublevel-set compactness argument to the convex-concave minimax setting. Strong convex-concavity of $J$ plays the role that strong monotonicity of $\mathcal{G}$ plays in the GNE proof (Appendix~\ref{app:nonsing_full}).

Setting. Let $\mathbf{w}\coloneqq\operatorname{col}(\mathbf{x},\mathbf{y})\in\mathbb{R}^{n_x+n_y}$ and combine the two stationarity blocks of~\eqref{eq:S_minimax} into the operator
\begin{equation}
\mathcal{F}(\mathbf{w})\coloneqq\operatorname{col}\!\bigl(\nabla_{\mathbf{x}} J(\mathbf{x},\mathbf{y}),\,-\nabla_{\mathbf{y}} J(\mathbf{x},\mathbf{y})\bigr),
\end{equation}
with combined equality Jacobian $\widetilde{\mathbf{A}}\coloneqq[\mathbf{A}\ \mathbf{B}]\in\mathbb{R}^{q\times(n_x+n_y)}$ of full row rank by the standing structural assumption of Sec.~\ref{sec:minmax}, and inequality map $\widetilde{\mathbf{G}}(\mathbf{w})\coloneqq\mathbf{G}(\mathbf{x},\mathbf{y})$. Under $m_x$-strong convexity of $J$ in $\mathbf{x}$ and $m_y$-strong concavity in $\mathbf{y}$, the symmetric part of $\nabla\mathcal{F}$ is block diagonal,
\begin{equation}
\tfrac{1}{2}\bigl(\nabla\mathcal{F}+\nabla\mathcal{F}^\top\bigr)
=\operatorname{diag}\!\bigl(\nabla^2_{\mathbf{xx}} J,\,-\nabla^2_{\mathbf{yy}} J\bigr)\succeq m\,\mathbf{I},
\end{equation}
with $m\coloneqq\min(m_x,m_y)>0$. Hence $\mathcal{F}$ is $m$-strongly monotone on $\mathbb{R}^{n_x+n_y}$. Stacking $\mathbf{z}\coloneqq\operatorname{col}(\mathbf{w},\boldsymbol{\lambda},\boldsymbol{\mu})$, the stationarity \rev{mapping}~\eqref{eq:S_minimax} reads
\begin{equation}
\begin{aligned}\mathbf{S}(\mathbf{z})=\operatorname{col}\!\bigl(&\mathcal{F}(\mathbf{w})+\nabla\widetilde{\mathbf{G}}(\mathbf{w})^\top\widetilde{\boldsymbol{\lambda}}+\widetilde{\mathbf{A}}^\top\boldsymbol{\mu},\\ &\boldsymbol{\phi}_\varepsilon(\boldsymbol{\lambda},-\widetilde{\mathbf{G}}(\mathbf{w})),\;\widetilde{\mathbf{A}}\mathbf{w}-\mathbf{b}\bigr).\end{aligned}
\end{equation}

\emph{Nonsingularity of $\nabla\mathbf{S}$ via block elimination.}

Step 1: FB sign convention.
Since $\varepsilon>0$, the diagonal entries of
$\mathbf{D}_a\coloneqq\operatorname{diag}(\partial_a\phi_\varepsilon)$ and
$\mathbf{D}_b\coloneqq\operatorname{diag}(\partial_b\phi_\varepsilon)$
lie strictly in $(-2,0)$ for every $\mathbf{z}$ and every $\varepsilon>0$; hence $\mathbf{D}_a,\mathbf{D}_b\prec\mathbf{0}$ strictly, and both are invertible.

Step 2: Jacobian block structure.
Differentiating $\mathbf{S}$ blockwise,
\begin{equation}
\nabla\mathbf{S}(\mathbf{z})=\begin{bmatrix}
\mathbf{H}_{\mathcal{L}}(\mathbf{w},\boldsymbol{\lambda}) & \nabla\widetilde{\mathbf{G}}(\mathbf{w})^\top\mathbf{D}_{\widetilde{\lambda}'} & \widetilde{\mathbf{A}}^\top\\[2pt]
-\mathbf{D}_b\,\nabla\widetilde{\mathbf{G}}(\mathbf{w}) & \mathbf{D}_a & \mathbf{0}\\[2pt]
\widetilde{\mathbf{A}} & \mathbf{0} & \mathbf{0}
\end{bmatrix},
\label{eq:minimax_blockJac}
\end{equation}
with generalized Lagrangian Hessian
$\mathbf{H}_{\mathcal{L}}(\mathbf{w},\boldsymbol{\lambda})\coloneqq\nabla\mathcal{F}(\mathbf{w})+\sum_{j=1}^{p}\widetilde{\lambda}_j\nabla^2\widetilde{G}_j(\mathbf{w})$,
where each $\widetilde{\lambda}_j$ is defined by~\eqref{eq:lambda_tilde}.
Convexity of each $\widetilde{G}_j$ and $\widetilde{\lambda}_j>0$ by construction give $\sum_j\widetilde{\lambda}_j\nabla^2\widetilde{G}_j\succeq\mathbf{0}$ unconditionally; together with the symmetric-part bound for $\nabla\mathcal{F}$ this yields
\begin{equation}
\mathbf{v}_1^{\!\top}\mathbf{H}_{\mathcal{L}}(\mathbf{w},\boldsymbol{\lambda})\,\mathbf{v}_1\ge m\,\|\mathbf{v}_1\|^2,\qquad\forall \mathbf{v}_1\in\mathbb{R}^{n_x+n_y}.
\label{eq:HL_coercive_mm}
\end{equation}

Step 3: Block elimination.
Let $\mathbf{v}=\operatorname{col}(\mathbf{v}_1,\mathbf{v}_2,\mathbf{v}_3)\in\mathbb{R}^{n_x+n_y}\times\mathbb{R}^p\times\mathbb{R}^q$ satisfy $\nabla\mathbf{S}(\mathbf{z})\,\mathbf{v}=\mathbf{0}$. The three block equations are
\begin{subequations}
\begin{align}
\mathbf{H}_{\mathcal{L}}\mathbf{v}_1+\nabla\widetilde{\mathbf{G}}^\top\mathbf{D}_{\widetilde{\lambda}'} \mathbf{v}_2+\widetilde{\mathbf{A}}^\top \mathbf{v}_3 &=\mathbf{0},\label{eq:mmbk1}\\
-\mathbf{D}_b\,\nabla\widetilde{\mathbf{G}}\,\mathbf{v}_1+\mathbf{D}_a\,\mathbf{v}_2 &=\mathbf{0},\label{eq:mmbk2}\\
\widetilde{\mathbf{A}}\,\mathbf{v}_1 &=\mathbf{0}.\label{eq:mmbk3}
\end{align}
\end{subequations}
Equation~\eqref{eq:mmbk3} gives $\mathbf{v}_1\in\ker(\widetilde{\mathbf{A}})$. Equation~\eqref{eq:mmbk2}, using invertibility of $\mathbf{D}_a$, yields $\mathbf{v}_2=\mathbf{D}_a^{-1}\mathbf{D}_b\,\nabla\widetilde{\mathbf{G}}\,\mathbf{v}_1$. Substituting into~\eqref{eq:mmbk1} and taking the inner product with $\mathbf{v}_1$ (which kills $\widetilde{\mathbf{A}}^\top \mathbf{v}_3$ by~\eqref{eq:mmbk3}) gives
\begin{equation}
\mathbf{v}_1^{\!\top}\Bigl(\mathbf{H}_{\mathcal{L}}+\nabla\widetilde{\mathbf{G}}^\top\mathbf{D}_{\widetilde{\lambda}'}\mathbf{D}_a^{-1}\mathbf{D}_b\,\nabla\widetilde{\mathbf{G}}\Bigr)\mathbf{v}_1=0.
\label{eq:Hhat_mm}
\end{equation}

Step 4: Positive-semidefinite argument.
Since $\mathbf{D}_a,\mathbf{D}_b\prec\mathbf{0}$ are both strictly negative diagonal, $\mathbf{D}_a^{-1}\mathbf{D}_b$ is strictly positive diagonal, and $\mathbf{D}_{\widetilde{\lambda}'}\succ\mathbf{0}$ (diagonal with entries $\widetilde{\lambda}_i'\in(0,1)$), so $\mathbf{D}_{\widetilde{\lambda}'}\mathbf{D}_a^{-1}\mathbf{D}_b\succ\mathbf{0}$ and $\nabla\widetilde{\mathbf{G}}^\top\mathbf{D}_{\widetilde{\lambda}'}\mathbf{D}_a^{-1}\mathbf{D}_b\,\nabla\widetilde{\mathbf{G}}\succeq\mathbf{0}$. Combined with~\eqref{eq:HL_coercive_mm},~\eqref{eq:Hhat_mm} becomes
\begin{equation}
\begin{split}
0&=\mathbf{v}_1^{\!\top}\mathbf{H}_{\mathcal{L}}\mathbf{v}_1+\mathbf{v}_1^{\!\top}\nabla\widetilde{\mathbf{G}}^\top\mathbf{D}_{\widetilde{\lambda}'}\mathbf{D}_a^{-1}\mathbf{D}_b\nabla\widetilde{\mathbf{G}}\,\mathbf{v}_1\\
&\ge m\,\|\mathbf{v}_1\|^2\ge 0.
\end{split}
\end{equation}
The sum of two nonnegatives vanishes only if both vanish, forcing $\mathbf{v}_1=\mathbf{0}$.

Step 5: Back-substitution.
With $\mathbf{v}_1=\mathbf{0}$: \eqref{eq:mmbk2} gives $\mathbf{D}_a \mathbf{v}_2=\mathbf{0}$, hence $\mathbf{v}_2=\mathbf{0}$; \eqref{eq:mmbk1} reduces to $\widetilde{\mathbf{A}}^\top \mathbf{v}_3=\mathbf{0}$, hence $\mathbf{v}_3=\mathbf{0}$ by full row rank of $\widetilde{\mathbf{A}}$. Therefore $\nabla\mathbf{S}(\mathbf{z})\,\mathbf{v}=\mathbf{0}\Rightarrow \mathbf{v}=\mathbf{0}$ and $\nabla\mathbf{S}(\mathbf{z})$ is nonsingular at every $\mathbf{z}$, and thus Assumption~\ref{ass:OLF_reg} holds automatically.

\emph{Existence and uniqueness of the saddle point.} Strong convex-concavity of $J$ renders $\mathcal{F}$ $m$-strongly monotone, so the \rev{variational inequality} associated with~\eqref{eq:minimax} is strongly monotone; under Slater's condition (Assumption~\ref{ass:Slater_general}), the associated variational inequality admits a unique solution $(\mathbf{x}^\star,\mathbf{y}^\star)$~\cite[Thm.~2.3.3(b)]{facchinei_finite-dimensional_2003}, and LICQ (Assumption~\ref{ass:LICQ_general}) yields uniqueness of the associated multipliers $(\boldsymbol{\lambda}^\star,\boldsymbol{\mu}^\star)$. No compactness assumption on the feasible set is required: the role played by compactness in Thm.~\ref{thm:minimax_main} is taken over here by strong monotonicity.

\emph{Sublevel-set compactness adapted to the convex-concave minimax setting.} We adapt the argument behind Prop.~\ref{prop:c_star} to the saddle-function structure of~\eqref{eq:minimax}, with stacked variable $\mathbf{z}=(\mathbf{x},\mathbf{y},\boldsymbol{\lambda},\boldsymbol{\mu})$ and $\mathbf{S}$ as in~\eqref{eq:S_minimax}. As in Prop.~\ref{prop:c_star}, we establish the existence of $c^\ast>0$ such that $\Omega_c$ is compact for every $c\in(0,c^\ast)$, by ruling out four escape modes on $\{V\le c\}$.

\emph{Case A ($\|\mathbf{x}\|\to\infty$ with the other coordinates bounded).} Strong convexity of $J$ in $\mathbf{x}$ with constant $m_x$ gives $\nabla_{\mathbf{x}}J(\mathbf{x},\mathbf{y})^\top\mathbf{x}\ge m_x\|\mathbf{x}\|^2-\|\nabla_{\mathbf{x}}J(\mathbf{0},\mathbf{y})\|\|\mathbf{x}\|$ for each fixed $\mathbf{y}$, so $\|\nabla_{\mathbf{x}}J(\mathbf{x},\mathbf{y})\|\to\infty$ uniformly on bounded $\mathbf{y}$. The first block of $\mathbf{S}$ (cf.~\eqref{eq:S_minimax}) is $\mathbf{s}_1=\nabla_{\mathbf{x}}J+\mathbf{A}^\top\boldsymbol{\mu}+\nabla_{\mathbf{x}}\mathbf{G}^\top\widetilde{\boldsymbol{\lambda}}$, in which the multiplier-dependent terms are bounded under bounded $(\boldsymbol{\lambda},\boldsymbol{\mu})$; hence $\|\mathbf{s}_1\|\to\infty$ and $V\to\infty$, contradicting $V\le c$.

\emph{Case A$'$ ($\|\mathbf{y}\|\to\infty$ with the other coordinates bounded).} Strong concavity of $J$ in $\mathbf{y}$ with constant $m_y$ gives $\nabla^2_{\mathbf{y}}J\preceq -m_y\mathbf{I}$, hence $-\nabla_{\mathbf{y}}J$ is coercive in $\mathbf{y}$ in the same sense as in Case~A. The second block of $\mathbf{S}$ is $\mathbf{s}_2=-\nabla_{\mathbf{y}}J+\mathbf{B}^\top\boldsymbol{\mu}+\nabla_{\mathbf{y}}\mathbf{G}^\top\widetilde{\boldsymbol{\lambda}}$, so $\|\mathbf{s}_2\|\to\infty$ as $\|\mathbf{y}\|\to\infty$ with bounded $(\boldsymbol{\lambda},\boldsymbol{\mu})$.

\emph{Case B ($\|\boldsymbol{\mu}\|\to\infty$ with the other coordinates bounded).} By the standing full-row-rank assumption on $[\mathbf{A}\;\mathbf{B}]$ (Sec.~\ref{sec:minmax}), $\|\mathbf{A}^\top\boldsymbol{\mu}\|^2+\|\mathbf{B}^\top\boldsymbol{\mu}\|^2\ge\sigma_{\min}^2\!\big([\mathbf{A}\;\mathbf{B}]\big)\|\boldsymbol{\mu}\|^2\to\infty$, so at least one of $\|\mathbf{s}_1\|,\|\mathbf{s}_2\|$ diverges. This case uses only full row rank of $[\mathbf{A}\;\mathbf{B}]$ to bound the $\boldsymbol{\mu}$-escape; the stronger stacked condition $\big[\begin{smallmatrix}[\mathbf{A}\;\mathbf{B}]\\ \mathbf{C}\end{smallmatrix}\big]$ full row rank is invoked only in the affine specialization below, where it secures the uniform Jacobian bound.

\emph{Case C ($\|\boldsymbol{\lambda}\|\to\infty$ with the other coordinates bounded).} The argument is identical to \rev{Stage~3} in the proof of Prop.~\ref{prop:c_star} (Appendix~\ref{app:c_star_proof}), with $\mathbf{g}$ replaced by $\mathbf{G}$ and $\mathbf{x}$ by $\mathbf{w}=(\mathbf{x},\mathbf{y})$. Along a sequence with $\|\boldsymbol{\lambda}^k\|\to\infty$, the sign screen gives a nonempty $\mathcal{J}_+=\{j:\lambda_j^k\to+\infty\}$; the FB limit gives $\phi_\varepsilon\to G_j(\rev{\mathbf{w}_\infty})$ for $j\in\mathcal{J}_+$; projecting the stacked gradient block onto $\ker(\widetilde{\mathbf{A}})$, $\widetilde{\mathbf{A}}=[\,\mathbf{A}\ \mathbf{B}\,]$, forces $\sum_{i\in\mathcal{J}_+}\rev{\omega_i}\nabla G_i(\rev{\mathbf{w}_\infty})\in\operatorname{range}(\widetilde{\mathbf{A}}^\top)$ irrespective of $\|\boldsymbol{\mu}\|$; and Slater's condition on the equality slice together with Cauchy--Schwarz gives $\liminf_k V(\mathbf{z}^k)\ge \rev{d^2/(2(1+\ell^2))}>0$, where $d$ is the feasibility margin of~\eqref{eq:minimax} and $\ell<\infty$ the constant of (C3). Hence there exists $c^\ast>0$ \rev{(any value in $(0,\,\min\{c_0,\,d^2/(2(1+\ell^2))\}]$, with $\ell$ computed over the $\mathbf{w}$-range of $\{V\le c_0\}$ as in Appendix~\ref{app:c_star_proof})} for which $\|\boldsymbol{\lambda}\|$ cannot escape on $\Omega_c$, $c<c^\ast$.

\emph{Joint escapes.} As in Appendix~\ref{app:c_star_proof}, the single-block cases do not by themselves exhaust simultaneous divergence. The reduction there carries over verbatim with $\mathbf{w}=(\mathbf{x},\mathbf{y})$ in place of $\mathbf{x}$ and $\widetilde{\mathbf{A}}=[\,\mathbf{A}\ \mathbf{B}\,]$ in place of $\mathbf{A}$: the equality block bounds the $\operatorname{range}(\widetilde{\mathbf{A}}^\top)$ component of $\mathbf{w}^k$; strong convex-concavity along $\ker(\widetilde{\mathbf{A}})$ precludes escape of $\mathbf{w}$ under bounded multipliers; and a simultaneous divergence of $\mathbf{w}^k$ and $\boldsymbol{\lambda}^k_{\mathcal{J}_+}$ would force a nonzero recession direction $\hat{\mathbf{w}}\in\ker(\widetilde{\mathbf{A}})$ with $G_j^\infty(\hat{\mathbf{w}})\le 0$ for \rev{every $j$ (the FB block gives $G_j\le\sqrt{2c}$)}, excluded by boundedness of the feasible set.

Combining Cases~A, A$'$, B, C with the joint-escape exclusion yields boundedness of $\Omega_c$; closedness is immediate from continuity of $V$. The block-elimination step above establishes nonsingularity of $\nabla\mathbf{S}(\mathbf{z})$ at every $\mathbf{z}\in\mathbb{R}^{n_x+n_y+p+q}$, hence in particular on $\Omega_c$. Lemma~\ref{lem:lin_PL} then establishes Assumption~\ref{ass:lojasiewicz} on $\Omega_c$ with $\alpha=1/2$, and Lemma~\ref{lem:fwd_inv} yields forward invariance of $\Omega_c$; in particular, for $\mathbf{z}(0)$ with $V(\mathbf{z}(0))\rev{\le c}<c^\ast$ the trajectory remains in $\Omega_c\subseteq\mathcal{U}=\mathbb{R}^{n_x+n_y+p+q}$ \rev{until it reaches $\mathcal{S}$}. The hypotheses of Thm.~\ref{thm:minimax_main} are therefore satisfied for any $\mathbf{z}(0)$ with $V(\mathbf{z}(0))<c^\ast$\rev{, and since $\mathbf{z}^\star_\varepsilon$ is the unique zero of $\mathbf{S}$ (Lemma~\ref{lem:smoothed_point}(i)), every cluster point of $\mathbf{z}(t)$ equals $\mathbf{z}^\star_\varepsilon$, so $\mathbf{z}(t)\to\mathbf{z}^\star_\varepsilon$}, yielding the \rev{regional} convergence claimed in Cor.~\ref{cor:scsc}.

\emph{Affine specialization: uniform Jacobian regularity and global convergence.} When $\mathbf{G}$ is affine in $(\mathbf{x},\mathbf{y})$, $\mathbf{G}(\mathbf{x},\mathbf{y})=\mathbf{C}\,\mathbf{w}-\mathbf{d}$ with $\mathbf{w}=(\mathbf{x},\mathbf{y})$ and $\mathbf{C}\in\mathbb{R}^{p\times(n_x+n_y)}$, the argument parallels the constrained-optimization case in Appendix~\ref{app:coro_proof}, with the combined equality Jacobian $\widetilde{\mathbf{A}}=[\,\mathbf{A}\ \mathbf{B}\,]$ in place of $\mathbf{A}$. The block $\mathbf{H}_{\mathcal{L}}$ reduces to \rev{$\nabla\mathcal{F}$, whose symmetric part is} $\operatorname{diag}(\nabla^2_{\mathbf{xx}}J,-\nabla^2_{\mathbf{yy}}J)\succeq m\mathbf{I}$ uniformly\rev{, and $\nabla\mathcal{F}$ is bounded because $\nabla J$ is Lipschitz}. Running the uniform-$\sigma_{\min}$ limit argument of Appendix~\ref{app:coro_proof} verbatim, with $\widetilde{\mathbf{A}}\mathbf{v}_1=\mathbf{0}$ and $\mathbf{c}_j^\top\mathbf{v}_1=0$ for the degenerate set $\mathcal{J}_0$, the residual relation $\sum_{j\in\mathcal{J}_0}v_{2,j}\,\mathbf{c}_j+\widetilde{\mathbf{A}}^\top\mathbf{v}_3=\mathbf{0}$ has only the trivial solution when $\big[\begin{smallmatrix}\widetilde{\mathbf{A}}\\ \mathbf{C}\end{smallmatrix}\big]$ has full row rank. Hence under this hypothesis $\sigma_{\min}(\nabla\mathbf{S})\ge\bar\sigma>0$ uniformly on $\mathbb{R}^{n_x+n_y+p+q}$, the global PL inequality holds, \rev{HGD trajectories have finite length (Appendix~\ref{app:fwd_inv_proof}),} and the closed-loop dynamics yield $\mathbf{z}(t)\to\mathbf{z}^\star_\varepsilon$ globally, irrespective of $c^\ast$, with $\|\mathbf{z}^\star_\varepsilon-(\mathbf{x}^\star,\mathbf{y}^\star,\boldsymbol{\lambda}^\star,\boldsymbol{\mu}^\star)\|=O(\varepsilon)$ \rev{by Lemma~\ref{lem:smoothed_point}(i),(iv), since full row rank gives Slater's condition and LICQ}.

\subsection{Proof of Theorem~\ref{thm:S1}}
\label{app:nonsing_full}

This appendix provides the complete block-elimination argument behind the nonsingularity step used in Thm.~\ref{thm:S1}, together with the adaptation of the sublevel-set compactness argument to the strongly monotone GNE setting. The block-elimination part draws on the FB Jacobian regularity results in~\cite{fischer_special_1992,liao-mcpherson_regularized_2019} and the KKT reformulation analysis in~\cite{dreves_solution_2011}.

Setting. Let $\mathbf{z}=\operatorname{col}(\mathbf{x},\boldsymbol{\lambda},\boldsymbol{\mu})\in\mathbb{R}^{n+p+q}$ and consider the stationarity \rev{mapping} $\mathbf{S}(\mathbf{z})$ of~\eqref{eq:S_def} with the inequality block $\boldsymbol{\phi}_\varepsilon(\boldsymbol{\lambda},-\mathbf{g}(\mathbf{x}))$ from~\eqref{eq:rineq}. Assumption~\ref{ass:S1_strong_mono} provides $m>0$ such that $\mathcal{G}$ is strongly monotone with constant $m$, $\mathbf{A}$ has full row rank, and $\mathbf{g}\in\rev{\mathcal{C}^2}$ \rev{(Sec.~\ref{sec:gne})}.

\emph{Nonsingularity of $\nabla\mathbf{S}$ via block elimination.}

Step 1: FB sign convention.
With $\boldsymbol{\phi}_\varepsilon(\boldsymbol{\lambda},-\mathbf{g}(\mathbf{x}))$, for each $j$ and every $\varepsilon>0$ the two partial derivatives
\begin{align}
\partial_a\phi_\varepsilon(\lambda_j,-g_j)
&=\frac{\lambda_j}{\sqrt{\lambda_j^2+g_j^2+\varepsilon^2}}-1, \notag\\
\partial_b\phi_\varepsilon(\lambda_j,-g_j)
&=\frac{-g_j}{\sqrt{\lambda_j^2+g_j^2+\varepsilon^2}}-1
\end{align}
lie strictly in $(-2,0)$. Hence the diagonal matrices
$\mathbf{D}_a\coloneqq\operatorname{diag}(\partial_a\phi_\varepsilon)$ and
$\mathbf{D}_b\coloneqq\operatorname{diag}(\partial_b\phi_\varepsilon)$
satisfy $\mathbf{D}_a\prec\mathbf{0}$ and $\mathbf{D}_b\prec\mathbf{0}$ strictly,
and are therefore invertible for every $\mathbf{z}$.

Step 2: Jacobian block structure.
Differentiating~\eqref{eq:S_def} blockwise, with the chain rule through the smoothed multiplier $\widetilde{\lambda}_j(\lambda_j)$ in the first block and the inner argument $-\mathbf{g}(\mathbf{x})$ in the third, gives
\begin{equation}
\nabla\mathbf{S}(\mathbf{z})=\begin{bmatrix}
\mathbf{H}_{\mathcal{L}}(\mathbf{x},\boldsymbol{\lambda}) & \nabla\mathbf{g}(\mathbf{x})^\top\mathbf{D}_{\widetilde{\lambda}'} & \mathbf{A}^\top\\[2pt]
\mathbf{A} & \mathbf{0} & \mathbf{0}\\[2pt]
-\mathbf{D}_b\,\nabla\mathbf{g}(\mathbf{x}) & \mathbf{D}_a & \mathbf{0}
\end{bmatrix},
\label{eq:full_blockJac}
\end{equation}
where $\mathbf{D}_{\widetilde{\lambda}'}=\operatorname{diag}(\widetilde{\lambda}_j')$ with $\widetilde{\lambda}_j'\in(0,1)$ from~\eqref{eq:lambda_tilde}, and $\mathbf{H}_{\mathcal{L}}(\mathbf{x},\boldsymbol{\lambda})\coloneqq\nabla\mathcal{G}(\mathbf{x})+\sum_{j=1}^{p}\widetilde{\lambda}_j\nabla^2 g_j(\mathbf{x})$ is the generalized Lagrangian operator. Here $\nabla\mathcal{G}(\mathbf{x})$ denotes the Jacobian of the pseudogradient $\mathcal{G}$ (not a Hessian, since $\mathcal{G}$ is generally not a gradient of a scalar); strong monotonicity of $\mathcal{G}$ in Assumption~\ref{ass:S1_strong_mono} is equivalent to $\tfrac12(\nabla\mathcal{G}+\nabla\mathcal{G}^\top)\succeq m\mathbf{I}$. Since each $g_j$ is convex ($\nabla^2 g_j\succeq\mathbf{0}$) and $\widetilde{\lambda}_j>0$ by construction~\eqref{eq:lambda_tilde}, the curvature contribution $\sum_j\widetilde{\lambda}_j\nabla^2 g_j\succeq\mathbf{0}$ holds unconditionally, so
\begin{equation}
\mathbf{v}_1^{\!\top}\mathbf{H}_{\mathcal{L}}(\mathbf{x},\boldsymbol{\lambda})\,\mathbf{v}_1\;\ge\;m\,\|\mathbf{v}_1\|^2,\qquad\forall \mathbf{v}_1\in\ker(\mathbf{A}).
\label{eq:HL_coercive}
\end{equation}
Equation~\eqref{eq:HL_coercive} uses only the restriction of strong monotonicity to $\ker(\mathbf{A})$; no sign assumption on $\boldsymbol{\lambda}$ is required.

Step 3: Block elimination.
Let $\mathbf{v}=\operatorname{col}(\mathbf{v}_1,\mathbf{v}_2,\mathbf{v}_3)\in\mathbb{R}^{n}\times\mathbb{R}^{p}\times\mathbb{R}^{q}$ satisfy $\nabla\mathbf{S}(\mathbf{z})\,\mathbf{v}=\mathbf{0}$. The three block equations read
\begin{subequations}
\begin{align}
\mathbf{H}_{\mathcal{L}}\,\mathbf{v}_1+\nabla\mathbf{g}^\top\mathbf{D}_{\widetilde{\lambda}'} \mathbf{v}_2+\mathbf{A}^\top \mathbf{v}_3 &=\mathbf{0},\label{eq:bk1}\\
\mathbf{A}\,\mathbf{v}_1 &=\mathbf{0},\label{eq:bk2}\\
-\mathbf{D}_b\,\nabla\mathbf{g}\,\mathbf{v}_1+\mathbf{D}_a\,\mathbf{v}_2 &=\mathbf{0}.\label{eq:bk3}
\end{align}
\end{subequations}
Equation~\eqref{eq:bk2} gives $\mathbf{v}_1\in\ker(\mathbf{A})$. Equation~\eqref{eq:bk3}, using invertibility of $\mathbf{D}_a$, yields $\mathbf{v}_2=\mathbf{D}_a^{-1}\mathbf{D}_b\,\nabla\mathbf{g}\,\mathbf{v}_1$. Substituting into~\eqref{eq:bk1} and taking the inner product with $\mathbf{v}_1$, while using~\eqref{eq:bk2} to eliminate the $\mathbf{A}^\top \mathbf{v}_3$ term, gives
\begin{equation}
\mathbf{v}_1^{\!\top}\!\underbrace{\Bigl(\mathbf{H}_{\mathcal{L}}+\nabla\mathbf{g}^\top\mathbf{D}_{\widetilde{\lambda}'}\mathbf{D}_a^{-1}\mathbf{D}_b\,\nabla\mathbf{g}\Bigr)}_{\eqqcolon\,\widehat{\mathbf{H}}}\!\mathbf{v}_1\;=\;0.
\label{eq:Hhat_equation}
\end{equation}

Step 4: Positive-semidefinite argument.
Since $\mathbf{D}_a\prec\mathbf{0}$ is diagonal, $\mathbf{D}_a^{-1}\prec\mathbf{0}$ is also diagonal. Combined with $\mathbf{D}_b\prec\mathbf{0}$, the product $\mathbf{D}_a^{-1}\mathbf{D}_b$ is strictly positive diagonal. Since $\mathbf{D}_{\widetilde{\lambda}'}\succ\mathbf{0}$ (diagonal with entries $\widetilde{\lambda}_j'\in(0,1)$), the product $\mathbf{D}_{\widetilde{\lambda}'}\mathbf{D}_a^{-1}\mathbf{D}_b\succ\mathbf{0}$ is strictly positive diagonal. Therefore $\nabla\mathbf{g}^\top\mathbf{D}_{\widetilde{\lambda}'}\mathbf{D}_a^{-1}\mathbf{D}_b\,\nabla\mathbf{g}\succeq\mathbf{0}$, and~\eqref{eq:Hhat_equation} combined with~\eqref{eq:HL_coercive} restricted to $\mathbf{v}_1\in\ker(\mathbf{A})$ yields
\begin{equation}
0\;=\;\mathbf{v}_1^{\!\top}\widehat{\mathbf{H}}\,\mathbf{v}_1
\;\ge\;\mathbf{v}_1^{\!\top}\mathbf{H}_{\mathcal{L}}\,\mathbf{v}_1
\;\ge\;m\,\|\mathbf{v}_1\|^2.
\end{equation}
Since $m>0$, this forces $\mathbf{v}_1=\mathbf{0}$.

Step 5: Back-substitution.
With $\mathbf{v}_1=\mathbf{0}$, \eqref{eq:bk3} immediately gives $\mathbf{D}_a\,\mathbf{v}_2=\mathbf{0}$; invertibility of $\mathbf{D}_a$ gives $\mathbf{v}_2=\mathbf{0}$. Equation~\eqref{eq:bk1} then reduces to $\mathbf{A}^\top \mathbf{v}_3=\mathbf{0}$; the full row rank of $\mathbf{A}$ implies $\mathbf{v}_3=\mathbf{0}$. Hence the only solution of $\nabla\mathbf{S}(\mathbf{z})\,\mathbf{v}=\mathbf{0}$ is $\mathbf{v}=\mathbf{0}$, i.e., $\nabla\mathbf{S}(\mathbf{z})$ is nonsingular at every $\mathbf{z}$, and thus Assumption~\ref{ass:OLF_reg} holds automatically.

\emph{Sublevel-set compactness adapted to the strongly monotone GNE setting.} We adapt the argument behind Prop.~\ref{prop:c_star} to the GNE setting, with stacked variable $\mathbf{z}=(\mathbf{x},\boldsymbol{\lambda},\boldsymbol{\mu})$ and $\mathbf{S}$ as in~\eqref{eq:S_def}. As in Prop.~\ref{prop:c_star}\rev{, and under the bounded-feasible-set hypothesis of Thm.~\ref{thm:S1}}, we establish the existence of $c^\ast>0$ such that $\Omega_c$ is compact for every $c\in(0,c^\ast)$, by ruling out three escape modes on $\{V\le c\}$.

\emph{Case A ($\|\mathbf{x}\|\to\infty$ with the other coordinates bounded).} Strong monotonicity of $\mathcal{G}$ on $\ker(\mathbf{A})$ does not control growth along $\ker(\mathbf{A})^\perp$, so we decompose $\mathbf{x}=\mathbf{x}_{\perp}+\mathbf{x}_{\parallel}$ with $\mathbf{x}_{\perp}\in\ker(\mathbf{A})$ and $\mathbf{x}_{\parallel}\in\ker(\mathbf{A})^\perp=\operatorname{range}(\mathbf{A}^\top)$.

\emph{Subcase A1 ($\|\mathbf{x}_{\parallel}\|\to\infty$):} writing $\mathbf{x}_{\parallel}=\mathbf{A}^\top\mathbf{w}$, full row rank of $\mathbf{A}$ gives $\|\mathbf{A}\mathbf{x}\|=\|\mathbf{A}\mathbf{x}_{\parallel}\|\ge\sqrt{\sigma_{\min}(\mathbf{A}\mathbf{A}^\top)}\,\|\mathbf{x}_{\parallel}\|$, so the equality block $\mathbf{r}_{\mathrm{eq}}(\mathbf{x})=\mathbf{A}\mathbf{x}-\mathbf{b}$ diverges and $V\to\infty$.

\emph{Subcase A2 ($\|\mathbf{x}_{\perp}\|\to\infty$ with $\|\mathbf{x}_{\parallel}\|$ bounded):} setting the reference point $\mathbf{x}_0\coloneqq\mathbf{x}_{\parallel}$ (which remains bounded along this subcase), we have $\mathbf{x}-\mathbf{x}_0=\mathbf{x}_\perp\in\ker(\mathbf{A})$ by construction, so strong monotonicity restricted to $\ker(\mathbf{A})$ applies:
\begin{equation}
\langle\mathcal{G}(\mathbf{x})-\mathcal{G}(\mathbf{x}_0),\,\mathbf{x}-\mathbf{x}_0\rangle\ge m\|\mathbf{x}_\perp\|^2\to\infty,
\end{equation}
hence $\|\mathcal{G}(\mathbf{x})\|\to\infty$ (using boundedness of $\mathcal{G}(\mathbf{x}_0)$). The first block of $\mathbf{S}$ is $\mathbf{s}_1=\mathcal{G}(\mathbf{x})+\nabla\mathbf{g}(\mathbf{x})^\top\widetilde{\boldsymbol{\lambda}}+\mathbf{A}^\top\boldsymbol{\mu}$, in which the multiplier-dependent terms are bounded under bounded $(\boldsymbol{\lambda},\boldsymbol{\mu})$; hence $\|\mathbf{s}_1\|\to\infty$ and $V\to\infty$.

\emph{Case B ($\|\boldsymbol{\mu}\|\to\infty$ with the other coordinates bounded).} Full row rank of $\mathbf{A}$ gives $\|\mathbf{A}^\top\boldsymbol{\mu}\|\ge\sigma_{\min}(\mathbf{A})\|\boldsymbol{\mu}\|\to\infty$; under bounded $\mathbf{x}$ and bounded $\widetilde{\boldsymbol{\lambda}}$, the remaining terms of $\mathbf{s}_1$ are bounded, so $\|\mathbf{s}_1\|\to\infty$.

\emph{Case C ($\|\boldsymbol{\lambda}\|\to\infty$ with the other coordinates bounded).} The argument is identical to \rev{Stage~3} in the proof of Prop.~\ref{prop:c_star} (Appendix~\ref{app:c_star_proof}). Along a sequence with $\|\boldsymbol{\lambda}^k\|\to\infty$, the sign screen gives a nonempty $\mathcal{J}_+=\{j:\lambda_j^k\to+\infty\}$; the FB limit gives $\phi_\varepsilon\to g_j(\rev{\mathbf{x}_\infty})$ for $j\in\mathcal{J}_+$; projecting the gradient block $\mathbf{s}_1$ onto $\ker(\mathbf{A})$ removes $\mathbf{A}^\top\boldsymbol{\mu}$ irrespective of $\|\boldsymbol{\mu}\|$ and forces $\sum_{i\in\mathcal{J}_+}\rev{\omega_i}\nabla g_i(\rev{\mathbf{x}_\infty})\in\operatorname{range}(\mathbf{A}^\top)$ (the pseudogradient $\mathcal{G}(\mathbf{x})$ stays bounded under bounded $\mathbf{x}$); and Slater's condition on the equality slice together with Cauchy--Schwarz gives $\liminf_k V(\mathbf{z}^k)\ge \rev{d^2/(2(1+\ell^2))}>0$, where $d$ is the feasibility margin of the shared constraints and $\ell<\infty$ the constant of (C3). Hence there exists $c^\ast>0$ \rev{(any value in $(0,\,\min\{c_0,\,d^2/(2(1+\ell^2))\}]$, with $\ell$ computed over the $\mathbf{x}$-range of $\{V\le c_0\}$ as in Appendix~\ref{app:c_star_proof})} for which $\|\boldsymbol{\lambda}\|$ cannot escape on $\Omega_c$, $c<c^\ast$.

\emph{Joint escapes.} As in Appendix~\ref{app:c_star_proof}, simultaneous divergence is excluded by the same reduction: the equality block bounds the $\operatorname{range}(\mathbf{A}^\top)$ component of $\mathbf{x}^k$; strong monotonicity on $\ker(\mathbf{A})$ precludes escape of $\mathbf{x}$ under bounded multipliers; and a joint divergence of $\mathbf{x}^k$ and $\boldsymbol{\lambda}^k_{\mathcal{J}_+}$ would force a nonzero recession direction $\hat{\mathbf{x}}\in\ker(\mathbf{A})$ with $g_j^\infty(\hat{\mathbf{x}})\le 0$ for \rev{every $j$ (the FB block gives $g_j\le\sqrt{2c}$)}, excluded by boundedness of the shared feasible set.

Combining Cases~A, B, C with the joint-escape exclusion yields boundedness of $\Omega_c$; closedness is immediate from continuity of $V$. The block-elimination argument above establishes nonsingularity of $\nabla\mathbf{S}(\mathbf{z})$ at every $\mathbf{z}\in\mathbb{R}^{n+p+q}$, hence in particular on $\Omega_c$. Lemma~\ref{lem:lin_PL} then establishes Assumption~\ref{ass:lojasiewicz} on $\Omega_c$ with $\alpha=1/2$, and Lemma~\ref{lem:fwd_inv} yields forward invariance of $\Omega_c$; in particular, for $\mathbf{z}(0)$ with $V(\mathbf{z}(0))\rev{\le c}<c^\ast$ the trajectory remains in $\Omega_c\subseteq\mathcal{U}=\mathbb{R}^{n+p+q}$ \rev{until it reaches $\mathcal{S}$}. \rev{Along the trajectory, $\dot V=-\sigma(V,t)$ drives $V(\mathbf{z}(t))\to 0$ with the chosen law, and since $\mathbf{z}^\star_\varepsilon$ is the unique zero of $\mathbf{S}$ (Lemma~\ref{lem:smoothed_point}(i)), every cluster point of $\mathbf{z}(t)$ in the compact set $\Omega_c$ equals $\mathbf{z}^\star_\varepsilon$, so $\mathbf{z}(t)\to\mathbf{z}^\star_\varepsilon$.}

\emph{Affine specialization: uniform Jacobian regularity and global convergence.} When $\mathbf{g}(\mathbf{x})=\mathbf{C}\mathbf{x}-\mathbf{d}$ is affine with $\mathbf{C}\in\mathbb{R}^{p\times n}$, the GNE analogue of the argument in Appendix~\ref{app:coro_proof} applies. Affineness eliminates the multiplier-weighted constraint Hessians, so the symmetric part of $\nabla\mathcal{G}+\sum_j\widetilde{\lambda}_j\nabla^2 g_j$ on $\ker(\mathbf{A})$ reduces to that of $\nabla\mathcal{G}$ alone, which is bounded below by $m\mathbf{I}$ on $\ker(\mathbf{A})$ by Assumption~\ref{ass:S1_strong_mono}\rev{, and $\nabla\mathcal{G}$ is bounded because $\mathcal{G}$ is Lipschitz}. Running the uniform-$\sigma_{\min}$ limit argument of Appendix~\ref{app:coro_proof} (\rev{$\mathbf{A}\mathbf{v}_1^k\to\mathbf{0}$ and boundedness of $\nabla\mathcal{G}$ give $\mathbf{v}_1^{k\top}\nabla\mathcal{G}(\mathbf{x}^k)\mathbf{v}_1^k\ge m\|\mathbf{v}_1^k\|^2-o(1)$ from strong monotonicity on $\ker(\mathbf{A})$ alone}), the residual relation $\sum_{j\in\mathcal{J}_0}v_{2,j}\,\mathbf{c}_j+\mathbf{A}^\top\mathbf{v}_3=\mathbf{0}$ has only the trivial solution when $\big[\begin{smallmatrix}\mathbf{A}\\ \mathbf{C}\end{smallmatrix}\big]$ has full row rank. Hence under this hypothesis $\sigma_{\min}(\nabla\mathbf{S})\ge\bar\sigma>0$ uniformly on $\mathbb{R}^{n+p+q}$, the global PL inequality holds, \rev{HGD trajectories have finite length (Appendix~\ref{app:fwd_inv_proof}),} and the closed-loop dynamics yield $\mathbf{z}(t)\to\mathbf{z}^\star_\varepsilon$ globally, irrespective of $c^\ast$, with $\|\mathbf{z}^\star_\varepsilon-\rev{(\mathbf{x}^\star,\boldsymbol{\lambda}^\star,\boldsymbol{\mu}^\star)}\|=O(\varepsilon)$ \rev{by Lemma~\ref{lem:smoothed_point}(i),(iv), since full row rank gives Slater's condition and LICQ}.

Remarks. (i) \rev{Besides Slater's condition and the hypotheses of the respective clause of Thm.~\ref{thm:S1} (a bounded shared feasible set, or affine $\mathbf{g}$ with $\big[\begin{smallmatrix}\mathbf{A}\\ \mathbf{C}\end{smallmatrix}\big]$ of full row rank and Lipschitz $\mathcal{G}$),} the proof uses only strong monotonicity of $\mathcal{G}$, convexity of $\mathbf{g}$, and full row rank of $\mathbf{A}$; in particular, \rev{LICQ is not assumed, and in the second clause it follows from full row rank}. 
(ii) The argument \rev{applies to every} $\varepsilon>0$ because the strict bounds $\mathbf{D}_a,\mathbf{D}_b\prec\mathbf{0}$ hold for every $\varepsilon>0$ at every state. 

\subsection{\rev{Hypotheses and Reach of the Results}}
\label{app:assumptions}
\rev{Table~\ref{tab:assumptions} summarizes the hypotheses and the reach of the convergence results, from local results under conditions on the OLF to global results under affine constraints. The decay law does not enter the hypotheses, so each result holds for all four laws.}

\begin{table}[t]
\centering
\caption{\rev{Hypotheses and reach of the convergence results}}
\label{tab:assumptions}
\footnotesize
\setlength{\tabcolsep}{3pt}
\renewcommand{\arraystretch}{1.15}
\rev{\begin{tabular}{@{}>{\raggedright\arraybackslash}p{0.24\columnwidth}>{\raggedright\arraybackslash}p{0.57\columnwidth}>{\raggedright\arraybackslash}p{0.13\columnwidth}@{}}
\hline
\textbf{Result} & \textbf{Main hypotheses} & \textbf{Reach} \\
\hline
Lemmas~\ref{lem:w_pseudo}, \ref{lem:pseudo}, \ref{lem:svb}, \ref{lem:fwd_inv} & on $\mathcal{U}$: Assumption~\ref{ass:OLF_reg} (HGD), invertible $\nabla\mathbf{S}$ (ND), or $\operatorname{sym}(\nabla\mathbf{S})\succeq m\mathbf{I}$ (GD); for Lemma~\ref{lem:fwd_inv}, also Assumption~\ref{ass:lojasiewicz} & local \\
Thms.~\ref{thm:unified_constrained}, \ref{thm:minimax_main} & Assumptions~\ref{ass:OLF_reg} and~\ref{ass:lojasiewicz} near a KKT or saddle point; for Thm.~\ref{thm:unified_constrained}, LICQ there and a unique zero of $\mathbf{S}$ nearby; for Thm.~\ref{thm:minimax_main}, convex-concave $J$, compact feasible set, and Slater's condition & local \\
Cors.~\ref{cor:auto_nonsing_general}, \ref{cor:scsc}, Thm.~\ref{thm:S1} & strongly convex $J$, strongly convex-concave $J$, or strongly monotone game; convex inequality and full-row-rank affine equality constraints; Slater's condition; bounded feasible set; LICQ in the corollaries & regional \\
Prop.~\ref{prop:affine_global}, Cor.~\ref{cor:scsc}, Thm.~\ref{thm:S1} & strong convexity or monotonicity as above; affine constraints whose stacked Jacobian has full row rank; Lipschitz gradient or pseudogradient; no boundedness & global \\
$p=q=0$ (Secs.~\ref{subsec:scsc_case}, \ref{subsec:gne_main}) & strongly convex $J$, strongly convex-concave $J$, or strongly monotone game & global \\
\hline
\end{tabular}}\\[3pt]
\begin{minipage}{\columnwidth}\footnotesize
\rev{Local: $\mathbf{z}(0)$ in a forward-invariant sublevel set $\Omega_c\subseteq\mathcal{U}$ (Lemma~\ref{lem:fwd_inv}); regional: $V(\mathbf{z}(0))<c^\ast$; global: every $\mathbf{z}(0)$. The theorems use HGD; ND enforces the same decay law wherever $\nabla\mathbf{S}$ is invertible, which holds everywhere in the last three rows, and GD applies in the last row.}
\end{minipage}
\end{table}

\subsection{\rev{Simulation Parameters and Data}}
\label{app:sim_params}
\rev{All examples use the laws~\eqref{eq:sigma_def} with $c=1$ (Exp), $k=1$ and $\gamma=1/2$ (FT), $a=b=1$, $\gamma=1/2$, and $\delta=2$ (FxT), and $\mu=3$ (PT); GD uses $m=1$.}
\rev{For PT we implement the bounded gain $\mu/(T+\bar\varepsilon-t)$ on $[0,T]$ in place of the singular gain $\mu/(T-t)$. Under HGD and ND it yields $V(\mathbf{z}(T))=V(\mathbf{z}(0))\,\big(\bar\varepsilon/(T+\bar\varepsilon)\big)^{\mu}$, and under GD at most this value, so we set $\bar\varepsilon$ before each run to obtain $\|\mathbf{S}(\mathbf{z}(T))\|=10^{-6}$: $T=1$~s and $\bar\varepsilon=2.7\times10^{-5}$ in Sec.~\ref{subsec:logsumexp_example}, and $T=10$~s with $\bar\varepsilon=7.5\times10^{-5}$, $2.8\times10^{-4}$, and $1.2\times10^{-4}$ in Secs.~\ref{subsec:NUM}, \ref{subsec:minimax_example}, and~\ref{subsec:cournot_gne}.}
\rev{CPU times are medians over 20 runs on one computational thread.}

\rev{The data of the examples are as follows.}

\subsubsection{\rev{NUM (Sec.~\ref{subsec:NUM})}}
\rev{The routing matrix is a fixed pseudo-random matrix of density $0.15$ with full row rank, $\mathbf{c}=\mathbf{R}\bar{\mathbf{x}}+0.25\cdot\mathbf{1}$ with $\bar x_j\in[0.5,1]$, and $\beta_j\in[0.5,2]$.}

\subsubsection{\rev{Constrained Minimax (Sec.~\ref{subsec:minimax_example})}}
\rev{$\mathbf{Q}_x=\operatorname{diag}(3,2,1.5)$, $\mathbf{Q}_y=\operatorname{diag}(2,3)$, $\mathbf{c}_x=(1,-0.5,0.5)$, $\mathbf{d}_y=(0.5,1)$, and $\mathbf{C}_{xy}$ has rows $(0.3,0.1)$, $(0.1,0.4)$, and $(0.2,0.2)$, so that $m_x=1.5$ and $m_y=2$; the constraints are $\mathbf{1}^\top\mathbf{x}+\mathbf{1}^\top\mathbf{y}=5$, $x_1+x_2+y_1\le 4$, and $x_2+x_3+y_2\le 3$; at the saddle point, $\mathbf{x}^\star=(0.500,1.256,1.007)$ and $\mathbf{y}^\star=(1.500,0.737)$, and the second inequality is active with multiplier $0.566$.}

\subsubsection{\rev{Cournot Game (Sec.~\ref{subsec:cournot_gne})}}
\rev{$\bar{\boldsymbol{\pi}}=(10,8)$, $\mathbf{D}=\operatorname{diag}(1,1)$, and the shared constraints are $\mathbf{e}_1^\top\mathbf{E}\mathbf{x}=12$ and $\mathbf{e}_2^\top\mathbf{E}\mathbf{x}\le 4$; at the v-GNE, every firm produces $\mathbf{x}_i^\star=(3,1)$, the market-2 capacity is active with multiplier $2$, and the equality multiplier is $-8$.}

\subsection{Discrete-Time Implementation Notes}
\label{app:discrete_notes}

This appendix expands on the brief discrete-time discussion of Sec.~\ref{subsec:discrete_impl}, surveying recent directions for the discrete-time realization of continuous-time optimization dynamics.

Discrete-gradient (DG) methods construct a step that exactly decreases a given Lyapunov/energy function, thereby preserving stability by construction~\cite{hernandez-solano_numerical_2023}. Related effective discretization schemes based on the Runge--Kutta (RK) family can preserve the Lyapunov decrease while reducing computational cost. In particular, reduced RK (RRK) methods simplify the standard RK update by removing numerically insignificant terms, maintaining the same stability properties with fewer operations~\cite{guedes_preservation_2025}.

For homogeneous systems, Lyapunov-based discretizations can preserve exponential, FT, and FxT convergence rates and the Lyapunov function across the continuous-to-discrete transition~\cite{sanchez_lyapunov-based_2021}. More generally, a unified framework for translating Lyapunov arguments into discrete-time updates was proposed in~\cite{ushiyama_unified_2023}, ensuring that discrete algorithms inherit the same Lyapunov guarantees as their continuous-time counterparts. FT and FxT guarantees for convex optimization, including proximal and minimax variants, have been established in continuous-time and analyzed under forward-Euler discretization (with small step sizes)~\cite{garg_fixed-time_2023}.

These results suggest that discrete implementations should be guided by Lyapunov inequalities rather than tied to a specific numerical scheme.

\subsection{\rev{Comparison With Existing Methods}}
\label{app:comparisons}

\rev{This appendix compares OLF with continuous-time methods that the framework does not contain as special cases; the flows it contains are listed in Appendix~\ref{app:known_flows}. The baselines are the projected primal-dual flow of~\cite{feijer_stability_2010} (NUM) and the projected saddle-point flow of~\cite{cherukuri_role_2018}, applied with the operator of the variational inequality in place of the primal gradient (minimax and Cournot); both project the inequality multipliers onto the nonnegative orthant. For NUM only, the baselines also include the primal-dual flow on the generalized Lagrangian of~\cite{ozaslan_exponential_2024} with its FT and FxT scalings, since this flow is designed for constrained optimization.}

\rev{Every method starts from the same ten random initial conditions, with primal variables uniform in $[0.1,1]$ (NUM) or $[0,1]$ (minimax and Cournot), inequality multipliers uniform in $[0,1]$, and equality multipliers uniform in $[-1,1]$. Every method is integrated by \texttt{ode15s} with \texttt{RelTol}$=10^{-9}$ and \texttt{AbsTol}$=10^{-12}$. Each run stops when the natural KKT residual of the unsmoothed problem falls to $10^{-5}$. This residual is computed with the nonnegative part of the multipliers and with complementarity written as $\min\{\lambda_j,-g_j\}=0$, so every method is measured against the exact KKT conditions.}

\rev{For each example, the OLF row of Table~\ref{tab:comparisons} is the fastest of 40 ND configurations, selected by median CPU time over five other random initial conditions. Each configuration combines a decay law (the Exp law, the FT law with $\gamma\in\{0.5,0.7\}$, or the PT law with $\mu\in\{2,5\}$), one of two implementations of the ND vector field, which are equal up to rounding, and one of four Jacobian options for \texttt{ode15s}. These options are the analytic Jacobian of the closed loop, finite differences, and the matrix $-(\sigma/(2V))\mathbf{I}$ in dense or sparse storage, which equals the analytic Jacobian at the zero of $\mathbf{S}$ under the Exp and PT laws. Since \texttt{ode15s} uses the Jacobian only in its Newton iterations, the tolerances set the accuracy in every case. The OLF runs use $\varepsilon=10^{-6}$, $T=10$~s, the bounded gain of Remark~\ref{rem:pt_gain} with $\bar\varepsilon$ set before each run as in Appendix~\ref{app:sim_params}, and otherwise the parameters of that appendix; the baselines use unit gains and finite-difference Jacobians. HGD computes $\nabla\mathbf{S}^\top\mathbf{S}$ without forming $\nabla\mathbf{S}$, and ND solves with $\nabla\mathbf{S}$ by block elimination in NUM.}

\rev{Table~\ref{tab:comparisons} reports three quantities, each as the median and the range over the initial conditions. The settling time is the time at which the stopping rule is met. The number of right-hand-side evaluations includes those that \texttt{ode15s} spends on finite-difference Jacobians, but not the evaluations of the Jacobian supplied to the OLF runs; these number $1$, $8$, and $6$ in median on NUM, minimax, and Cournot, and each is cheaper than a right-hand-side evaluation. The CPU time is itself the median of ten interleaved timed rounds on one computational thread. Evaluation counts and CPU times are insensitive to a constant rescaling of time, whereas settling times scale with the gains of each method.}

\begin{table*}[t]
\centering
\caption{\rev{Comparison on the constrained examples: median (range) over ten initial conditions}}
\label{tab:comparisons}
\footnotesize
\setlength{\tabcolsep}{6pt}
\rev{\begin{tabular}{@{}lccc@{}}
\hline
\textbf{Method} & \textbf{Settling time (s)} & \textbf{Right-hand-side evaluations} & \textbf{CPU time (ms)} \\
\hline
\multicolumn{4}{@{}l}{NUM (Sec.~\ref{subsec:NUM})}\\
OLF & 9.98 (9.98--9.98) & 1234 (1083--1480) & 100.7 (86.5--118.9) \\
Projected primal-dual flow~\cite{feijer_stability_2010} & 33.14 (31.39--34.89) & 8954 (7671--14530) & 246.2 (214.3--334.2) \\
Primal-dual, generalized Lagrangian~\cite{ozaslan_exponential_2024} & 27.30 (26.89--27.55) & 2348 (2200--2495) & 103.8 (95.5--112.8) \\
Same, FT scaling~\cite{ozaslan_exponential_2024} & 6.64 (6.43--6.76) & 2762 (2501--2941) & 102.4 (88.8--108.3) \\
Same, FxT scaling~\cite{ozaslan_exponential_2024} & 3.75 (3.67--3.82) & 3020 (2792--3208) & 115.2 (106.2--122.6) \\
\hline
\multicolumn{4}{@{}l}{Constrained minimax (Sec.~\ref{subsec:minimax_example})}\\
OLF & 10.00 (10.00--10.00) & 350 (268--425) & 21.1 (15.3--27.0) \\
Projected primal-dual flow~\cite{cherukuri_role_2018} & 22.26 (21.76--58.40) & 1606 (1539--2787) & 53.2 (49.9--91.6) \\
\hline
\multicolumn{4}{@{}l}{Cournot GNE (Sec.~\ref{subsec:cournot_gne})}\\
OLF & 10.00 (10.00--10.00) & 362 (272--630) & 20.4 (14.7--39.9) \\
Projected primal-dual flow~\cite{cherukuri_role_2018} & 17.41 (17.29--17.59) & 675 (632--876) & 21.8 (19.7--31.2) \\
\hline
\end{tabular}}
\end{table*}

\rev{In every example, the fastest configuration is ND with the PT law and the Jacobian $-(\sigma/(2V))\mathbf{I}$, with $\mu=5$ on NUM and $\mu=2$ on minimax and Cournot, and the OLF runs stop by $T$. From every initial condition, OLF needs fewer right-hand-side evaluations than every baseline. On NUM and minimax, OLF is faster than the projected flows from every initial condition, by factors of $2.4$ and $2.5$ in median. On Cournot, its median CPU time is $6$ percent smaller than that of the projected flow, and on NUM it is $2$ to $13$ percent smaller than those of the three generalized-Lagrangian flows. HGD is not in the table, since in a preliminary screening with the same protocol its fastest configuration needed $1.7$ to $5.3$ times the CPU time of the fastest ND configuration. In the runs of Sec.~\ref{subsec:NUM}, however, from another initial condition and stopped at $\|\mathbf{S}\|=10^{-6}$, HGD needs less CPU time than ND, so the relative cost of the two realizations depends on the example, the initial condition, and the stopping rule. Among the compared methods, only OLF has a PT guarantee; the baselines reach at most the FxT regime~\cite{ozaslan_exponential_2024}. All runs end within $2.3\times10^{-5}$ of the exact KKT point.}

\rev{The projected flows have a vector field that is discontinuous where a multiplier reaches zero, so they are also run with \texttt{ode45} as a check of their settling times; these runs are not part of the comparison. The two solvers agree to within a relative $10^{-3}$ at 29 of the 30 initial conditions, and their median settling times coincide in all three examples. At the fourth minimax initial condition, \texttt{ode15s} stops at $58.40$~s against $22.78$~s with \texttt{ode45}, which sets the upper end of the settling-time range of the projected flow on minimax in Table~\ref{tab:comparisons}.}

\rev{A quadratic test shows the stiffness of the HGD vector field. For $J(\mathbf{x})=\tfrac12\mathbf{x}^\top\mathbf{H}\mathbf{x}$ with $\mathbf{H}=\operatorname{diag}(1,1/\kappa_{\mathbf{H}})$ and the exponential law, the moduli of the eigenvalues of the Jacobian of the HGD vector field differ, along most of the trajectory from $(1,1)$, by a factor of about $\kappa_{\mathbf{H}}^2$ for $\kappa_{\mathbf{H}}=10$, $100$, and $1000$, against $\kappa_{\mathbf{H}}$ for the gradient flow and $1$ for ND, whose vector field is $-(c/2)\,\mathbf{x}$.}

\subsection{\rev{Time and Memory per Evaluation}}
\label{app:scaling}

\rev{Table~\ref{tab:scaling} reports the time of one feedback evaluation for the function of Sec.~\ref{subsec:logsumexp_example} as $n$ grows from 10 to 2000, listed as the median of five timed batches on one computational thread. HGD and ND use the dense Hessian $\nabla\mathbf{S}=\nabla^2 J$, whose formation, together with that of $\nabla J$, is timed separately; matrix-free HGD computes $\nabla^2 J\,\nabla J$ in closed form from $\mathbf{x}$ and $\nabla J$ in $O(n)$ operations without forming the Hessian, and GD needs only $\nabla J$. The dense realizations store the $n^2$ entries of the Hessian and $O(n)$ vectors, whereas matrix-free HGD and GD store only $O(n)$ vectors. At $n=2000$, forming $\nabla J$ and the Hessian takes $21$~ms and the dense ND takes $76$~ms, whereas matrix-free HGD takes $0.14$~ms. For high-dimensional problems, ND therefore calls for sparse or iterative linear algebra, while HGD needs only the product $\nabla\mathbf{S}^\top\mathbf{S}$, which reverse-mode automatic differentiation computes without forming $\nabla\mathbf{S}$. The implicit solver used in our examples also stores an $N\times N$ Jacobian of the closed-loop vector field; in high dimension, explicit or Jacobian-free implicit integrators avoid this storage.}

\begin{table}[t]
\centering
\caption{\rev{Time ($\mu$s) to form $\nabla J$ and $\nabla^2 J$ and to evaluate each feedback, log-sum-exp example}}
\label{tab:scaling}
\footnotesize
\setlength{\tabcolsep}{4pt}
\rev{\begin{tabular}{@{}rrrrrr@{}}
\hline
$n$ & $\nabla J,\nabla^2 J$ & HGD & HGD, matrix-free & ND & GD \\
\hline
10 & 3.23 & 0.45 & 1.21 & 2.61 & 0.39 \\
20 & 4.95 & 0.48 & 1.70 & 6.09 & 0.32 \\
50 & 7.15 & 1.39 & 3.07 & 18.4 & 0.39 \\
100 & 19.3 & 1.34 & 5.75 & 50.2 & 0.44 \\
200 & 82.0 & 8.90 & 22.4 & 210 & 0.56 \\
500 & 988 & 28.0 & 29.3 & 2010 & 0.41 \\
1000 & 5320 & 140 & 59.1 & 11800 & 0.56 \\
2000 & 21400 & 1160 & 139 & 75800 & 0.76 \\
\hline
\end{tabular}}
\end{table}

\subsection{\rev{Smoothing Error}}
\label{app:eps_sweep}

\rev{Fig.~\ref{fig:eps_sweep} plots $\|\mathbf{z}^\star_\varepsilon-\mathbf{z}^\star\|$ against $\varepsilon$ for the three constrained examples. The smoothed stationarity point $\mathbf{z}^\star_\varepsilon$ is computed by Newton's method on $\mathbf{S}=\mathbf{0}$, with continuation in $\varepsilon$, and the KKT point $\mathbf{z}^\star$ by a semismooth Newton method (NUM) or by enumerating the active sets (minimax and Cournot). For $\varepsilon\ge10^{-5}$, $\lambda^\star_{\varepsilon,j}\,g_j=-\varepsilon^2/2$ holds at $\mathbf{z}^\star_\varepsilon$ for every $j$ to within a relative $3\times10^{-5}$, as Appendix~\ref{app:eps_displacement} shows; for smaller $\varepsilon$, $\varepsilon^2/2$ approaches the rounding level of these products. In NUM and minimax, $\|\mathbf{z}^\star_\varepsilon-\mathbf{z}^\star\|/\varepsilon$ settles at $0.93$ and $0.62$, a linear rate, which Lemma~\ref{lem:smoothed_point}(iv) guarantees for minimax; NUM, whose utilities are not strongly concave, lies outside the lemma but behaves alike. Each of the two has an inactive constraint, whose smoothed multiplier $\widetilde{\lambda}_j\approx\varepsilon/2$ makes the linear rate tight. In the Cournot example, whose only inequality constraint is active with multiplier $2$, the distance decays as $\varepsilon^2$ and vanishes to machine precision for $\varepsilon\le10^{-7}$.}

\begin{figure}[t]
  \centering
  \includegraphics[width=0.9\linewidth]{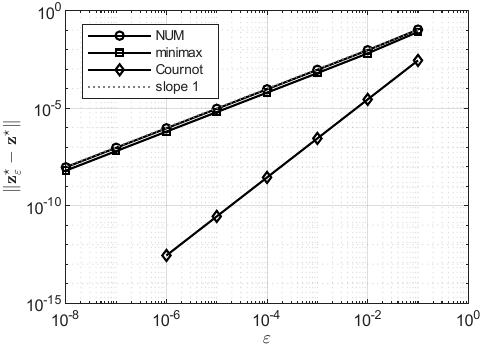}
  \caption{\rev{Distance $\|\mathbf{z}^\star_\varepsilon-\mathbf{z}^\star\|$ between the smoothed stationarity point and the KKT point against $\varepsilon$ (dotted: slope one).}}
  \label{fig:eps_sweep}
\end{figure}

\fi

\iffullversion\ifcameraready\else\newpage\fi\fi
\bibliographystyle{IEEEtran}
\bibliography{Bib1}
\end{document}